\documentclass[11pt]{article}

\usepackage{amsmath,amsthm,amssymb}
\usepackage{graphicx}
\usepackage{subfigure}
\usepackage{authblk}
\usepackage{makecell}
\usepackage{enumitem}
\usepackage{algorithm}
\usepackage{algpseudocode}
\usepackage[top=2.5cm,bottom=2.5cm,left=3cm,right=3cm]{geometry}
\usepackage{hyperref}

\newtheorem{theorem}{Theorem}
\newtheorem{assumption}{Assumption}
\newtheorem{corollary}{Corollary}
\newtheorem{definition}{Definition}

\newtheorem{lemma}{Lemma}[section]
\newtheorem{proposition}{Proposition}[section]
\newtheorem{remark}{Remark}[section]
\newtheorem{example}{Example}[section]

\allowdisplaybreaks[4]

\def\begeqn{\begin{equation}}
\def\endeqn{\end{equation}}
\def\begth{\begin{theorem}}
\def\endth{\end{theorem}}
\def\begprop{\begin{proposition}}
\def\endprop{\end{proposition}}
\def\begcor{\begin{corollary}}
\def\endcor{\end{corollary}}
\def\begdef{\begin{definition}}
\def\enddef{\end{definition}}
\def\beglemm{\begin{lemma}}
\def\endlemm{\end{lemma}}
\def\begexm{\begin{example}}
\def\endexm{\end{example}}
\def\begrem{\begin{remark}}
\def\endrem{\end{remark}}
\def\begassum{\begin{assumption}}
\def\endassum{\end{assumption}}

\def\EE{\mathbb{E}}

\def\O{\mathcal{O}}

\def\R{\mathbb{R}}

\def\Y{\mathcal{Y}}

\def\E{\mathcal{E}}

\def\L{\mathcal{L}}

\def\A{\mathcal{A}}
\def\B{\mathcal{B}}

\def\R{\mathbb{R}}

\def\LL{{\mathcal L}}
\def\RR{{\mathbb R}}

\def\NN{{\mathbb N}}

\def\R{{\mathcal R}}
\def\BB{{\mathcal B}}
\def\FF{{\mathcal F}}
\def\ee{{\mathcal E}}

\title{Adaptive Schauder Stochastic Mirror Descent in Banach Spaces$^\dag$\footnotetext{\dag~The work of Shuai Lu was supported by the National Key Research and
Development Program of China (2023YFA1009103) and the Science and Technology Commission of Shanghai Municipality (23JC1400501). The work of Lei Shi was supported by the National Natural Science Foundation of China (12171093 and 12571099). Email addresses:  24110180001@m.fudan.edu.cn (J. Bai), slu@fudan.edu.cn (S. Lu), leishi@fudan.edu.cn (L. Shi). The corresponding author is Lei Shi.}}

\author[1]{Jinhui Bai}
\author[1,2]{Shuai Lu}
\author[1,2]{Lei Shi}
\affil[1]{School of Mathematical Sciences, Fudan University, Shanghai 200433, China}
\affil[2]{Shanghai Key Laboratory for
	Contemporary Applied Mathematics, Fudan University, Shanghai 200433, China}

\date{}
\begin{document}

\maketitle
\begin{abstract}
In this paper, we solve risk functional minimization problems with general losses in Banach spaces within the framework of stochastic approximation. We introduce an adaptive regularization strategy for stochastic mirror descent (SMD). Specifically, using a Schauder basis, we construct nested finite-dimensional subspaces whose dimensions adapt to the sample size $n$ to balance approximation and optimization errors. We solve each SMD subproblem over the corresponding subspace and project the stochastic subgradient onto its dual space. This yields closed-form solutions and coordinate-wise coefficient updates with low computational and storage complexity. For risk minimization in $\mathcal{L}^p$ spaces with $1<p<\infty$, we construct Bregman distances adapted to the uniform convexity of the underlying Banach space using $\max\{2,p\}$-convex functionals. At the $\mathcal{L}^1$ endpoint, we instead construct a locally strongly convex functional based on the entropy function. By developing a new analytical framework, we show that the risk functional values converge to the minimum at a rate of $\mathcal{O}\left(n^{-\min\{\frac12,\frac1p\}}\right)$, up to logarithmic factors. In the misspecified setting, we prove convergence under weaker regularity assumptions and derive explicit polynomial rates under suitable regularity conditions. Finally, we apply the method to statistical inverse problems and illustrate its empirical performance through numerical experiments in both settings.
\end{abstract}



\section{Introduction}\label{section: introduction}

We consider the minimization of a risk functional in an infinite-dimensional Banach space $\BB$:
\begin{equation}\label{eq:problem}
f_\rho\in\arg\min_{f\in\BB}\ee(f),\quad \ee(f) :=\EE_{\xi\sim\rho}\left[\ell(f,\xi)\right],
\end{equation}
where $\ell$ is the loss function. Since the distribution $\rho$ of $\xi$ is typically unknown, one can usually only evaluate the sample-wise loss $\ell(f,\xi_n)$ and a stochastic subgradient $\partial\ell(f,\xi_n)$ from the $n$-th sample $\xi_n$ of $\xi$. This naturally leads to a stochastic approximation method that constructs an approximation to the minimizer $f_\rho$ after each sample. When $\BB$ is a Hilbert space, a classical stochastic approximation method is stochastic gradient descent (SGD), which recursively generates an iterative sequence $\{f_n\}_{n\geq1}$ using stochastic subgradients:
\begin{equation}\label{eq:SGD direct}
f_n = f_{n-1}-\eta_n\partial\ell(f_{n-1},\xi_n),
\end{equation}
where $\eta_n$ denotes the step size. Since the seminal work of Robbins and Monro on stochastic approximation \cite{robbins1951stochastic}, SGD and its variants have been widely used in inverse problems \cite{jin2020convergence,lu2022stochastic,huang2025early}, PDE-constrained optimization \cite{geiersbach2020stochastic} and geophysical inversion \cite{van2011seismic}. Indeed, SGD can be viewed as a stochastic proximal method based on the squared Euclidean norm $\frac12\Vert\cdot\Vert_\BB^2$, and is therefore well adapted to the geometry of Euclidean spaces:
\begin{equation}\label{eq:SGD equation}
f_n=\arg\min_{f\in\BB}\left\{\frac12\Vert f-f_{n-1}\Vert_\BB^2+\eta_n\left\langle \partial\ell(f_{n-1},\xi_n),f-f_{n-1}\right\rangle_\BB\right\}.
\end{equation}
However, a Banach space $\BB$ is generally not well described by Euclidean geometry and need not admit an inner product compatible with its norm. Nemirovski and Yudin introduced the mirror descent method \cite{NemirovskiYudin1983}. In finite-dimensional stochastic optimization, stochastic mirror descent (SMD) has since attracted considerable attention \cite{nemirovski2009robust,lan2012optimal,nedic2014stochastic,lei2020convergence,hanzely2021fastest,jiang2026mirror}. Its basic idea is to replace the squared Euclidean norm $\frac12\Vert\cdot\Vert_\BB^2$ in \eqref{eq:SGD equation} with the Bregman distance
$ D_\R(f,g):=\R(g)-\R(f)+\langle \partial\R(f),f-g\rangle$ induced by a strictly convex and differentiable function $\R:\BB\to\RR$. By choosing $\R$ appropriately, SMD can better adapt to the geometry of the space $\BB$. Early work \cite{nemirovski2009robust} showed that, for problems over the simplex, SMD with the entropy function and $\ell^1$ geometry can improve the constant factors in the convergence bounds. More recently, \cite{hanzely2021fastest} showed that, even when $\ee$ is not strongly convex under the squared Euclidean geometry, an $\O(1/n)$ convergence rate can still be achieved if $\ee$ is relatively strongly convex and relatively smooth with respect to $\R$.

In this paper, we consider risk functional minimization in a classical family of Banach spaces. Let $\xi=(X,Y)\in\Omega\times\RR\subset\RR^d\times\RR$ follow an unknown distribution $\rho$, and let $\rho_X$ denote the marginal distribution of $X$. For $1\leq p<\infty$, we use $\LL_{\rho_X}^p(\Omega)$ to denote the space of functions that are $p$-th power integrable with respect to $\rho_X$. Let $\{\phi_j\}_{j\geq1}$ be a Schauder basis of $\LL_{\rho_X}^p(\Omega)$. In supervised learning, we aim to estimate statistical targets of interest for the response $Y$ based on the input $X$ (with these targets typically represented by functions in $\LL_{\rho_X}^p(\Omega)$). Accordingly, we write the sample-wise loss as $\ell(f,\xi)=\ell(f(X),Y)$ and consider the following risk functional minimization problem in $\LL_{\rho_X}^p(\Omega)$:
\begin{equation}\label{minimizes the risk functional}
    \min_{f\in\LL_{\rho_X}^p(\Omega)}\EE_\rho\left[\ell(f(X),Y)\right].
\end{equation}
For example, the least-squares risk is minimized by the conditional mean, the pinball risk by conditional quantiles \cite{steinwart2011estimating}, and the expectile (asymmetric least-squares) risk by conditional expectiles \cite{newey1987asymmetric}. Previous studies of general loss functions have primarily focused on Hilbert spaces \cite{rosasco2004are,christmann2007consistency,marteau2019beyond,della2024nystrom,bai2025truncated}, whose geometry is induced by an inner product. In this work, we study stochastic approximation methods for general loss functions in Banach spaces with non-Euclidean geometry for solving \eqref{minimizes the risk functional}. Classical Besov spaces form a scale of function spaces that provides a refined characterization of function smoothness in the Lebesgue space $\LL^p(\Omega)$ \cite{cohen2003numerical}. By extending their wavelet characterization to the Schauder basis $\{\phi_j\}_{j\geq1}$, we define a generalized Besov space $\BB_{p,p}^s\subset\LL_{\rho_X}^p(\Omega)$, with $s>\frac12-\frac1{2p}$. The construction and analysis of a stochastic approximation method in this space face several challenges. First, for $p\neq2$, the primal space $\BB_{p,p}^s$ and its dual space $\left(\BB_{p,p}^s\right)^*$ cannot be canonically identified as in Hilbert spaces. As $f_{n-1}\in\BB_{p,p}^s$ while $\partial\ee(f_{n-1})\in\left(\BB_{p,p}^s\right)^*$, the SGD iteration in \eqref{eq:SGD direct} does not make sense \cite{bittar2022stochastic}. We therefore turn to SMD, whose construction requires a Bregman distance adapted to the geometry of $\BB_{p,p}^s$. More specifically, Hilbert spaces have the maximal modulus of convexity among normed spaces \cite{nordlander1960modulus}, which ensures that the functional $\frac{1}{2}\Vert\cdot\Vert^2$ is strongly convex. For $1<p<\infty$, the space $\BB_{p,p}^s$ is $\max\{p,2\}$-convex, which ensures that the norm-based functional $\R(f)=\frac{1}{\max\{p,2\}}\Vert f\Vert_{\BB_{p,p}^s}^{\max\{p,2\}}$ is $\max\{p,2\}$-convex. We choose $\R$ as the mirror map. For $p>2$, this functional is not strongly convex. Consequently, finite-dimensional SMD analyses that rely on strongly convex mirror maps \cite{nemirovski2009robust,lan2012optimal,nedic2014stochastic} cannot be applied directly in this regime, necessitating a new analytical framework that accommodates uniform convexity in infinite-dimensional Banach spaces. At the endpoint $p=1$, the space $\BB_{1,1}^s$ is no longer uniformly convex, so the preceding norm-based functional is not $r$-convex for any $2\leq r<\infty$. Many studies have shown that the entropy function is well suited to the $\ell_1$ geometry of finite-dimensional simplices \cite{kivinen1997exponentiated,nemirovski2009robust}, but extending this approach to infinite-dimensional Banach spaces is nontrivial. Second, as pointed out in \cite{lu2018relatively}, the SMD subproblems should be efficiently solvable. In finite-dimensional spaces, suitable choices of $\R$ may lead to explicit solutions of the subproblems \cite{nemirovski2009robust}, or the subproblems can be solved efficiently by numerical algorithms. In infinite-dimensional Banach spaces, solving these subproblems with low computational and storage complexity remains a substantial challenge. Finally, when the minimizer $f_\rho$ does not possess the smoothness required by $\BB_{p,p}^s$, i.e. $f_\rho\notin\BB_{p,p}^s$, the Bregman distance $D_\R(f_n,f_\rho)$ may fail to be well defined and finite. We thus need to introduce an appropriate regularization strategy to ensure the convergence of the algorithm in this setting.

We introduce a new regularization strategy for SMD to address the latter two challenges in a unified manner. In ill-posed inverse problems, previous studies have achieved regularization through projection onto finite-dimensional subspaces, with the subspace dimension serving as a regularization parameter to balance approximation error against the error caused by noise amplification due to ill-posedness \cite{natterer1977regularisierung,hamarik2016regularization}. We incorporate this idea into SMD. More precisely, for $1<p<\infty$, we take the finite-dimensional subspace $\BB_{L_n}=\mathrm{span}\{\phi_j\}_{1\leq j\leq L_n}$, where $L_n$ increases with the sample size $n$. After the $n$-th sampling step, we project the stochastic gradient onto the dual space $\BB_{L_n}^*$ and restrict the SMD subproblem to the primal space $\BB_{L_n}$. At the endpoint $p=1$, we instead choose a fixed finite-dimensional subspace $\BB_{L_N}$ according to the total sample size $N$. Since $\BB_{1,1}^s$ is not uniformly convex, we construct a mirror map based on the entropy function for SMD on $\BB_{L_N}$ that is locally strongly convex with respect to the $\BB_{1,1}^s$ geometry. In \autoref{section:main results}, we derive explicit solutions of these subproblems, leading to coordinate-wise updates of the basis coefficients and an implementation with low computational and storage complexity. We incorporate suffix averaging into SMD and refer to the algorithm as Adaptive Schauder Stochastic Mirror Descent (AS-SMD), since the dimension $L_n$ is adaptively adjusted according to the sample size. By developing a new analytical framework adapted to the geometry of $\BB_{p,p}^s$, we show that, for $1<p<\infty$ and $f_\rho\in\BB_{p,p}^s$, the excess risk achieves a convergence rate of $\O\left(n^{-\min\{\frac12,\frac1p\}}\right)$ up to logarithmic factors. In the misspecified setting, we prove that the risk converges to its minimum under weaker regularity assumptions and derive explicit polynomial convergence rates when the regularity exceeds a specified threshold. For $p=1$, we establish an $\O(N^{-1/2})$ convergence rate. In our analysis, the subspace dimension $L_n$ also plays a regularizing role. Let $S_{L_n}f_\rho$ denote the projection of $f_\rho$ onto $\BB_{L_n}$. When $f_\rho\in\BB_{p,p}^s$, $\R(S_{L_n}f_\rho)$ remains uniformly bounded, so the analysis requires only a lower bound on the growth of $L_n$ to ensure that the approximation error does not affect the convergence rate. Additional smoothness of the loss yields sharper approximation error bounds and allows slower growth of the subspace dimension. In the misspecified case, $\R(S_{L_n}f_\rho)$ grows with $L_n$, introducing a growing term in the optimization error. Meanwhile, the growth of $L_n$ should be controlled so that both the optimization and approximation errors tend to zero. In addition, we extend the proposed algorithm to statistical inverse problems, where the aim is to recover an unknown function from randomly sampled point evaluations of its image under a linear operator, corrupted by random noise \cite{blanchard2018optimal,helin2024least}. This extension retains the regularization strategy of AS-SMD, thereby allowing the SMD subproblems to be solved efficiently. Numerical experiments illustrate the empirical performance of the method in nonparametric regression and statistical inverse problems.

This paper is organized as follows. In \autoref{section:Preliminary}, we review Schauder bases in Banach spaces and discuss the generalized Besov spaces together with their geometric properties. In \autoref{section:main results}, we first state the assumptions on the loss function, propose the AS-SMD algorithm, and establish its convergence guarantees. We then extend the algorithm to inverse problems and compare it with the classical SMD method in this setting. In \autoref{Numerical simulations}, we report numerical simulation results \footnote{The AS-SMD code is available at  \url{https://github.com/Researcher-Bai-1/AS-SMD-algorithm.git}.}. To improve readability, we present the proofs of the main convergence results in \autoref{sec:proof} and those of several technical lemmas in \autoref{sec:appendix}.

\section{Preliminaries}\label{section:Preliminary}

In this section, we briefly review Schauder bases in Banach spaces and then characterize the geometric properties of the generalized Besov spaces induced by the Schauder basis. The Banach space $\mathcal{L}_{\rho_X}^p(\Omega)$ is equipped with the norm $\Vert f\Vert_{\LL^p}=\left(\int_\Omega |f(x)|^pd\rho_X(x)\right)^{1/p}$ for $f\in \mathcal{L}_{\rho_X}^p(\Omega)$. A sequence $\{\phi_j\}_{j\geq 1} \subset \mathcal{L}_{\rho_X}^p(\Omega)$ is said to form a Schauder basis for $\mathcal{L}_{\rho_X}^p(\Omega)$ if, for each $f \in \mathcal{L}_{\rho_X}^p(\Omega)$, there exists a unique sequence of coefficients $\{\beta_j\}_{j\ge 1} \subset \mathbb{R}$ such that
$$\lim_{n\to\infty}\left\Vert f-\sum_{j=1}^n\beta_j\phi_j\right\Vert_{\LL^p}=0.$$
The partial sum projections $\{S_n\}_{n\geq1}$ associated with the Schauder basis $\{\phi_j\}_{j \ge 1}$ are defined by $S_n(f)=\sum_{j=1}^n \beta_j \phi_j.$ The operators $\{S_n\}_{n\geq1}$ are linear and uniformly bounded, that is,
$$\sup_{n\ge 1}\Vert S_n\Vert := \sup_{n\ge 1} \sup_{f\in\mathcal{L}_{\rho_X}^p(\Omega),\ \Vert f\Vert_{\LL^p}\leq 1} \Vert S_n(f)\Vert_{\LL^p}<\infty,$$  
see \cite[Proposition~1.1.4]{albiac2006topics}. We define the linear functionals of the basis $\{\phi_j\}_{j\ge 1}$ by $\phi_j^* : \mathcal{L}_{\rho_X}^p(\Omega) \to \mathbb{R},\ \phi_j^*(f) = \beta_j$. Each $\phi_j^*$ is a bounded linear functional, and $\{\phi_j^*\}$ satisfy the biorthogonality property $\phi_j^*(\phi_i)=\delta_{ij},\ \forall i,j\in\NN_+$, where $\delta_{ij}$ denotes the Kronecker symbol (see \cite[Definition 1.1.2 and Theorem~1.1.3]{albiac2006topics}). The functionals $\{\phi_j^*\}$ are referred to as the biorthogonal functionals of the basis $\{\phi_j\}$. We denote the duality pairing between $f\in \mathcal{L}_{\rho_X}^p(\Omega)$ and an element in the dual space $f^*\in \left(\mathcal{L}_{\rho_X}^p(\Omega)\right)^*$ by $\langle f^*, f\rangle = f^*(f)$. For $f \in \mathcal{L}_{\rho_X}^p(\Omega)$,
$$ f=\sum_{j=1}^\infty \langle\phi_j^*, f\rangle\phi_j.$$
Schauder bases are widely available in $\mathcal{L}_{\rho_X}^p(\Omega)$. When $\rho_X$ is the Lebesgue measure on $[0,1]$, both the Haar basis and the Franklin system form Schauder bases of $\mathcal{L}_{\rho_X}^p([0,1])$ for $1 \leq p < \infty$ \cite{bovckarev1976existence,zink1988franklin,albiac2006topics}. For general domains $\Omega$, many wavelet systems form Schauder bases of $\mathcal{L}_{\rho_X}^p(\Omega)$ \cite{cohen2000multiscale}.

Previous studies in Hilbert spaces typically assume that $f_\rho$ belongs to smooth benchmark classes, such as a space of continuous functions or a Sobolev space. In this paper, we extend the analysis to generalized Besov spaces. Classical Besov spaces $\BB_{p,q}^s$ can be intuitively viewed as spaces of functions possessing derivatives of order $s > 0$ that belong to $\mathcal{L}_{\rho_X}^p(\Omega)$, where the parameter $q$ provides a finer control over the regularity of the functions. For a more in-depth discussion of Besov spaces, we refer the reader to \cite{cohen2003numerical,triebel2008function}. Here, we adopt the wavelet-based characterization of Besov spaces \cite[Theorem~3.10.5]{cohen2003numerical}. Given a suitable $\L^2$-normalized wavelet basis $\{\Phi_{0,j}\}_{j \in \Lambda_0}\cup\{\Psi_{k,j}\}_{j \in \Lambda_k,k> 0}$, where $k$ denotes the scale index and $j$ the translation index, the Besov space $\BB_{p,q}^s$ consists of all functions $f$ satisfying
\begin{equation}\label{Besov space 1}
\begin{aligned}
&f = \sum_{j \in \Lambda_0}\beta_{0,j}\Phi_{0,j}+\sum_{k> 0}\sum_{j\in\Lambda_k}\beta_{k,j}\Psi_{k,j},\\
&\Vert f\Vert_{\BB_{p,q}^s}:=\left(\sum_{k\geq 0} 2^{kq\left(s+d(1/2-1/p)\right)}\left(\sum_{j\in\Lambda_k} |\beta_{k,j}|^p\right)^{q/p}\right)^{1/q}<\infty,
\end{aligned}
\end{equation}
where $1 \le p,q < \infty$ and $s > 0$. For standard wavelet constructions on a bounded domain $\Omega$, there exist constants $0<c_2\leq c_1$ such that $c_2 2^{kd}\leq|\Lambda_k|\leq c_1 2^{kd}$. If we reorder the wavelet basis $\{\Phi_{0,j}\}_{j \in \Lambda_0}\cup\{\Psi_{k,j}\}_{j \in \Lambda_k,k> 0}$ into a single sequence $\{\Psi_j\}_{j \ge 1}$ according to the lexicographic order, \autoref{lemma:A.1} provides a more concise equivalent characterization of the Besov space $\BB_{p,q}^s$ in the case $p = q$,
\begin{equation}\label{Besov space 2}
\BB_{p,p}^s= \left\{ f=\sum_{j=1}^\infty \beta_j\Psi_j\,\Big\vert\, \Vert f\Vert_{\BB_{p,p}^s}=\left(\sum_{j=1}^\infty j^{p\left(\frac{s}{d}+(1/2-1/p)\right)}|\beta_j|^p\right)^{1/p}<\infty \right\}.
\end{equation}
In this paper, we focus on Schauder bases and extend the wavelet-based characterization of Besov spaces to Schauder bases, thereby defining generalized Besov spaces. Given a Schauder basis $\{\phi_j\}$ of $\mathcal{L}_{\rho_X}^p(\Omega)$ and $s>0,1\leq p<\infty$, we define the generalized Besov space $\BB_{p,p}^s$ as
\begin{equation}\label{generalized Besov space}
 \BB_{p,p}^s=\left\{ f=\sum_{j=1}^\infty \beta_j\phi_j\,\Big\vert\, \Vert f\Vert_{\BB_{p,p}^s}:=\left(\sum_{j=1}^\infty j^{2sp}|\beta_j|^p\right)^{1/p}<\infty \right\}.
\end{equation}
Hereafter, $\BB_{p,p}^s$ refers to the generalized Besov space defined in \eqref{generalized Besov space}. In \autoref{lemma:A.2}, we show that the embedding $\BB_{p,p}^s \subset \mathcal{L}_{\rho_X}^p(\Omega)$ holds for $s > \frac{1}{2} - \frac{1}{2p}$. Moreover, for $1<p<\infty$, $\BB_{p,p}^s$ is reflexive, with dual space characterized in \eqref{eq:dual space 2}.

Next, we turn to the geometric properties of $\BB_{p,p}^s$.
\begin{definition}
Let $\BB$ be a Banach space. $\BB$ is said to be smooth if for every $0 \neq f \in \BB$ there exists a unique $f^* \in \BB^*$ such that $\langle f^*, f\rangle = \Vert f\Vert_{\BB}$ and $ \Vert f^*\Vert_{\BB^*}=1$. The modulus of convexity of $\BB$ is defined as the function $\delta_{\BB} : (0,2] \to [0,1]$, given by
$$\delta_\BB(\tau):= \inf \left\{1 - \left\| \frac{f+g}{2} \right\|_{\BB}\,\Big\vert\, \|f\|_{\BB} = \|g\|_{\BB} = 1,\ \|f-g\|_{\BB} \ge \tau, f,g\in\BB\right\}.$$
The space $\BB$ is uniformly convex if $\delta_\BB(\tau)>0$, and is $p$-convex for $p\geq2$ if $\delta_\BB(\tau)\geq C_p\tau^p$ for some constant $C_p>0$.
\end{definition}
The following lemma can be found in \cite{cioranescu1990geometry,xu1991characteristic,schuster2012regularization}.
\begin{lemma}\label{lemma:convexity properties}
If the Banach space $\BB$ is $p$-convex and smooth, the following properties hold.
\begin{enumerate}[label=(\roman*)]
  \item The functional $\R_\BB(f) := \frac{1}{p}\Vert f\Vert_{\BB}^p$ is Gâteaux differentiable, with derivative given by
$$\partial \R_\BB(f)=\left\{ f^*\in\BB^*\,\Big\vert\, \langle f^*,f\rangle=\Vert f\Vert_\BB\cdot\Vert f^*\Vert_{\BB^*},\ \text{and}\ \Vert f\Vert_\BB^{p-1}=\Vert f^*\Vert_{\BB^*}\right\}.$$
  \item The functional $\R_{\BB}(f)$ is $p$-convex, i.e., there exists a constant $C_{p} > 0$ such that
\begin{equation}\label{eq:p convex property}
 \R_\BB (g)\geq \R_\BB(f) - \langle \partial \R_\BB(f),f-g\rangle +\frac{C_{p}}{2}\Vert f-g\Vert_\BB^p,\quad \forall f,g\in\BB.
\end{equation}
\end{enumerate}
\end{lemma}
Some commonly used Banach spaces, such as $\mathcal{L}_{\rho_X}^p(\Omega)$, Sobolev spaces $W^{s,p}$, and sequence spaces $\ell^p$, are smooth and $\max\{p,2\}$-convex for $1<p<\infty$ (see \cite{schuster2012regularization}). However, $\ell^1$ and $\mathcal{L}_{\rho_X}^1(\Omega)$ are generally not uniformly convex. In \autoref{lemma:A.3}, we show that the generalized Besov spaces $\BB_{p,p}^s$ are smooth and $\max\{p,2\}$-convex for $1<p<\infty$, whereas $\BB_{1,1}^s$ is not uniformly convex. We first consider the case $1<p<\infty$. For $f=\sum_{j=1}^\infty \beta_j\phi_j$, the subgradient of the functional $\R_{s,p}(f):=\frac{1}{\max\{p,2\}}\Vert f\Vert_{\BB^s_{p,p}}^{\max\{p,2\}}$ is given by
$$ \partial \R_{s,p}(f)=\Vert f\Vert_{\BB^s_{p,p}}^{\max\{0,2-p\}}\sum_{j=1}^\infty |\beta_j|^{p-1}\text{sign}(\beta_j) j^{2sp}\phi_j^*.$$
By \autoref{lemma:convexity properties}, the functional $\R_{s,p}$ is also $\max\{p,2\}$-convex and satisfies \eqref{eq:p convex property}. This is analogous to strong convexity in Hilbert spaces. In particular, when $1<p\leq2$, $\R_{s,p}$ is strongly convex. We naturally choose $\R_{s,p}$ as the mirror map and define the induced Bregman distance by
\begin{equation}\label{eq:Bregman distance}
D_{s,p}(f,g):= \R_{s,p} (g) - \R_{s,p}(f) + \langle \partial \R_{s,p}(f),f-g\rangle,\quad \forall f,g\in\BB_{p,p}^s.
\end{equation}
The Bregman distance satisfies $D_{s,p}(f,g)\geq \frac{C_{p,s}}{2}\Vert f-g\Vert_{\BB_{p,p}^s}^{\max\{p,2\}}$ for some constant $C_{p,s}>0$. When $p=1$, since $\BB_{1,1}^s$ is not uniformly convex, constructing a mirror map $\R_{s,1}$ with suitable convexity properties via the norm functional becomes difficult. In \autoref{sub:Stochastic Mirror Descent}, we address this issue by introducing a locally strongly convex entropy function on a subspace of $\BB_{1,1}^s$.

\section{Main Results}\label{section:main results}

In \autoref{sub:Stochastic Mirror Descent}, we introduce the Adaptive Schauder Stochastic Mirror Descent algorithm. We establish convergence guarantees for the algorithm in \autoref{subsec:Theoretical Analysis}. In \autoref{subsec:Link with Inverse Problems in Banach Spaces}, we analyze the connection between the problem studied in this paper and statistical inverse problems and extend AS-SMD to this setting. We first introduce the assumptions on the loss function.
\begin{assumption}\label{assume:1}
Let $\Y\subset\RR$. Assume that the loss function $\ell:\RR\times\Y\to\RR$ is continuous on $\RR\times\Y$ and convex in its first argument $u$ for each fixed $y\in\Y$. Assume further that there exists a Borel measurable function $g_\ell:\RR\times\Y\to\RR$ such that, for every $(u,y)\in\RR\times\Y$, $g_\ell(u,y)$ is a subgradient of $\ell(\cdot,y)$ at $u$. Suppose that there exist $1\leq q<\infty$ and a constant $C_\ell>0$ such that
$$ 0\leq \ell(u,y)\leq C_\ell(1+|u|+|y|)^q, \qquad \forall\,u\in\RR,\ y\in\Y.$$
\end{assumption}
\autoref{assume:1} imposes convexity, continuity, and a polynomial growth condition on the loss function, and requires the existence of a Borel measurable subgradient. When $1\leq q\leq p$ and $\mathbb{E}_\rho[|Y|^p]<\infty$, this growth condition $\ell(u,y)\leq C_\ell(1+|u|+|y|)^q$ ensures that the risk functional $\mathcal{E}(f)=\mathbb{E}_\rho[\ell(f(X),Y)]$ is finite for every $f\in\LL_{\rho_X}^p(\Omega)$. This assumption allows for nondifferentiable losses and covers many common regression losses, as well as various binary classification losses when $\mathcal{Y}=\{-1,+1\}$. By \autoref{lemma:add A.1}, it also implies local Lipschitz continuity of the loss function:
$$|\ell(u_1,y)-\ell(u_2,y)|\leq 2^qC_\ell|u_1-u_2|(1+|u_1|+|u_2|+|y|)^{q-1}.$$
In the subsequent convergence analysis, we use this local Lipschitz condition to control the approximation error. By imposing an appropriate smoothness condition on the loss function, we can obtain sharper approximation error bounds. To this end, we introduce the following additional assumption.
\begin{assumption}\label{assume:2}
For $1<q<\infty$ specified above, assume that $\ell(u,y)$ is differentiable with respect to $u$ for each $y\in\Y$. Suppose that there exist $1<\nu\leq\min\{q,2\}$ and a constant $C_\nu>0$ such that
$$\left|\partial_u\ell(u,y)-\partial_u\ell(v,y)\right| \leq C_\nu |u-v|^{\nu-1}(1+|u|+|v|+|y|)^{q-\nu}, \qquad \forall\,u,v\in\RR,\ y\in\Y.$$
\end{assumption}
Many commonly used loss functions satisfy both \autoref{assume:1} and \autoref{assume:2}. We list below several examples widely used in regression and classification, together with the corresponding exponents $q$ and $\nu$. For losses satisfying only \autoref{assume:1}, we set $\nu=1$. Throughout, $u\in\RR$ and $t_+:=\max\{t,0\}$.
\begin{itemize}
    \item Power loss: $\ell(u,y)=|u-y|^q$, where $\Y=\RR$ and $q\geq1$. We take $\nu=\min\{q,2\}$.
    \item Huber loss: $\ell(u,y)=\sqrt{1+(y-u)^2}-1$, where $\Y=\RR$. One may choose any $q>1$ and take $\nu=\min\{q,2\}$.
    \item Logistic loss: $\ell(u,y)=\log(1+e^{-yu})$, where $\Y=\{-1,+1\}$. One may choose any $q>1$ and take $\nu=\min\{q,2\}$.
    \item Power hinge loss: $\ell(u,y)=(1-yu)_+^q$, where $\Y=\{-1,+1\}$ and $q\geq1$. We take $\nu=\min\{q,2\}$.
    \item Expectile loss: $\ell_b(u,y)=b(y-u)_+^2+(1-b)(u-y)_+^2$, where $\Y=\RR$ and $b\in(0,1)$. We take $q=\nu=2$.
    \item Pinball loss: $\ell_b(u,y)=b(y-u)_++(1-b)(u-y)_+$, where $\Y=\RR$ and $b\in(0,1)$. We take $q=1$ and set $\nu=1$.
\end{itemize}
These losses arise in a variety of supervised learning problems. Power losses, particularly the least-squares loss, are widely used in nonparametric regression, while the Huber loss is used in robust regression \cite{huber1964robust}. The logistic loss is a standard choice for binary classification, and the power hinge loss is used in $q$-norm soft margin support vector machines \cite{chen2004support}. The pinball loss is used to estimate conditional quantiles \cite{steinwart2011estimating}, whereas the expectile loss is used to estimate conditional expectiles through asymmetric least squares \cite{newey1987asymmetric}.

\subsection{Adaptive Schauder Stochastic Mirror Descent}\label{sub:Stochastic Mirror Descent}

In this subsection, we discuss the construction of Adaptive Schauder Stochastic Mirror Descent separately for the cases $1<p<\infty$ and $p=1$.

\subsubsection{The Case  \texorpdfstring{$1<p<\infty$}{p>1}}

We first analyze the subgradient of the risk functional $\mathcal{E}(f)=\EE_\rho\left[\ell\left(f(X),Y\right)\right]$ in the generalized Besov space $\BB_{p,p}^s$ for $s>\frac{1}{2}-\frac{1}{2p}$ and $1\leq q\leq p$. Here we set $0^0=1$ and $\text{sign}(0)=0$. Assume that $\EE[|Y|^p]<\infty$, and let $g_\ell$ be a measurable function such that $g_\ell(u,y)$ is a subgradient of $\ell(\cdot,y)$ at $u$ for every $(u,y)\in\RR\times\Y$. Then, for any $f,h\in\BB_{p,p}^s$, the convexity of the loss function yields
$$\EE_\rho\left[\ell(f(X)+h(X),Y)\right]-\EE_\rho\left[\ell(f(X),Y)\right]\geq \EE_\rho\left[g_\ell(f(X),Y)h(X)\right].$$
\autoref{lemma:A.4} shows that the functional $\EE_\rho\left[g_\ell(f(X),Y)h(X)\right]$ is bounded with respect to $h$, for any fixed function $f\in\mathcal{L}_{\rho_X}^p(\Omega)$, and hence is a subgradient of the risk functional at $f$. An unbiased estimator of this subgradient is denoted by 
\begin{align*}
\widehat{\partial\mathcal{E}}(f)(h)=&g_\ell(f(X),Y)h(X)=g_\ell(f(X),Y)\sum_{j=1}^\infty \langle \phi_j^*,h\rangle \phi_j(X).
\end{align*}
Although we obtain an explicit expression for the stochastic subgradient, its series form makes it difficult to apply directly in algorithm implementation. Following the regularization strategy outlined in \autoref{section: introduction}, we consider the $L_n$-dimensional subspace $\BB_{L_n}=\text{span}\{\phi_1,\dots,\phi_{L_n}\}$ of $\BB_{p,p}^s$ and its dual space $\BB_{L_n}^*=\text{span}\{\phi_1^*,\dots,\phi_{L_n}^*\}$, where $L_n\in\mathbb{N}$ is chosen adaptively according to the sample size $n$. We project the stochastic subgradient onto $\BB_{L_n}^*$ as follows:
$$\widehat{\partial\mathcal{E}}(f)\big|_{\BB_{L_n}^*}= g_\ell(f(X),Y)\sum_{j=1}^{L_n}\phi_j(X)\phi_j^*.$$
If $\widehat{\partial\mathcal{E}}(f)$ is bounded, \autoref{lemma:A.4}  shows that  $\widehat{\partial\mathcal{E}}(f)\big|_{\BB_{L_n}^*}$ is the projection of the stochastic subgradient onto $\BB_{L_n}^*$.

With the above preparations, we are now ready to present the Adaptive Schauder Stochastic Mirror Descent algorithm. We initialize $f_0=0$. Upon receiving the $n$-th sample $(X_n,Y_n)$, with $\widehat{\partial\mathcal{E}}(f_{n-1})\big|_{\BB_{L_n}^*}=g_\ell(f_{n-1}(X_n),Y_n)\sum_{j=1}^{L_n}\phi_j(X_n)\phi_j^*$, the stochastic mirror descent update rule is given by
\begin{equation}\label{eq:stochastic mirror descent}
\begin{aligned}
f_{n} =& \mathop{\arg\min}_{f\in\BB_{L_n}}\left\{ D_{s,p}(f_{n-1},f)+\eta_n\left\langle \widehat{\partial\mathcal{E}}(f_{n-1})\big|_{\BB_{L_n}^*}, f-f_{n-1}\right\rangle\right\}\\
=& \mathop{\arg\min}_{f\in\BB_{L_n}}\left\{ \R_{s,p}(f)-\left\langle  \partial \R_{s,p}(f_{n-1})-\eta_n\widehat{\partial\mathcal{E}}(f_{n-1})\big|_{\BB_{L_n}^*}, f\right\rangle\right\}
\end{aligned}
\end{equation}
where $\eta_n$ denotes the step size, and $L_n$ is typically chosen as $L_n=\lceil n^\theta\rceil$ for some $\theta>0$, where $\lceil\cdot\rceil$ denotes the ceiling function. It follows from \eqref{eq:stochastic mirror descent} that the SMD update consists of two steps. The gradient is first updated in the dual space according to $\partial \R_{s,p}(f_{n-1})-\eta_n\widehat{\partial\mathcal{E}}(f_{n-1})\big|_{\BB_{L_n}^*}$. By solving the subproblem, the updated element in the dual space is pulled back to the primal space $\BB_{L_n}$, yielding $f_n\in\BB_{L_n}$. Under this projection, the subproblem in the algorithm \eqref{eq:stochastic mirror descent} admits a closed-form solution.
If we denote $f_{n-1}=\sum_{j=1}^{L_{n}}\beta_j^{(n-1)}\phi_j,\ f=\sum_{j=1}^{L_{n}}\beta_j\phi_j\in\BB_{L_{n}}$  and note that $w(f_{n-1}) = g_\ell(f_{n-1}(X_n),Y_n)$, then \eqref{eq:stochastic mirror descent} can be equivalently rewritten as
\begin{equation*}
\begin{aligned}
f_{n} =& \mathop{\arg\min}_{f=\sum_{j=1}^{L_{n}}\beta_j\phi_j\in\BB_{L_{n}}}\Bigg\{ \frac{1}{\max\{p,2\}}\left(\sum_{j=1}^{L_n}j^{2sp}|\beta_j|^p\right)^{\max\{1,2/p\}} \\
&-\sum_{j=1}^{L_n}\left(\Vert f_{n-1}\Vert_{\BB^s_{p,p}}^{\max\{0,2-p\}} |\beta_j^{(n-1)}|^{p-1}\text{sign}(\beta_j^{(n-1)}) j^{2sp}-\eta_nw(f_{n-1})\phi_j(X_n)\right)\beta_j \Bigg\}\\
=& \mathop{\arg\min}_{(\beta_1,\dots,\beta_{L_n})\in\RR^{L_n}}\Bigg\{ \frac{1}{\max\{p,2\}}\left(\sum_{j=1}^{L_n}j^{2sp}|\beta_j|^p\right)^{\max\{1,2/p\}}-\sum_{j=1}^{L_n}\alpha_j\beta_j \Bigg\},
\end{aligned}
\end{equation*}
where $\alpha_j := \Vert f_{n-1}\Vert_{\BB^s_{p,p}}^{\max\{0,2-p\}} |\beta_j^{(n-1)}|^{p-1}\text{sign}(\beta_j^{(n-1)}) j^{2sp}-\eta_ng_\ell(f_{n-1}(X_n),Y_n)\phi_j(X_n).$ The above subproblem can be reformulated as a convex optimization problem
in $\RR^{L_n}$. If $\alpha_j=0$ for all $1\leq j\leq L_n$, then the solution is $\beta_j=0$ for all $1\leq j\leq L_n$. Otherwise, the solution is given by
$$ \beta_j = A\cdot j^{-2sp'}|\alpha_j|^{\frac{1}{p-1}}\text{sign}(\alpha_j),\quad \forall j=1,\dots,L_n,\quad A = \left(\sum_{j=1}^{L_n} j^{-2sp'}|\alpha_j|^{p'}\right)^{\frac{\min\{0,p-2\}}{p}},$$
where $p'=\frac{p}{p-1}$ is the H\"{o}lder conjugate of $p$. For fixed $0<\alpha<1$, we output the suffix average
$ \bar{f}^s_{\alpha n}=\frac{1}{\alpha n}\sum_{i=(1-\alpha)n}^{n-1} f_i.$
The algorithm is presented in Algorithm \ref{alg:theoretical}.

\begin{algorithm}
\small
\begin{algorithmic}
\State{\textbf{set}:  $s>\frac{1}{2}-\frac{1}{2p},\ 1\leq q\leq p$,\ $\theta>0$}
\State{\textbf{initialize}: $f_0 = 0,\ L_0 = 1$}
\For{$n=1,2,3,\dots$}
\State{Collect sample $(X_n,Y_n)$, calculate $\eta_n$ and $L_n=\lceil n^\theta\rceil$}
\State{With \ $f_{n-1}=\sum_{j=1}^{L_{n}}\beta_j^{(n-1)}\phi_j$ and $w(f_{n-1})= g_\ell(f_{n-1}(X_n),Y_n)$, calculate $\{\alpha_j\}_{1\leq j\leq L_n}$:\ $$\alpha_j \leftarrow \left(\sum_{l=1}^{L_n}l^{2sp}|\beta_l^{(n-1)}|^p \right)^{\frac{\max\{0,2-p\}}{p}} |\beta_j^{(n-1)}|^{p-1}\text{sign}(\beta_j^{(n-1)}) j^{2sp}-\eta_nw(f_{n-1})\phi_j(X_n)$$ }
\State{Update $f_n\leftarrow\sum_{j=1}^{L_n}\beta_j^{(n)}\phi_j$:}
\If{$\alpha_j=0$ for all $j=1,\dots,L_n$}
\State{$\beta_j^{(n)}\leftarrow 0,\quad \forall j=1,\dots,L_n$}
\Else
\State{$$ \beta_j^{(n)} \leftarrow A\cdot j^{-2sp'}|\alpha_j|^{\frac{1}{p-1}}\text{sign}(\alpha_j),\quad \forall j=1,\dots,L_n,\quad A = \left(\sum_{j=1}^{L_n} j^{-2sp'}|\alpha_j|^{p'}\right)^{\frac{\min\{0,p-2\}}{p}}$$
}
\EndIf
\EndFor{}
\State{\textbf{return}\ $f_n,\bar{f}^s_{\alpha (n+1)}=\frac{1}{\alpha  (n+1)}\sum_{i=(1-\alpha)(n+1)}^{n} f_i$}
\end{algorithmic}
\caption{Adaptive Schauder Stochastic Mirror Descent}
\label{alg:theoretical}
\end{algorithm}

We now analyze the computational and storage complexity of the algorithm. In the implementation, we store the coefficients in the basis expansion $f_n=\sum_{j=1}^{L_n}\beta_j^{(n)}\phi_j$ and update them coordinate-wise, rather than manipulating the function $f_n$ directly. We assume that evaluating a single basis function $\phi_j$ at a given point $X_n$ costs $\mathcal{O}(1)$ time. At each iteration, we first need to compute $L_n$ coefficients $\{\alpha_j\}_{1\le j\le L_n}$. The core computational cost arises from evaluating the norm $\Vert f_{n-1}\Vert_{\BB_{p,p}^s}=\left(\sum_{j=1}^{L_n} j^{2sp}|\beta_j^{(n-1)}|^p\right)^{1/p}$ and the function value $f_{n-1}(X_n)=\sum_{j=1}^{L_n}\beta_j^{(n-1)}\phi_j(X_n)$. Both operations require $\mathcal{O}(L_n)$ time. Consequently, computing the coefficients $\{\alpha_j\}_{1\le j\le L_n}$ requires $\mathcal{O}(L_n)$ time per iteration. Similarly, updating the coefficients $\{\beta_j^{(n)}\}_{1\le j\le L_n}$ also requires $\mathcal{O}(L_n)$ time. If the total number of samples $n$ is known in advance, the averaging procedure can start from the $(1-\alpha)n$-th iteration. We recursively compute $\bar{f}_i^s=\frac{i-(1-\alpha)n}{i+1-(1-\alpha)n}\bar{f}_{i-1}^s+\frac{1}{i+1-(1-\alpha)n}f_i$ with $i\geq (1-\alpha)n$ and $\bar{f}_{(1-\alpha)n-1}^s=0$. After performing $\alpha n$ iterations, we obtain the averaged estimator $\bar{f}^s_{\alpha n}$, with a total computational cost $\mathcal{O}(nL_n)$. Thus, processing $n$ samples requires $\mathcal{O}(nL_n)$ time, i.e., $\mathcal{O}(n^{1+\theta})$. Regarding storage complexity, at each iteration the algorithm only needs to store the coefficients of $f_n,\bar{f}^s_{\alpha (n+1)}$. The overall storage complexity is $\mathcal{O}(L_n)$, that is, $\mathcal{O}(n^{\theta})$. We discuss the choice of $\theta$ further in \autoref{subsec:Theoretical Analysis}.

\subsubsection{The Case  \texorpdfstring{$p=1$}{p=1}}\label{sub:Stochastic Mirror Descent for p equal 1}

Unlike the case $1<p<\infty$, the $p=1$ formulation requires the total sample size $N$ to be known in advance. Here we no longer update the finite-dimensional subspace $\BB_{L_n}$ after the $n$-th sample. Instead, we determine the dimension $L_N$ and the subspace $\BB_{L_N}$ directly from the total sample size $N$. As before, let $L_N=\lceil N^\theta\rceil$ for some $\theta>0$, and let $\BB_{L_N}=\text{span}\{\phi_1,\dots,\phi_{L_N}\}$. The entropy function is a standard mirror map on the simplex and is strongly convex with respect to the $\ell^1$ norm \cite{kivinen1997exponentiated,nemirovski2009robust}. Its integral form is also commonly used for problems in $\mathcal{L}_{\rho_X}^1(\Omega)$ \cite{huang2025early}. The entropy function is defined only for nonnegative components, and its strong convexity with respect to the $\ell^1$ norm generally requires a bounded domain. Accordingly, we reparametrize elements of $\BB_{L_N}$ by decomposing their coefficients into positive and negative parts, and embed the resulting representations into a subset of $\RR_+^{2L_N}$. On this subset, we construct a Bregman distance induced by the entropy function. For each new sample, the stochastic subgradient is projected onto the dual space $\BB_{L_N}^*$, and the stochastic mirror descent update is then performed in this subset of $\RR_+^{2L_N}$. We first introduce the closed ball of $\BB_{1,1}^s$ with radius $L>0$, $\mathcal{C}_L=\{f\in\BB^s_{1,1}\, |\, \Vert f\Vert_{\BB^s_{1,1}}\leq L\}$, where $L$ can be chosen  sufficiently large. We then define the set $\mathcal{A}_L^{2L_N}\subset \RR^{2L_N}_+$ by
$$\mathcal{A}_L^{2L_N}=\left\{\left(\alpha_1^+,\dots,\alpha_{L_N}^+,\alpha_1^-,\dots,\alpha_{L_N}^-\right)\in\RR^{2L_N}_+\,\Bigg|\,\sum_{j=1}^{L_N}\left(\alpha_j^++\alpha_j^-\right)j^{2s}\leq L\right\}.$$
To decompose the coefficients into their positive and negative components, we introduce the mapping $\mathcal{T}$,
\begin{align*}
\mathcal{T}:  \BB_{L_N}\cap\mathcal{C}_L&\to \mathcal{A}_L^{2L_N}\quad \sum_{j=1}^{L_N}\alpha_j\phi_j\to \left(\alpha_1^+,\dots,\alpha_{L_N}^+,\alpha_1^-,\dots,\alpha_{L_N}^-\right),
\end{align*}
where $\alpha_j^+=\max\{\alpha_j,0\}$ and $\alpha_j^-=\max\{-\alpha_j,0\}$. Define the synthesis map $\mathcal{T}^\sharp$ from the positive and negative components,
\begin{align*}
\mathcal{T}^\sharp: \mathcal{A}_L^{2L_N} \to \BB_{L_N}\cap\mathcal{C}_L\quad \left(\alpha_1^+,\dots,\alpha_{L_N}^+,\alpha_1^-,\dots,\alpha_{L_N}^-\right)\to\sum_{j=1}^{L_N}(\alpha_j^+-\alpha_j^- )\phi_j.
\end{align*}
 This method is inspired by the ``EG$\pm$ trick'' proposed in \cite{kivinen1997exponentiated}, and allows us to use the entropy function to handle the negative coefficients in elements of $\BB_{L_N}\cap\mathcal{C}_L$. The map $T^\sharp$ naturally synthesizes the positive and negative components into an element of $\BB_{L_N}$. Since $\mathcal{T}^\sharp$ is generally not injective, the representation of an element of $\BB_{L_N}\cap\mathcal{C}_L$ by an element of $\mathcal{A}_L^{2L_N}$ is not unique. The map $\mathcal{T}$ defined above provides one specific choice. We define the mirror map on $\mathcal{A}_L^{2L_N}$ by
$$\R^{L_N}_{s,1}(\alpha^{\pm})=\sum_{j=1}^{L_N}\left(j^{2s}\alpha_j^+\right)\log\left(j^{2s}\alpha_j^+\right)
+\sum_{j=1}^{L_N}\left(j^{2s}\alpha_j^-\right)\log\left(j^{2s}\alpha_j^-\right),$$
where $\alpha^\pm = \left(\alpha_1^+,\dots,\alpha_{L_N}^+,\alpha_1^-,\dots,\alpha_{L_N}^-\right)\in \mathcal{A}_L^{2L_N}$ and $0\log(0) := 0$. The function $\R^{L_N}_{s,1}$ is continuously differentiable on 
$\mathcal{A}_L^{2L_N}\cap\RR^{2L_N}_{++}$, and its partial derivatives are given by $\frac{\partial\R^{L_N}_{s,1}}{\partial \alpha_j^+}=j^{2s}\left(1+\log\left(j^{2s}\alpha_j^+\right)\right),\ \frac{\partial\R^{L_N}_{s,1}}{\partial \alpha_j^-}=j^{2s}\left(1+\log\left(j^{2s}\alpha_j^-\right)\right)$ for $1\leq j\leq L_N$. We define the Bregman distance
$D_{s,1}^{L_N}: \left(\mathcal{A}_L^{2L_N}\cap\RR^{2L_N}_{++}\right)\times \mathcal{A}_L^{2L_N}\to\RR$ by
\begin{align*}
&D_{s,1}^{L_N}(\alpha^\pm,\beta^\pm) = \R^{L_N}_{s,1} (\beta^\pm) - \R_{s,1}^{L_N}(\alpha^\pm) + \langle \partial \R_{s,1}^{L_N} (\alpha^\pm),\alpha^\pm-\beta^\pm\rangle_{\ell_2} \\
 =&\sum_{j=1}^{L_N}j^{2s}\left(\beta_j^+\log\left(\frac{\beta_j^+}{\alpha_j^+}\right)+\beta_j^-\log\left(\frac{\beta_j^-}{\alpha_j^-}\right)\right)
+\sum_{j=1}^{L_N}j^{2s}\left(\left(\alpha_j^+-\beta_j^+\right)+\left(\alpha_j^--\beta_j^-\right)\right),
\end{align*}
where $\langle\cdot,\cdot\rangle_{\ell_2}$ denotes the inner product on $\ell_2$. \autoref{lemma:A.7} shows that $\R^{L_N}_{s,1}$ satisfies strong convexity; that is, $D_{s,1}^{L_N}(\alpha^\pm,\beta^\pm)\geq \frac{1}{2L} \Vert \mathcal{T}^\sharp(\alpha^\pm)-\mathcal{T}^\sharp(\beta^\pm)\Vert_{\BB^s_{1,1}}^2$. For $q=1$, the stochastic subgradient projected onto $\BB_{L_N}^*$ is similarly given by
$$\widehat{\partial\mathcal{E}}(f)\big|_{\BB_{L_N}^*}= g_\ell(f(X_n),Y_n)\sum_{j=1}^{L_N}\phi_j(X_n)\phi_j^*, \quad \forall f\in\BB^s_{1,1}.$$
We further embed the projected stochastic subgradient into $\RR^{2L_N}$ by defining
$\xi_n^\pm(f)=g_\ell(f(X_n),Y_n)\cdot(\phi_1(X_n),\dots,\phi_{L_N}(X_n),-\phi_1(X_n),\dots,-\phi_{L_N}(X_n))$.
Then, given the $n$-th sample $(X_n,Y_n)$ with $1\leq n\leq N$, the stochastic mirror descent update takes the form
\begin{equation}\label{eq:stochastic mirror descent p equal 1}
\begin{aligned}
\alpha^{\pm,(n)} =& \mathop{\arg\min}_{\alpha^{\pm}\in\mathcal{A}_L^{2L_N}}\left\{ D_{s,1}^{L_N}(\alpha^{\pm,(n-1)},\alpha^\pm)+\eta_n\left\langle \xi_n^\pm(\mathcal{T}^\sharp \alpha^{\pm,(n-1)}), \alpha^{\pm}-\alpha^{\pm,(n-1)}\right\rangle_{\ell_2}\right\}
\end{aligned}
\end{equation}
where $\eta_n$ denotes the step size. The functional $\xi_n^\pm$ preserves the action of the original stochastic subgradient. Specifically, for any $\alpha^\pm\in\mathcal{A}_L^{2L_N}$, we have
\begin{equation}\label{eq:property of gradient}
\left\langle \xi_n^\pm(\mathcal{T}^\sharp \alpha^{\pm,(n-1)}), \alpha^{\pm}-\alpha^{\pm,(n-1)}\right\rangle_{\ell_2} = \left\langle \widehat{\partial\mathcal{E}}(\mathcal{T}^\sharp \alpha^{\pm,(n-1)})\big|_{\BB_{L_N}^*},\mathcal{T}^\sharp \alpha^{\pm}-\mathcal{T}^\sharp \alpha^{\pm,(n-1)}\right\rangle.
\end{equation}
For a fixed constant $0<\alpha<1$, we output the suffix averaging iterate defined by $ \bar{f}^s_{\alpha (N+1)}=\frac{1}{\alpha (N+1)}\sum_{n=(1-\alpha)(N+1)}^{N}\mathcal{T}^\sharp\alpha^{\pm,(n)}$. The iterates  $\{\alpha^{\pm,(n)}\}_{0\leq n\leq N}$ are required to remain in $\mathcal{A}_L^{2L_N}\cap\RR^{2L_N}_{++}$ to ensure that the Bregman distance in each subproblem is well defined. As shown by the subsequent analysis of the explicit solution to the subproblem, this condition is guaranteed whenever $\alpha^{\pm,(0)}\in\mathcal{A}_L^{2L_N}\cap\RR^{2L_N}_{++}$. 
In this paper, we initialize the iteration with $\alpha^{\pm,(0)}=\frac{3L}{\pi^2}(1^{-2s-2},2^{-2s-2},\dots,L_N^{-2s-2},1^{-2s-2},2^{-2s-2},\dots,L_N^{-2s-2})$. The analysis in \autoref{subsec:Theoretical Analysis} shows that this initialization is sufficient to guarantee the convergence of the algorithm. The above subproblem is a finite-dimensional constrained convex optimization problem. For notational simplicity, we write
$ \xi_n^\pm(\mathcal{T}^\sharp \alpha^{\pm,(n-1)})=\left( \xi_1^+,\dots,\xi_{L_N}^+,\xi_1^-,\dots,\xi_{L_N}^-\right)$
and $\alpha^{\pm,(n-1)}=\left(\alpha_1^{+,(n-1)},\dots \alpha_{L_N}^{+,(n-1)},\alpha_1^{-,(n-1)},\dots \alpha_{L_N}^{-,(n-1)}\right)$.
Using the Karush--Kuhn--Tucker conditions, we obtain the following explicit solution for $\alpha^{\pm,(n)}$: for each $1\leq j\leq L_N$,
$$\alpha^{+,(n)}_j= \frac{1}{Z_n} \alpha_j^{+,(n-1)}\exp(-j^{-2s}\eta_n\xi_j^+),\quad \alpha^{-,(n)}_j=\frac{1}{Z_n}\alpha_j^{-,(n-1)}\exp(-j^{-2s}\eta_n\xi_j^-),$$
where
$$Z_n =\max\left\{1, \frac1L \sum_{j=1}^{L_N}j^{2s}\left( \alpha_j^{+,(n-1)}\exp(-j^{-2s}\eta_n\xi_j^+)+\alpha_j^{-,(n-1)}\exp(-j^{-2s}\eta_n\xi_j^-)\right)\right\}.$$
It follows by induction from this explicit solution that, if $\alpha^{\pm,(0)}\in\mathcal{A}_L^{2L_N}\cap\RR^{2L_N}_{++}$, then $\alpha^{\pm,(n)} \in\mathcal{A}_L^{2L_N}\cap\RR^{2L_N}_{++}$ for all $1\leq n\leq N$.

\subsection{Theoretical Analysis}\label{subsec:Theoretical Analysis}

We first establish a convergence guarantee for the case $1<p<\infty$ and $f_\rho\in\BB_{p,p}^s$ in the following theorem. Let the samples $\{(X_i,Y_i)\}_{i\geq1}$ be independently and identically distributed according to $\rho$. For two nonnegative sequences $\{a_n\}_{n\geq 0}$ and $\{b_n\}_{n\geq 0}$, we write $a_n\lesssim b_n$ if there exists a constant $C>0$, independent of $n$, such that $a_n\leq Cb_n$ for all $n\geq 0$. We write $a_n\asymp b_n$ if both $a_n\lesssim b_n$ and $b_n\lesssim a_n$ hold.

\begin{theorem}\label{theorem:main}
Let $1<p<\infty$, and let $\{\phi_j\}_{j\geq 1}$ be a Schauder basis of $\mathcal{L}_{\rho_X}^p(\Omega)$ satisfying $\Vert\phi_j\Vert_{\mathcal{L}^p}\leq 1$. Assume that $s>\frac{1}{2}-\frac{1}{2p}$, $f_\rho\in\BB_{p,p}^s$, and $\mathbb{E}_\rho[|Y|^p]<\infty$. Set $p_1=\max\{p,2\}$ and $p_1'=\frac{p_1}{p_1-1}=\min\{\frac{p}{p-1},2\}$. Suppose that, for some $t\in[p_1,\infty)$, $\sum_{j=1}^{\infty}j^{-2sp_1'}\Vert\phi_j\Vert_{\LL^t}^{p_1'}=:M^2<\infty$. Assume that the loss function $\ell$ satisfies \autoref{assume:1} with $1\leq q\leq\min\left\{p,\frac{p+p_1}{2}-\frac{p}{t}\right\}$, and set $\nu=1$. If $\ell$ additionally satisfies \autoref{assume:2}, take $1<\nu\leq2$ to be the exponent specified therein. Choose $L_n=\lceil n^\theta\rceil$ with $\theta>\frac{p}{p_1\nu\left(2sp-(p-1)\right)}$, and set $\eta_n=\eta_0n^{-\frac{1}{p_1'}}\left(\ln(n+1)\right)^{-\frac{2}{p_1'}}$ with $\eta_0>0$. Then we have
\begin{equation*}
\begin{aligned}
\EE\left[\mathcal{E}(\bar{f}^s_{\alpha n})-\mathcal{E}(f_\rho)\right]\lesssim n^{-\frac{1}{p_1}}\left(\ln(n+1)\right)^{2(1-1/p_1)}.
\end{aligned}
\end{equation*}
\end{theorem}
\autoref{theorem:main} is proved in \autoref{proof of theorem 1}. To bound the optimization error, we combine classical stochastic mirror descent arguments \cite{nemirovski2009robust,nedic2014stochastic} with the uniform convexity of the underlying Banach space \cite{schuster2012regularization}. The approximation error is controlled by the approximation properties of the generalized Besov space $\BB^s_{p,p}$ and the local Lipschitz continuity implied by \autoref{assume:1}. Under \autoref{assume:2}, the additional smoothness of the loss yields a sharper approximation bound. Since uniform boundedness of the stochastic subgradients is not assumed, we use integral inequalities, such as H\"older's inequality, to control their moments. Our analysis requires only norm conditions on the basis, without the additional assumptions used in previous Hilbert space analyses, such as uniform boundedness and continuity \cite{zhang2022sieve} or wavelet regularity and vanishing moment conditions \cite{liautaud2026minimax}. The convergence rate in \autoref{theorem:main} reflects the convexity of the chosen mirror map. Its $\max\{p,2\}$-convexity yields a rate $\mathcal{O}\left(n^{-\min\{\frac12,\frac1p\}}\right)$, up to logarithmic factors. The analysis further reveals the regularizing role of $L_n$. For $f_\rho\in\BB_{p,p}^s$, the term $\R_{s,p}(S_{L_n}f_\rho)$ in the optimization error satisfies $\sup_{n\geq1}\R_{s,p}(S_{L_n}f_\rho)\leq\R_{s,p}(f_\rho)<\infty$, so this term does not affect the exponent of the optimization error. We therefore choose $\theta>\frac{p}{p_1\nu(2sp-(p-1))}$ to ensure that the approximation error satisfies $\ee(S_{L_n}f_\rho)-\ee(f_\rho)\lesssim n^{-1/p_1}$, preserving the stated convergence rate.

Although the distribution $\rho$ is typically unknown and our analysis requires $\{\phi_j\}_{j\geq1}$ to be a Schauder basis of $\mathcal{L}_{\rho_X}^p(\Omega)$, this requirement is not difficult to satisfy in practice. Let $\omega$ be a known measure on $\Omega$, such as the Lebesgue measure, and let $\{\phi_j\}_{j\geq1}$ be a Schauder basis of $\mathcal{L}_{\omega}^p(\Omega)$. By \autoref{lemma:A.9}, $\{\phi_j\}_{j\geq1}$ is also a Schauder basis of $\LL_{\rho_X}^p(\Omega)$ whenever the Radon-Nikodym derivative $\frac{d\rho_X}{d\omega}$ is bounded above and away from zero $\omega$-almost everywhere. Once such a basis is chosen, the remaining basis norm condition $\sum_{j=1}^\infty j^{-2sp_1'}\Vert \phi_j\Vert_{\LL^t}^{p_1'}<\infty$ in \autoref{theorem:main} can also be readily verified in practice. To illustrate how the norm condition can be checked for commonly used Schauder bases, we consider three classes of bases in the Lebesgue space $\L^p[0,1]$ with $1<p<\infty$: the Fourier basis $\{1\}\cup\{\cos(2\pi jx),\sin(2\pi jx)\}_{j\ge 1}$ \cite{heil2011basis}, the Haar basis \cite{albiac2006topics}, and the Franklin system \cite{bovckarev1976existence}. The Haar basis is a classical wavelet basis, defined by $\chi_1(x)=1$ and
$ \chi_{2^n+k}(x) = \sqrt{2^n} I_{\left[\frac{k-1}{2^n},\frac{2k-1}{2^{n+1}}\right)}-\sqrt{2^n} I_{\left[\frac{2k-1}{2^{n+1}},\frac{k}{2^n}\right)},\quad 1\leq k\leq 2^n, \quad n\geq0$.
Define $\psi_0(x) = 1$ and $\psi_j(x) = \int_{0}^x \chi_j(t)dt$ for $j\geq1$, and then apply the Schmidt orthonormalization procedure to $\{\psi_j\}_{j\ge 0}$ to obtain the Franklin system $\{\phi_j\}_{j\ge 0}$. Since the Fourier basis is uniformly bounded, the norm condition holds whenever $s>\frac{1}{2p_1'}$. For the Haar basis and the Franklin system, we use the estimates $\Vert\chi_{j}\Vert_{\L^p}\asymp j^{\frac12-\frac1p}$ and $\Vert\phi_{j}\Vert_{\L^p}\asymp j^{\frac12-\frac1p}$ \cite{ciesielski1966properties,bovckarev1976existence}. After normalizing these two systems in $\L^p[0,1]$, we obtain $\Vert \Vert\chi_{j}\Vert_{\L^p}^{-1}\chi_{j}\Vert_{\L^t}\asymp j^{\frac1p-\frac1t}$ and $\Vert \Vert\phi_{j}\Vert_{\L^p}^{-1}\phi_{j}\Vert_{\L^t}\asymp j^{\frac1p-\frac1t}$. The basis norm condition holds provided that $s>\frac12\left(\frac{1}{p_1'}+\frac1p-\frac1t\right)$.

Next, we establish a convergence guarantee for the case $p=1$.
\begin{theorem}\label{theorem:main1}
Let $s>0$, and let $\{\phi_j\}_{j\geq1}$ be a Schauder basis of $\mathcal{L}_{\rho_X}^1(\Omega)$ satisfying $\Vert\phi_j\Vert_{\mathcal{L}^1}\leq1$. Assume that $M_1:=\sup_{j\geq1}\left(j^{-2s}\sup_{x\in\Omega}|\phi_j(x)|\right)<\infty$ and that $f_\rho=\sum_{j=1}^\infty\beta_j^\rho\phi_j\in\mathcal{C}_L\subset\BB_{1,1}^s$ with $M_2:=\sum_{j=1}^\infty j^{2s}|\beta_j^\rho|\log j<\infty$. Suppose that the loss function $\ell$ satisfies \autoref{assume:1} with $q=1$ and that $\mathbb{E}_\rho[|Y|]<\infty$. Assume that the total sample size $N$ is known in advance. Choose the step size $\eta_n=\eta_0N^{-1/2}$ with $\eta_0>0$, and let $L_N=\lceil N^\theta\rceil$ with $\theta\geq1/(4s)$. Then, for any fixed $0<\alpha<1$, we have
\begin{equation*}
\begin{aligned}
\EE\left[\mathcal{E}(\bar{f}^s_{\alpha N})-\mathcal{E}(f_\rho)\right]\lesssim N^{-\frac12}.
\end{aligned}
\end{equation*}
\end{theorem}
As discussed in \autoref{sub:Stochastic Mirror Descent}, the entropy functional is strongly convex on bounded subsets of $\mathbb{R}_+^{2L_N}$. We therefore solve the SMD subproblem over $\mathcal{A}_L^{2L_N}$. The parameter $L$ must satisfy $\Vert f_\rho\Vert_{\BB_{1,1}^s}\leq L$, which is one of the key assumptions in \autoref{theorem:main1}. In practice, this requires either prior knowledge of $\Vert f_\rho\Vert_{\BB_{1,1}^s}$ or choosing $L$ sufficiently large to ensure that this condition holds.

General losses have been extensively studied in Hilbert spaces \cite{christmann2007consistency,marteau2019beyond,bai2025truncated}. In this work, we develop algorithms and establish convergence guarantees for general losses in Banach spaces with non-Euclidean geometry. A classical setting covered by this framework is regression with heavy-tailed responses. For $1\leq p<2$, the moment condition on the response is $\mathbb{E}_\rho[|Y|^p]<\infty$, which allows the marginal distribution $\rho_Y$ to have infinite variance. For example, consider the regression model $Y=f_\rho(X)+\delta$, where the conditional distribution of $\delta$ given $X$ is symmetric about zero. For symmetric convex losses, such as the absolute or Huber loss, this symmetry ensures that $f_\rho$ is a minimizer \cite{steinwart2007compare}. When both the noise $\delta$ and $f_\rho(X)$ are heavy-tailed and have finite $p$-th moments for some $1<p<2$, the response satisfies $\mathbb{E}_\rho[|Y|^p]<\infty$. In this case, \autoref{theorem:main} guarantees a convergence rate of $\mathcal{O}\left(n^{-\frac12}\ln(n+1)\right)$ when $1<p<2$, whereas \autoref{theorem:main1} gives a rate of $\mathcal{O}\left(N^{-\frac12}\right)$ when $p=1$.

We next establish convergence results in the misspecified case, i.e., $f_\rho\notin\BB_{p,p}^s$. With $0<r\leq 1$, we define the lower-regularity generalized Besov space $\BB^{s,r}_{p,p}$ as
$$\BB^{s,r}_{p,p}:=\left\{ f=\sum_{j=1}^\infty \beta_j\phi_j\,\Big\vert\, \sum_{j=1}^\infty j^{2sr p}|\beta_j|^p<\infty \right\} \cap\mathcal{L}_{\rho_X}^p(\Omega).$$
If $r=1$, then $\BB^{s,r}_{p,p}=\BB^{s}_{p,p}$, and $\BB^{s,r_2}_{p,p}\subset\BB^{s,r_1}_{p,p}$ for all $0<r_1\leq r_2\leq1$. The space $\BB^{s,r}_{p,p}$ exhibits weaker regularity as $r$ decreases.

\begin{theorem}\label{theorem:main2}
Let $1<p<\infty$, and let $\{\phi_j\}_{j\geq1}$ be a Schauder basis of $\mathcal{L}_{\rho_X}^p(\Omega)$ satisfying $\Vert\phi_j\Vert_{\mathcal{L}^p}\leq1$. Assume that $s>\frac{1}{2}-\frac{1}{2p}$, $f_\rho\in\BB_{p,p}^{s,r}$ with $0<r<1$, and $\mathbb{E}_\rho[|Y|^p]<\infty$. Set $p_1=\max\{p,2\}$ and $p_1'=\frac{p_1}{p_1-1}$. Suppose that, for some $t\in[p_1,\infty)$, $\sum_{j=1}^{\infty}j^{-2sp_1'}\Vert\phi_j\Vert_{\LL^t}^{p_1'}=:M^2<\infty$. Assume that the loss function $\ell$ satisfies \autoref{assume:1} with $1\leq q\leq\min\left\{p,\frac{p+p_1}{2}-\frac{p}{t}\right\}$, and set $\nu=1$. If $\ell$ additionally satisfies \autoref{assume:2}, take $\nu$ to be the exponent specified therein. Choose $\eta_n=\eta_0n^{-\tau}$ with $1-\frac{1}{p_1}<\tau<1$ and $\eta_0>0$, and $L_n=\lceil n^\theta\rceil$ with $0<\theta<\frac{1-\tau}{2sp_1(1-r)}$. Then we have
\begin{equation*}
\begin{aligned}
\EE\left[\mathcal{E}(\bar{f}^s_{\alpha n})-\mathcal{E}(f_\rho)\right]\lesssim n^{-1+\tau}\sum_{i=1}^ni^{-\tau}\left(\mathcal{E}\left(S_{L_i}(f_\rho)\right)-\mathcal{E}(f_\rho)\right)+
n^{-1+\tau+2sp_1\theta(1-r)}.
\end{aligned}
\end{equation*}
Since $\theta>0$, we have $\lim_{i\to\infty} \mathcal{E}\left(S_{L_i}(f_\rho)\right)=\mathcal{E}(f_\rho)$ and hence we obtain
\begin{equation*}
\begin{aligned}
\lim_{n\to\infty}\EE\left[\mathcal{E}(\bar{f}^s_{\alpha n})-\mathcal{E}(f_\rho)\right]=0.
\end{aligned}
\end{equation*}
\end{theorem}
\autoref{theorem:main2} establishes convergence of our algorithm in the misspecified setting, but does not provide an explicit convergence rate. When $\frac{1}{2sp'}<r<1$, \autoref{lemma:A.10} provides an explicit polynomial rate for the approximation error. Using this estimate, we derive a polynomial convergence rate in \autoref{lemma:main1}. For $0<r\leq\frac{1}{2sp'}$, however, the present argument only establishes $\lim_{i\to\infty}\left[\ee(S_{L_i}f_\rho)-\ee(f_\rho)\right]=0$, without providing a polynomial rate for this approximation error. Additional structural assumptions on the Schauder basis, such as orthogonality or a $p$-frame condition, may allow us to derive an explicit rate for this approximation error.
\begin{lemma}\label{lemma:main1}
Under the assumptions of \autoref{theorem:main2}, suppose that $1\leq q\leq \min\left\{\frac{p+p_1}{2}-\frac{p}{t},p\right\}$ and $\frac{1}{2sp'}<r<1$. Choose $\eta_n=\eta_0n^{-\tau}$ with $\eta_0>0$ and $1-\frac{1}{p_1}<\tau<1$, and set $\theta=\frac{1-\tau}
{2sp_1(1-r)+\nu(2sr-1/p')}$ with $1\leq \nu\leq2$. Then
$$\EE\left[\mathcal{E}(\bar{f}^s_{\alpha n})-\mathcal{E}(f_\rho)\right]\lesssim n^{-\frac{(1-\tau)\nu(2sr-1/p')}{2sp_1(1-r)+\nu(2sr-1/p')}}.$$
\end{lemma}
The bound in \autoref{theorem:main2} shows that increasing $\theta$ improves the approximation error bound but worsens the optimization error bound. To make the regularizing role of $L_n$ more explicit, suppose that $\ee(S_{L_n}f_\rho)-\ee(f_\rho)\lesssim n^{-\theta\beta}$ for some $\beta>0$. If $0<\theta\beta<1-\tau$, then $$n^{-1+\tau}\sum_{i=1}^n i^{-\tau}\left[\ee(S_{L_i}f_\rho)-\ee(f_\rho)\right]\lesssim n^{-\theta\beta}.$$ Consequently, \autoref{theorem:main2} gives $\EE\left[\ee(\bar f^s_{\alpha n})-\ee(f_\rho)\right]\lesssim n^{-\theta\beta}+n^{-1+\tau+2sp_1\theta(1-r)}$. For fixed $\tau$, choosing $\theta=\frac{1-\tau}{\beta+2sp_1(1-r)}$ ensures that $0<\theta\beta<1-\tau$ and makes the two terms decay at the same rate, yielding a convergence rate of $\O\left(n^{-\frac{(1-\tau)\beta}{\beta+2sp_1(1-r)}}\right)$. When $\frac{1}{2sp'}<r<1$, \autoref{lemma:A.10} provides $\beta=\nu(2sr-1/p')$, which gives the choice of $\theta$ and the convergence rate stated in \autoref{lemma:main1}.

\subsection{Extension to Statistical Inverse Problems}\label{subsec:Link with Inverse Problems in Banach Spaces}

We consider statistical inverse problems that aim to recover an unknown function $f^\dag$ from possibly noisy, randomly sampled point evaluations of its image under an operator $A$ that need not admit a bounded inverse. Let $A:\BB\to\mathcal{Z}$ be a known bounded linear operator, where $\BB$ is a Banach space and $\mathcal{Z}$ is a Banach space of real-valued functions on $\RR$. We seek to recover $f^\dag\in\BB$ from the noisy observations 
\begin{equation}\label{eq:inverse_problem} 
Y_i=(Af^\dag)(X_i)+\delta_i,\qquad i\geq1, 
\end{equation} 
where the pairs $\{(X_i,\delta_i)\}_{i\geq1}$ are independent and identically distributed, $X_i$ is drawn from a typically unknown probability distribution $\rho_X$, and $\EE[\delta_i | X_i]=0$. For a generic observation $\xi=(X,Y)$, we consider the loss $\ell_A(f,\xi)=\ell((Af)(X),Y)$. The associated risk functional is $\mathcal{E}_A(f)=\EE[\ell_A(f,\xi)]$. When $A$ is the identity operator, we recover the risk functional minimization problem considered in this paper. Large-scale inverse problems in science and engineering often involve large amounts of data collected under sampling conditions that cannot be fully controlled or designed. Such problems naturally motivate statistical formulations with unknown sampling distributions, contributing to the growing interest in statistical inverse problems. For instance, \cite{blanchard2018optimal} and \cite{helin2024least} studied spectral regularization and regularization by projection in Hilbert spaces, respectively, while \cite{bubba2023convex} investigated Tikhonov regularization with convex penalties in Banach spaces. Stochastic approximation methods use only a small, randomly sampled portion of the available information at each iteration and therefore scale efficiently with the problem size. In the classical framework of inverse problems in Banach spaces, \cite{jin2023convergence} and \cite{jin2023stochastic,huang2025early} analyzed the convergence of stochastic gradient descent and stochastic mirror descent, respectively. These studies consider finite systems of operator equations with deterministically perturbed data and losses of least-squares type, and establish convergence under suitable stopping rules as the noise bound tends to zero. Extending these methods to the statistical inverse problem \eqref{eq:inverse_problem} in Banach spaces with general loss functions, where data are sampled from an unknown distribution, while maintaining computational and storage efficiency remains challenging, particularly when only moment conditions are imposed, allowing unbounded noise such as Gaussian noise.

Our algorithm extends naturally to the statistical inverse problem \eqref{eq:inverse_problem}. We take $\BB$ to be the generalized Besov space $\BB^s_{p,p}$ and set $\mathcal{Z}=\L^p_{\rho_X}(\RR)$. Assume that the loss function $\ell$ satisfies \autoref{assume:1} with $1\leq q\leq p$ and that $\mathbb{E}[|\delta|^p]<\infty$. We consider the risk functional $\mathcal{E}_A(f)=\EE[\ell((Af)(X),Y)]$. For notational convenience, we write $A_X(f):=(Af)(X)$. For $1<p<\infty$, upon receiving the $n$-th sample $(A_{X_n},Y_n)$, the SMD update is given by
\begin{equation*}
\begin{aligned}
f_{n} = &\mathop{\arg\min}_{f\in\BB_{L_n}}\left\{ D_{s,p}(f_{n-1},f)+\eta_n\left\langle \zeta_n, f-f_{n-1}\right\rangle\right\}\\
\left\langle \zeta_n, f-f_{n-1}\right\rangle :=& g_\ell(A_{X_n}f_{n-1},Y_n)A_{X_n}(f-f_{n-1}).
\end{aligned}
\end{equation*}
Let $a_j(X_n):=A_{X_n}\phi_j$. The above subproblem admits a closed-form solution of the same form as \eqref{eq:stochastic mirror descent}. The coordinate-wise coefficient updates for the statistical inverse problem are obtained from Algorithm~\ref{alg:theoretical} by replacing $w(f_{n-1})$ with $g_\ell(A_{X_n} f_{n-1},Y_n)$ and $\phi_j(X_n)$ with $a_j(X_n)$. The case $p=1$ is similar, and the algorithm in \autoref{sub:Stochastic Mirror Descent for p equal 1} can be extended directly to this setting with a minor modification. We still use $\mathcal{T}$ to map $\BB_{L_N}\cap\mathcal{C}_L$ into $\mathcal{A}_L^{2L_N}$ by decomposing the coefficients of each element into their positive and negative parts, with $\mathcal{T}^\sharp$ as the corresponding synthesis map. Using the mirror map $\R_{s,1}^{L_N}$ and its Bregman distance $D_{s,1}^{L_N}$, the SMD update takes the form
\begin{equation*}
\begin{aligned}
\alpha^{\pm,(n)} =& \mathop{\arg\min}_{\alpha^{\pm}\in\mathcal{A}_L^{2L_N}}\left\{ D_{s,1}^{L_N}(\alpha^{\pm,(n-1)},\alpha^\pm)+\eta_n\left\langle \zeta_{n}^\pm, \alpha^{\pm}-\alpha^{\pm,(n-1)}\right\rangle\right\}\\
\left\langle \zeta_{n}^\pm, \alpha^{\pm}-\alpha^{\pm,(n-1)}\right\rangle:=& g_\ell(A_{X_n}\mathcal{T}^\sharp \alpha^{\pm,(n-1)},Y_n)A_{X_n}\left(\mathcal{T}^\sharp \alpha^{\pm}-\mathcal{T}^\sharp \alpha^{\pm,(n-1)}\right).
\end{aligned}
\end{equation*}
We also use the same initialization as in \autoref{sub:Stochastic Mirror Descent for p equal 1}. With $r_n=g_\ell(A_{X_n}\mathcal{T}^\sharp \alpha^{\pm,(n-1)},Y_n)$, we similarly obtain
$$\alpha^{\sigma,(n)}_j= \frac{1}{Z_n} \alpha_j^{\sigma,(n-1)}\exp\left(-\sigma j^{-2s}\eta_nr_na_j(X_n)\right),\quad \forall \sigma\in\{-1,+1\},$$
where $Z_n =\max\left\{1, \frac1L \sum_{j=1}^{L_N}\sum_{\sigma\in\{-1,+1\}}j^{2s}\left( \alpha_j^{\sigma,(n-1)}\exp\left(-\sigma j^{-2s}\eta_nr_na_j(X_n)\right)\right)\right\}$. Finally, we output the suffix average $ \bar{f}^s_{\alpha (N+1)}=\frac{1}{\alpha (N+1)}\sum_{n=(1-\alpha)(N+1)}^{N}\mathcal{T}^\sharp\alpha^{\pm,(n)}$. We thus extend our algorithm to statistical inverse problems, leaving the corresponding convergence analysis for future work.

\section{Numerical Experiments}\label{Numerical simulations}

In this section, we present numerical experiments to assess the performance of the proposed algorithm. \autoref{subsec:Nonparametric Regression in Banach Spaces} considers a regression problem in Banach spaces and illustrates the theoretical guarantees established in \autoref{subsec:Theoretical Analysis}. \autoref{subsec:Linear Inverse Problem} investigates the applicability of the proposed algorithm to ill-posed inverse problems.

\subsection{Nonparametric Regression}\label{subsec:Nonparametric Regression in Banach Spaces}

We use the Fourier, Haar, and Franklin systems in $\L^p([0,1])$. The regression model is $Y=f_\rho(X)+\delta$, where $X\sim\mathcal{U}[0,1]$ and the additive noise $\delta$ is
independent of $X$ and symmetric about zero. We consider the loss function $\ell(u,y)=|u-y|^q$, with $1\leq q<p$ for $p>1$ and $q=1$ for $p=1$. Under these conditions, $f_\rho$ is a minimizer of the corresponding risk functional \cite{steinwart2007compare}. The specific values of $q$ for $p>1$ are listed in
\autoref{Table1}. The function $f_\rho$ is defined as
\begin{align*}
f_\rho(x)=&\exp\big(0.7\cos(2\pi x)\big)\sin(2\pi x) + 0.5\sin(6\pi x) - 0.3\sin(10\pi x) + 0.1\sin(14\pi x) \\
&+ 2\left|\cos(2\pi x)\right|^{7/2}- 0.5\left|\sin\big(2\pi(x-0.25)\big)\right|^{7/2}.
\end{align*}
We separately investigate the empirical behavior of the algorithm for $p\in(1,\infty)$ and $p=1$.
\subsubsection{The Case \texorpdfstring{$1<p<\infty$}{p>1}}
For $1<p<\infty$, the Fourier, Haar, and Franklin systems are Schauder bases of $\L^p([0,1])$. Since $f_\rho$ is three times continuously differentiable, it belongs to the generalized Besov space $\BB_{p,p}^s$ induced by the Fourier basis for every $s<1$. We consider two experimental settings. In the first, we use only the Fourier basis and choose the hyperparameters to satisfy the assumptions of \autoref{theorem:main}. In the second, we use all three bases and adopt more aggressive choices of $\theta$ and $s$ that go beyond the scope of the theorem: a smaller $\theta$ reduces the algorithmic complexity, while a smaller $s$ yields larger effective learning rates for high-index basis coefficients. We implement the algorithm in $\L^{1.35}([0,1])$, $\L^{1.75}([0,1])$, $\L^{2.5}([0,1])$, and $\L^4([0,1])$. For the first two spaces, we examine the performance of the algorithm under heavy-tailed noise generated from an $\alpha$-stable distribution $S(\alpha,0)$ \cite{chambers1976method}, which has a finite $\alpha_1$-th moment for every $\alpha_1<\alpha$. For the latter two spaces, we consider uniformly distributed noise $\mathcal{U}[-0.05,0.05]$. The detailed simulation settings are summarized in \autoref{Table1}.

\begin{table}[htbp]
\centering
\setlength\tabcolsep{7mm}
\renewcommand{\arraystretch}{1.3}
\resizebox{1.0\linewidth}{!}{
\begin{tabular}{|c|c|c|c|c|}
\hline
\  &  Example 1 & Example 2 & Example 3 & Example 4\\ \hline
$p$ & 1.35 & 1.75 & 2.5 & 4.0\\ 
$q$ & 1.25 & 1.5 & 2.0 & 3.0 \\
\makecell{Fourier \\theorem-aligned\\ $(\eta_n,s,\theta)$} &  $(10n^{-\frac12},0.65,0.4)$ & $(5n^{-\frac12},0.65,0.4)$ & $(10n^{-\frac35},0.6,0.4)$ & $(3n^{-\frac34},0.55,0.4)$ \\
\makecell{General setting\\ $(\eta_n,s,\theta)$} &  $(5n^{-\frac12},0.5,0.35)$ & $(2.5n^{-\frac12},0.5,0.35)$ & $(2n^{-\frac12},0.4,0.35)$ & $(3n^{-\frac12},0.4,0.35)$ \\
$\delta$ & $S(1.36,0)$ & $S(1.76,0)$ & $\mathcal{U}[-0.05,0.05]$ & $\mathcal{U}[-0.05,0.05]$ \\
\hline
\end{tabular}
}
\caption{Experimental settings for $1<p<\infty$}
\label{Table1}
\end{table}
The experimental results are reported in \autoref{fig:1}. The black dashed line represents the reference rate predicted in \autoref{theorem:main}, with slope $-\frac{1}{\max\{p,2\}}$ on the log--log scale. The first group, shown by dashed curves, uses hyperparameters satisfying the assumptions and agrees with the predicted rate. The results of the second group, corresponding to the three bases, are shown by solid curves of different colors. Under these hyperparameter choices, the Fourier basis and the Franklin system exhibit convergence behavior comparable to, or even better than, that observed in the first group. One possible explanation is that decreasing $s$ facilitates the learning of high-index basis coefficients, thereby helping to reduce the approximation error, while a smaller $\theta$ helps control the optimization error. Together, these choices may provide effective regularization. In contrast, the Haar basis consists of piecewise constant functions and is less efficient for approximating smooth functions \cite{strang1993wavelet}, which may explain its weaker numerical performance. For the larger values of $q$, the error is more likely to remain relatively high during the initial iterations, suggesting that the algorithm is more sensitive to large residuals.

\begin{figure}[htbp]
\centering
\subfigure{
\includegraphics[width=0.4\linewidth]{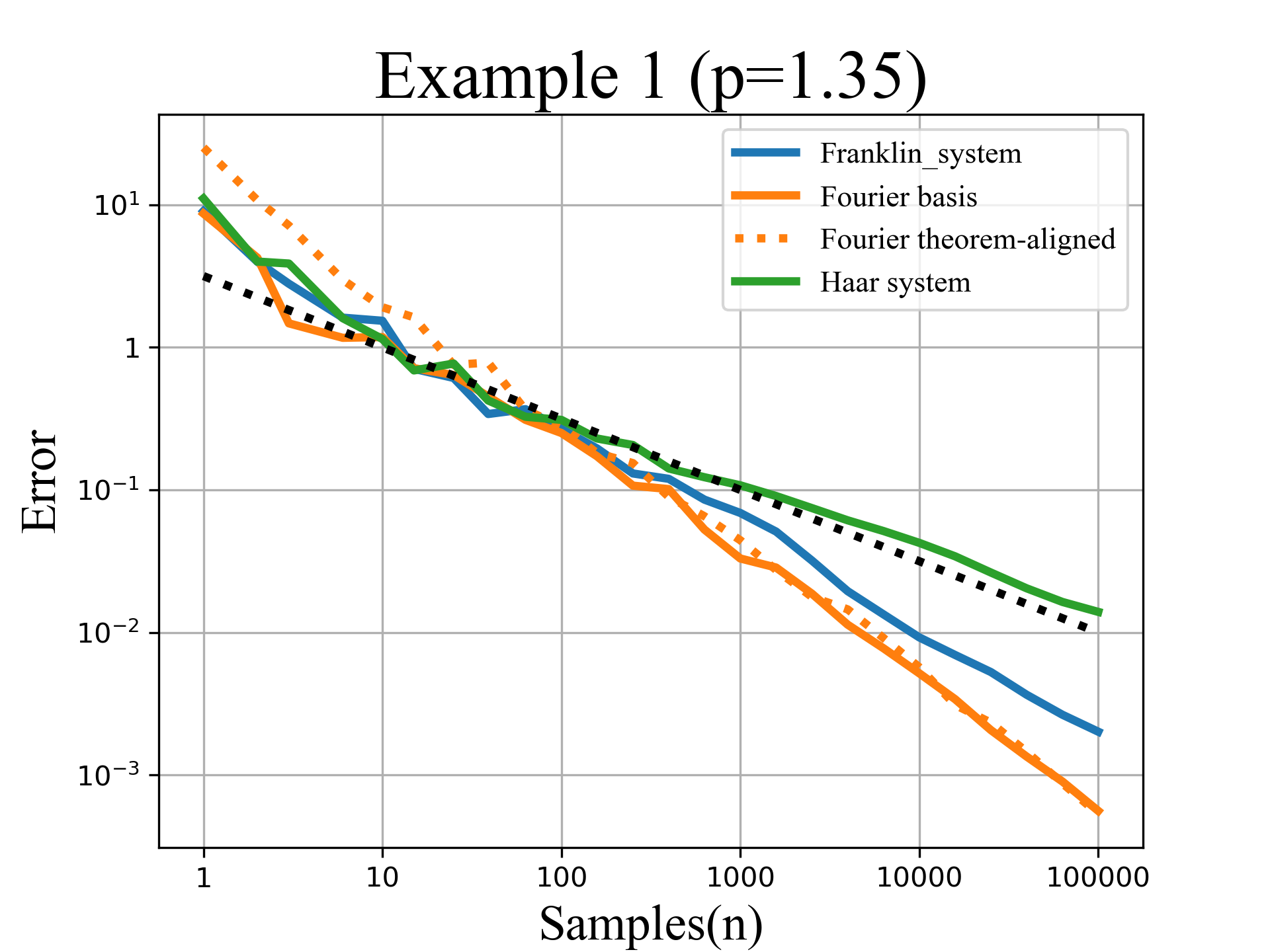}
}
\subfigure{
\includegraphics[width=0.4\linewidth]{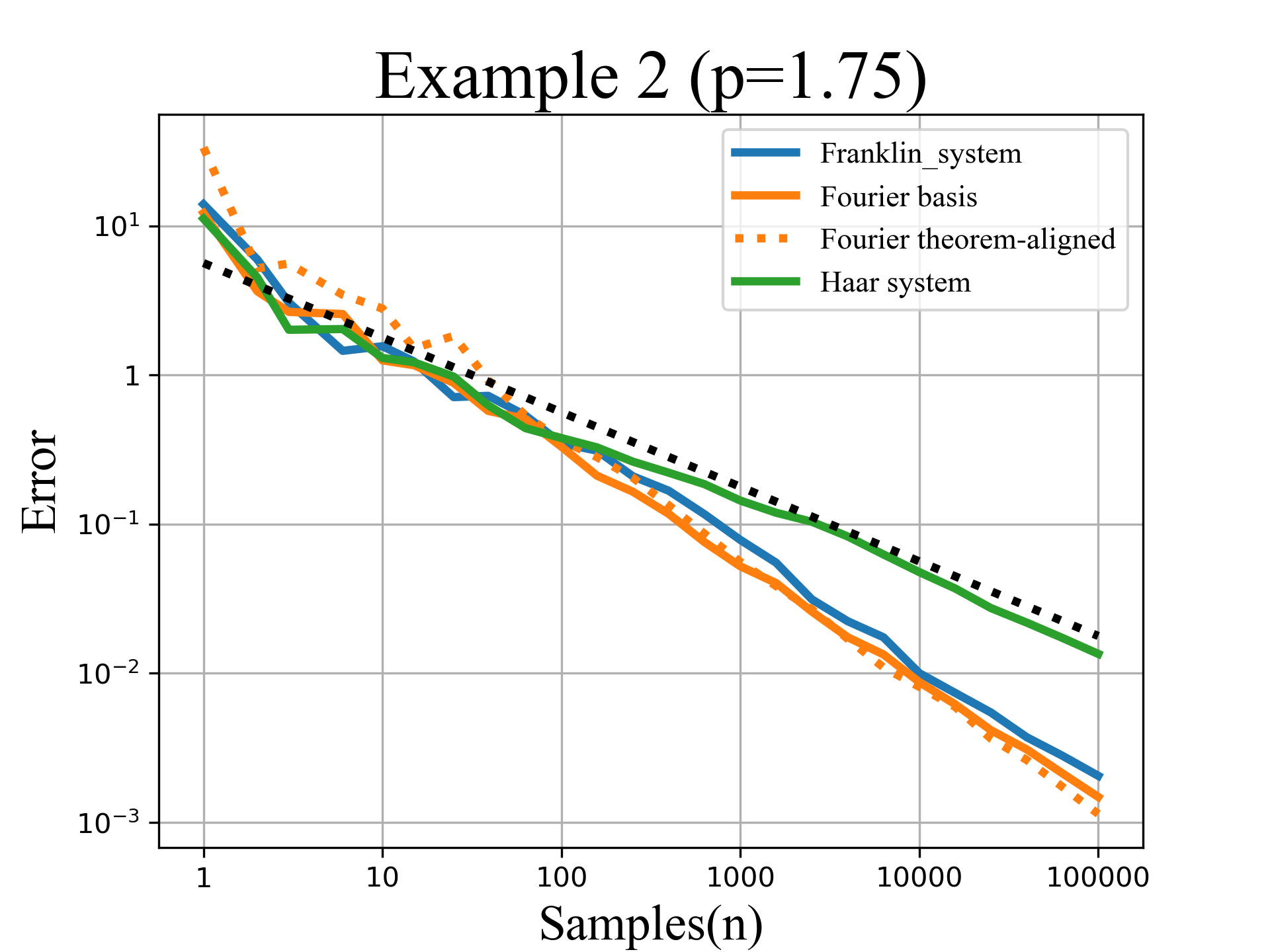}
}
\\
\subfigure{
\includegraphics[width=0.4\linewidth]{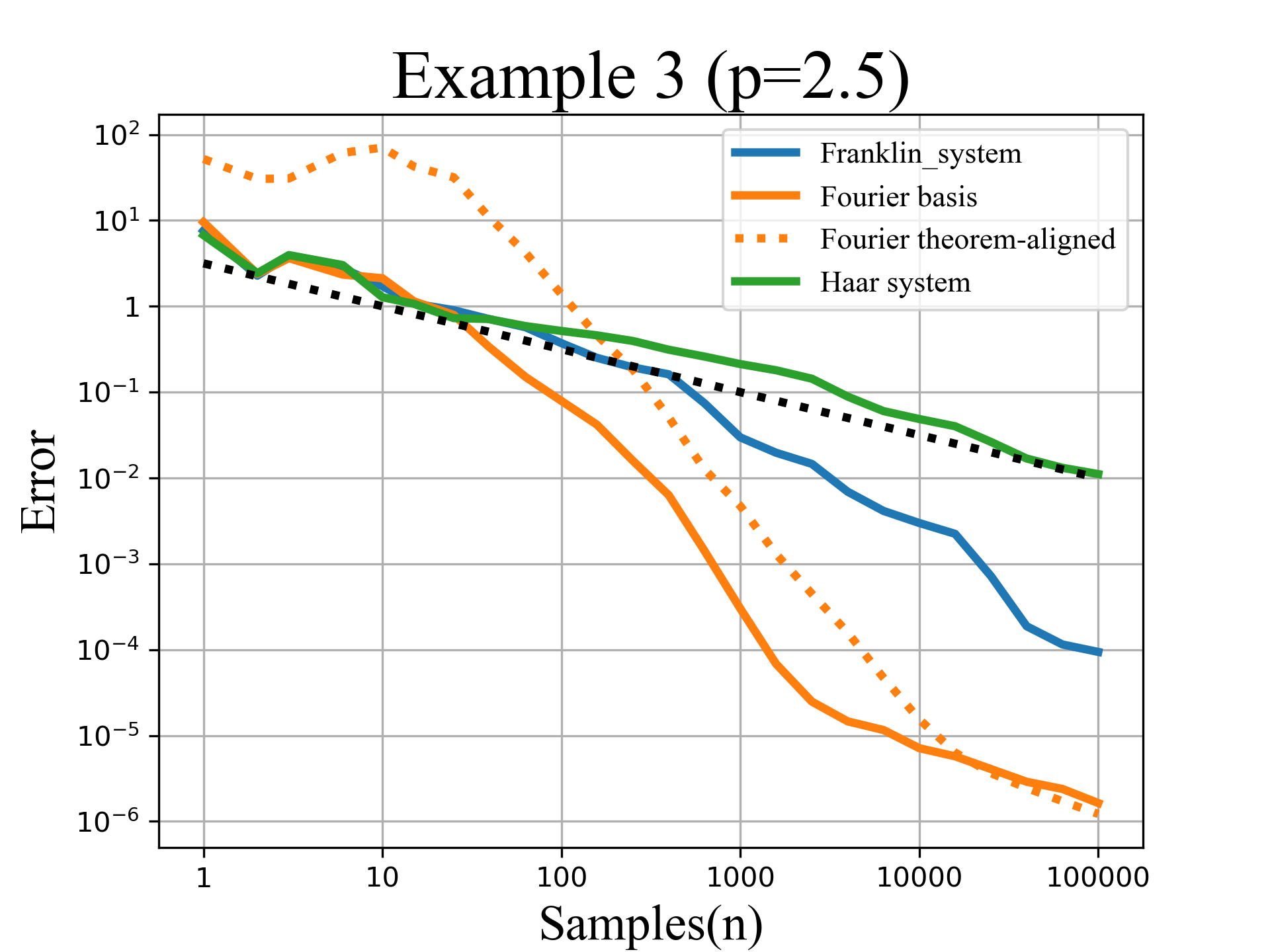}
}
\subfigure{
\includegraphics[width=0.4\linewidth]{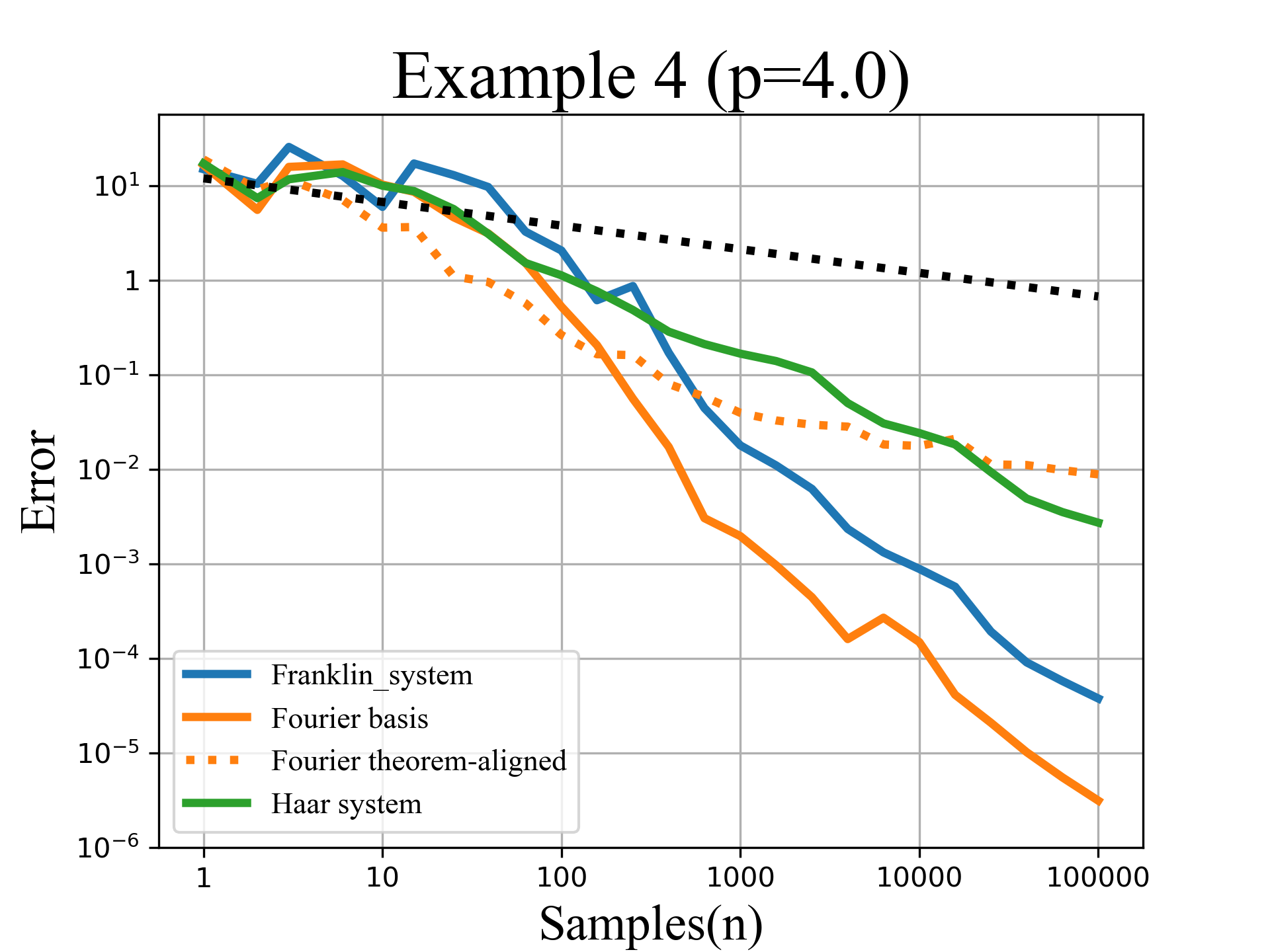}
}
\caption{The error shown in the figure is defined as
$Error=\EE\left[\ee(\bar{f}_{n/2}^s)-\ee(f_\rho)\right]$.}
\label{fig:1}
\end{figure}

\subsubsection{The Case \texorpdfstring{$p=1$}{p=1}}
For $p=1$, the Haar basis and the Franklin system are Schauder bases of $\L^1([0,1])$, whereas the Fourier system is no longer a Schauder basis. For the Franklin system, the Ciesielski inequality in \cite{ciesielski1966properties} implies that $f_\rho=\sum\beta_j^\rho\phi_j\in\BB_{1,1}^s$ and $\sum j^{2s}|\beta_j^\rho|\log j<\infty$ for $s<\frac{3}{4}$. Similarly, we construct two groups of experiments. The first group uses only the Franklin system, with hyperparameters $(\eta_n,s,\theta,L)=(7.5N^{-1/2},0.5,0.5,40)$ chosen to satisfy the assumptions of \autoref{theorem:main1}. The second group uses all three systems, with hyperparameters $(\eta_n,s,\theta,L)=(5N^{-1/2},0.25,0.35,10)$. In both groups, we consider heavy-tailed noise generated from $S(1.1,0)$. The experimental results are reported in \autoref{fig:2}. Although the Fourier system is no longer a Schauder basis, the algorithm still converges effectively. This suggests that the algorithm may remain effective under the weaker condition $\sup_{n\geq1}\Vert S_n(f_\rho)\Vert<\infty$. The results of Examples 1 and 2, together with the $p=1$ experiment, further indicate that the proposed algorithm can effectively handle heavy-tailed noise.

\begin{figure}[htbp]
\centering
\subfigure{
\includegraphics[width=0.4\linewidth]{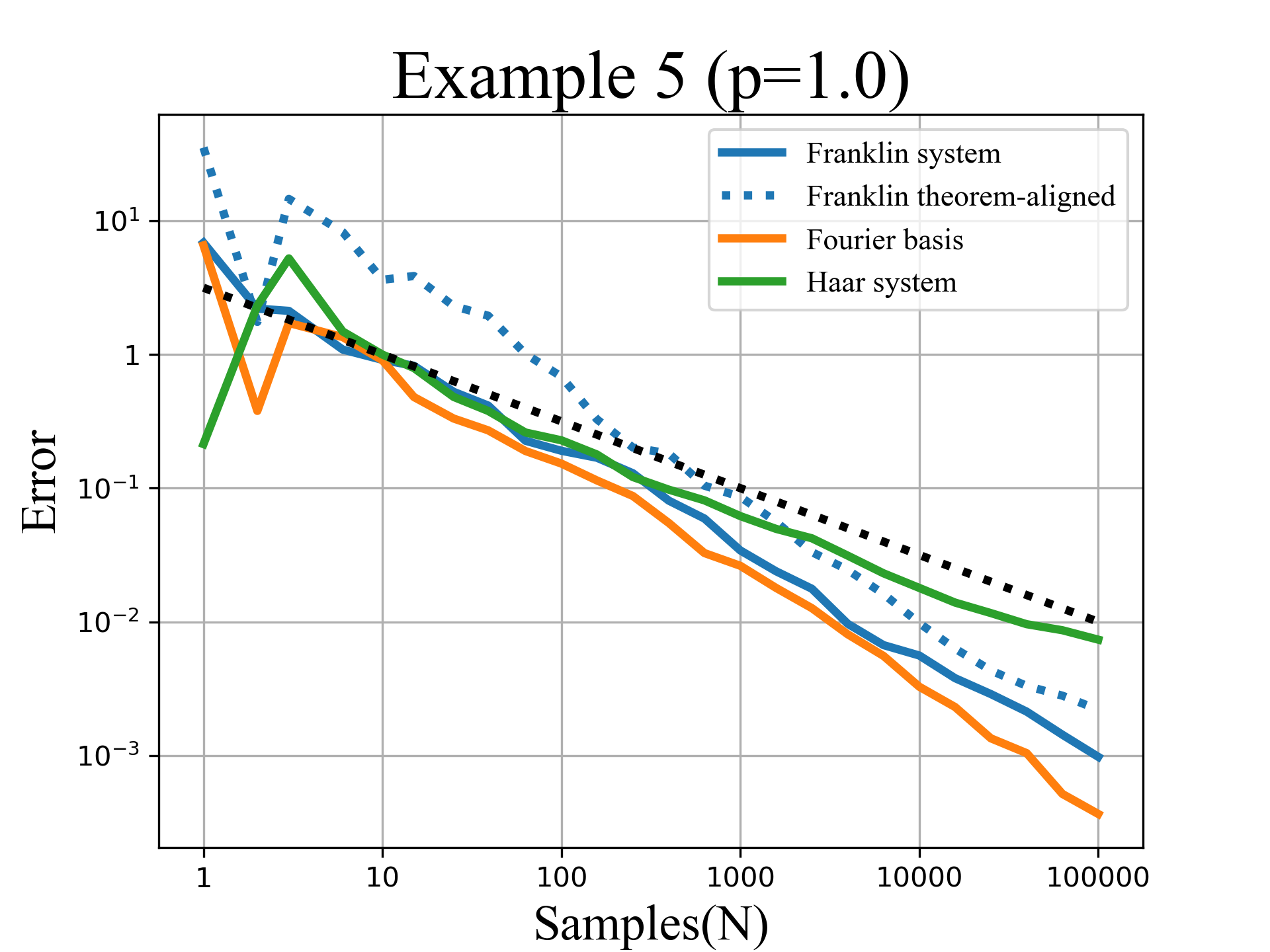}
}
\caption{The error shown in the figure is defined as
$Error=\EE\left[\ee(\bar{f}_{N/2}^s)-\ee(f_\rho)\right]$.}
\label{fig:2}
\end{figure}

\subsection{Statistical Inverse Problems}\label{subsec:Linear Inverse Problem}

In \autoref{subsec:Link with Inverse Problems in Banach Spaces}, we extended the AS-SMD algorithm to statistical inverse problems in Banach spaces. Here, we examine its empirical performance through a concrete example. We consider the linear integral operator $A:\BB^s_{p,p}\to\mathcal{L}^p([0,1])$ defined by
\begin{equation*}
    (Af)(x)=\int_{[0,1]}\left(1-2|x'-x|\right)_+f(x')\,dx',
    \qquad x\in[0,1].
\end{equation*}
The observations follow the model $Y=(Af^\dag)(X)+\delta$, where $X$ follows the uniform distribution $U[0,1]$, $\delta$ is Gaussian noise with distribution $N(0,0.5^2)$, and $f^\dag\in\BB^s_{p,p}$ is a continuous function on $[0,1]$ specified below. We again use the loss function $\ell(u,y)=|u-y|^q$.

We consider two experimental settings corresponding to $p=2$ and $p=1$. We use the Fourier system to implement AS-SMD, since $\{A_X(\phi_j)\}$ can be readily computed. For $p=2$, we set $q=2$ and choose the hyperparameters of AS-SMD as $\eta_n=4n^{-1/2}$, $s=0.25$, and $\theta=0.35$. As a baseline, we implement the stochastic gradient descent method in \cite{jin2023convergence} with step size $\eta_n=0.5n^{-\frac12}$ and approximate $A_Xf$ using a quadrature rule on a uniform partition of $[0,1]$ into $M$ subintervals, where $M=1000$ or $M=5000$. The exact solution is chosen as $f^\dag(x) = \sin(2\pi x)+\cos(8\pi x)+\exp(x)-(ex+(1-x))+\frac32(1-4|x-\frac12|)_+$. For $p=1$, we set $q=1$ and use the method described in Example~5.4 of \cite{huang2025early} as the baseline, with step size $\eta_n=0.5n^{-1/2}$. We use the same quadrature procedure to approximate $A_Xf$. Since the formulation in Example~5.4 imposes nonnegativity and normalization constraints on functions in $\L^1[0,1]$, whereas AS-SMD does not, we consider two choices of $f^\dag$. We use $f_1^\dag(x)=4(1-4|x-\frac12|)_+$ for comparison with the baseline method and $f_2^\dag(x)=\sin(2\pi x)+\cos(8\pi x)+\exp(x)-(ex+(1-x))+\frac32(1-4|x-\frac12|)_+$ to examine the empirical performance of AS-SMD in a more general setting. The results for $f_1^\dag$ and $f_2^\dag$ are shown by solid and dotted lines, respectively, in \autoref{fig:3}. For AS-SMD, we set the hyperparameters to $\eta_n=5N^{-1/2}$, $s=0.25$, $\theta=0.35$, and $L=20$. The results in \autoref{fig:3} show that AS-SMD achieves a steady reduction in the error with randomly sampled streaming data corrupted by Gaussian noise. In the settings where the baseline method is applicable, AS-SMD exhibits faster error decay after $10000$ iterations. The results for $f_2^\dag$ further demonstrate its applicability without the nonnegativity and normalization constraints.
\begin{figure}[htbp]
\centering
\subfigure{
\includegraphics[width=0.4\linewidth]{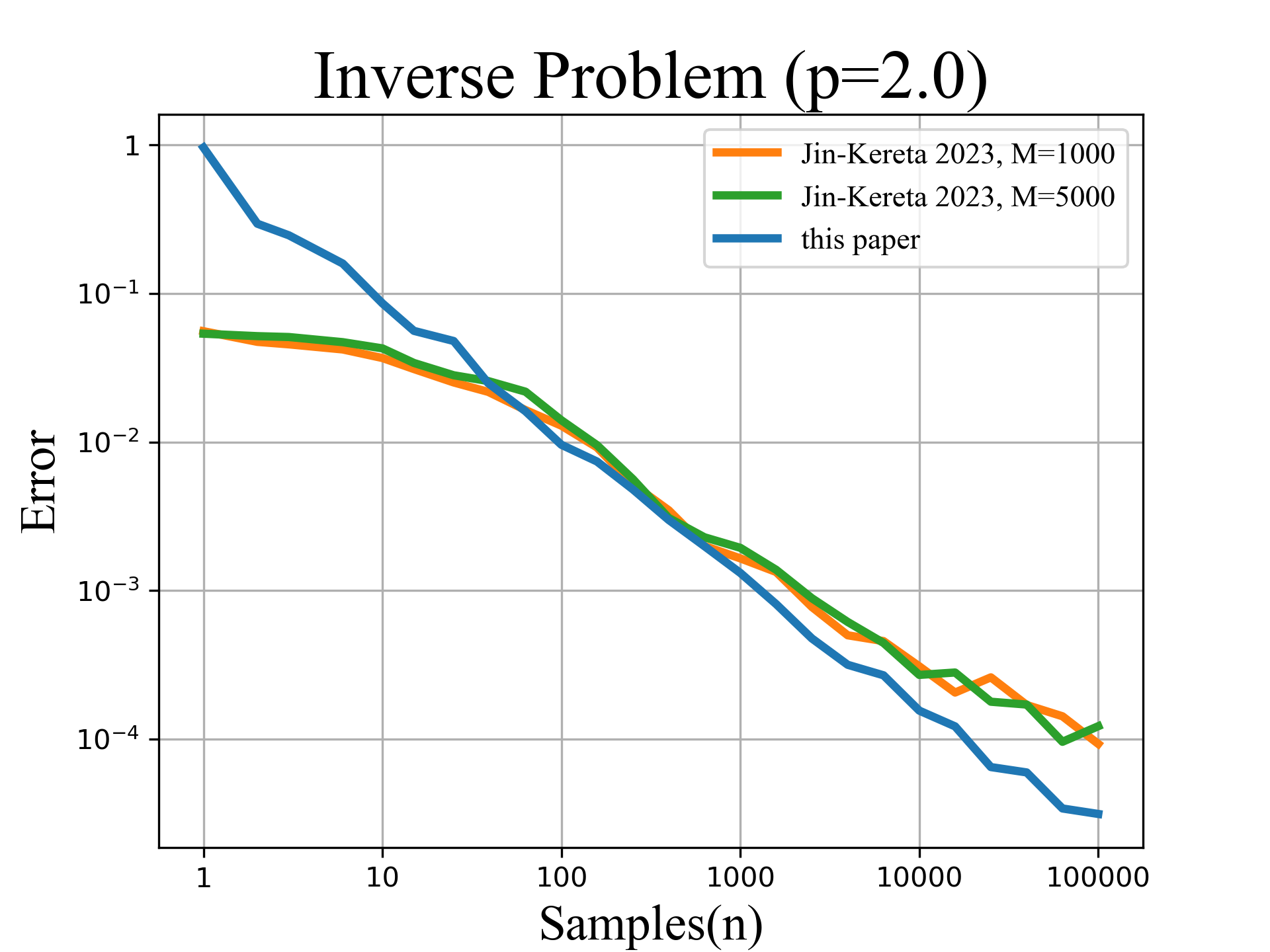}
}
\subfigure{
\includegraphics[width=0.4\linewidth]{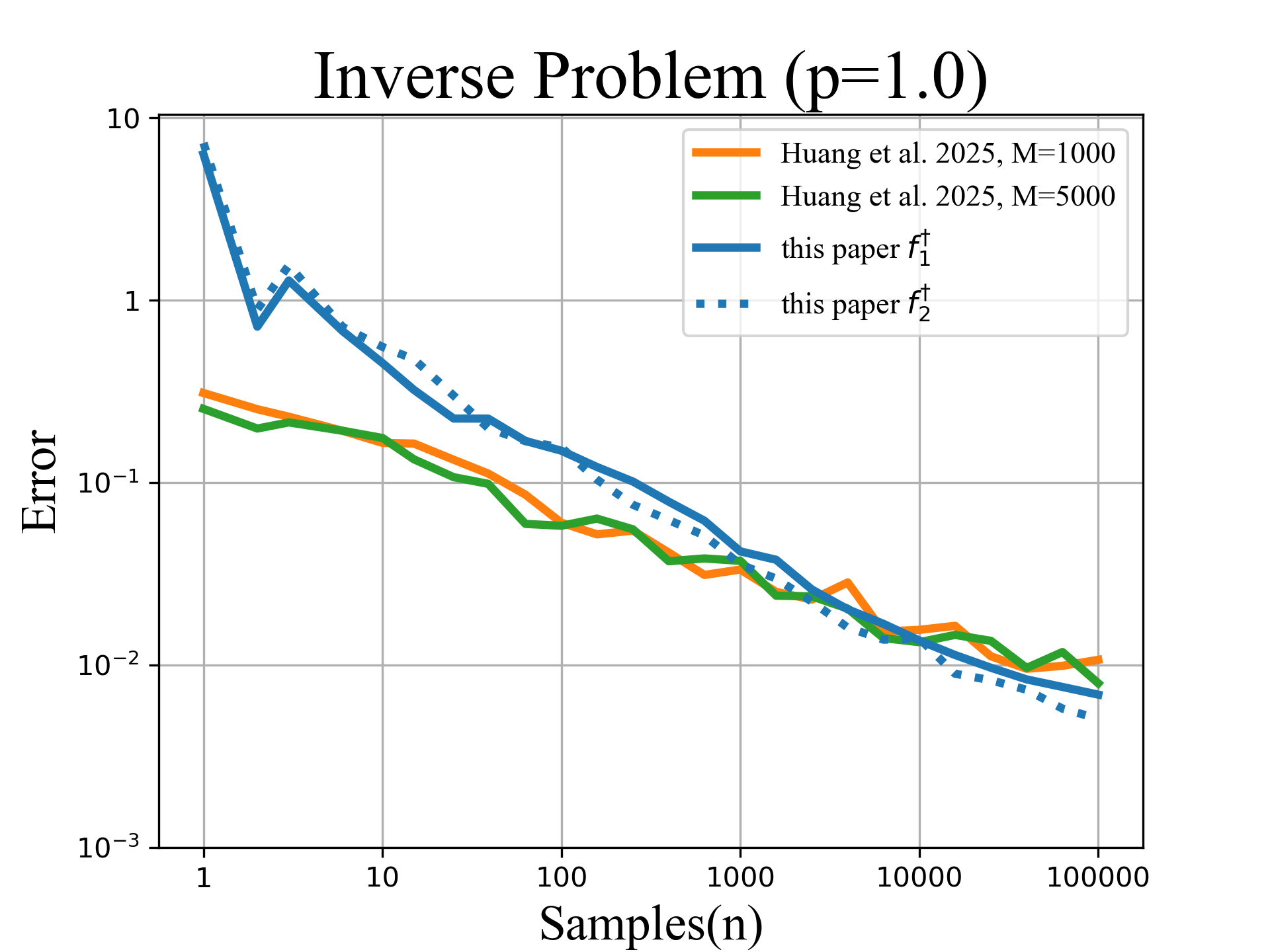}
}
\caption{The error shown in the figure is defined as
$Error=\EE\left[|(Af)(X)-(Af^\dag)(X)|^p\right]$.}
\label{fig:3}
\end{figure}
\section{Proofs}\label{sec:proof}

We prove \autoref{theorem:main}, \autoref{theorem:main1}, and \autoref{theorem:main2} below. The derivation of \autoref{lemma:main1} is given immediately after its statement, while the auxiliary lemmas and their proofs are collected in \autoref{sec:appendix}. The following lemma retains the core assumptions of \autoref{theorem:main} and \autoref{theorem:main2}, without imposing regularity assumptions on $f_\rho$ or prescribing the step size and $L_n$.
\begin{lemma}\label{lemma:5.1}
Let $1<p<\infty$, and let $\{\phi_j\}_{j\ge1}$ be a Schauder basis of $\LL_{\rho_X}^p(\Omega)$ satisfying $\Vert\phi_j\Vert_{\LL^p}\le1$. Assume that the loss function $\ell$ satisfies \autoref{assume:1} with exponent $q$, and that $s>\frac12-\frac{1}{2p}$ and $\mathbb{E}_\rho[|Y|^p]<\infty$. Write $p'=p/(p-1)$. For $2\le p<\infty$, suppose that there exists $t\in[p,\infty)$ such that $\sum_{j=1}^{\infty}j^{-2sp'}\Vert\phi_j\Vert_{\LL^t}^{p'}=:M^2<\infty$, and assume that $1\le q\le p-\frac{p}{t}$. For $1<p<2$, suppose that there exists $t\in[2,\infty)$ such that $\sum_{j=1}^{\infty}j^{-4s}\Vert\phi_j\Vert_{\LL^t}^{2}=:M^2<\infty$, and assume that $1\le q\le\min\{\frac{p}{2}-\frac{p}{t}+1,p\}$. Set $p_1=\max\{p,2\}$ and $p_1'=\frac{p_1}{p_1-1}=\min\{\frac{p}{p-1},2\}$. Then, for any $v\in\BB_{L_n}$, we have
\begin{equation}\label{eq:lemma5.1}
\begin{aligned}
\eta_n\left(\mathcal{E}(f_{n-1})-\mathcal{E}(v)\right)\leq &\left(1+C'\eta_n^{p_1'}\right)D_{s,p}(f_{n-1},v)-\EE\left[D_{s,p}(f_{n},v)|\FF_{n-1}\right]\\
&+C''\left( 2+\left(\EE[| Y_n|^p]\right)^{\frac{1}{\gamma'}}+\Vert v\Vert_{\LL^p}^{\frac{p}{\gamma'}} \right)\eta_n^{p_1'}.
\end{aligned}
\end{equation}
Here, $C'$ and $C''$ are constants independent of $n$ and $v$, given by
$$ C' = \frac{1}{p_1'}\left(\frac{C_{p,s}}{2}\right)^{-\frac{p_1'}{p_1}}(2^qC_\ell)^{p_1'}2^{\frac{3p}{\gamma'}}M^2C_1^{p_1}\left(\frac{C_{p,s}}{2}\right)^{-1},\quad\ C''=\frac{1}{p_1'}\left(\frac{C_{p,s}}{2}\right)^{-\frac{p_1'}{p_1}}(2^qC_\ell)^{p_1'}2^{\frac{3p}{\gamma'}}M^2.$$
For $q>1$, if $2\leq p<\infty$, then $\gamma'=\frac{p-1}{q-1}$, whereas if $1<p<2$, then $\gamma'=\frac{p}{2(q-1)}$. If $q=1$, we define $\frac{1}{\gamma'}=0$. The constant $C_1 = \left(\frac{2sp'}{2sp'-1}\right)^{1/p'}$ and $C_{p,s}$ denotes the $\max\{p,2\}$-convex constant of $\R_{s,p}(f)=\frac{1}{\max\{p,2\}}\Vert f\Vert_{\BB^s_{p,p}}^{\max\{p,2\}}$.

\end{lemma}
\begin{proof}
As shown in \cite[Lemma 2.62]{schuster2012regularization}, the Bregman distance $D_{s,p}$ satisfies the following three-point identity for all
$v\in\BB_{L_n}$:
\begin{equation*}
D_{s,p}(f_{n-1},v)-D_{s,p}(f_{n},v)-D_{s,p}(f_{n-1},f_n)=\langle \partial\R_{s,p}(f_n)-\partial\R_{s,p}(f_{n-1}),v-f_n\rangle.
\end{equation*}
The first-order optimality condition for the stochastic mirror descent subproblem (see \cite[Theorem 7.12-3]{ciarlet2013linear}) implies
\begin{equation*}
\begin{aligned}
\left\langle \partial\R_{s,p}(f_n)-\partial\R_{s,p}(f_{n-1})+\eta_n\widehat{\partial\mathcal{E}}(f_{n-1})\big|_{\BB_{L_n}^*}, v-f_n\right\rangle\geq 0,\quad\forall v\in\BB_{L_n}.
\end{aligned}
\end{equation*}
Combining this relation with the three-point identity above and the $\max\{p,2\}$-convexity of $D_{s,p}$ in \autoref{section:Preliminary}, we arrive at 
\begin{equation}\label{eq:sec5.1}
\begin{aligned}
&D_{s,p}(f_{n-1},v)-D_{s,p}(f_{n},v)-\frac{C_{p,s}}{2}\Vert f_{n-1}-f_n\Vert_{\BB_{p,p}^s}^{\max\{p,2\}}\\
\geq &D_{s,p}(f_{n-1},v)-D_{s,p}(f_{n},v)-D_{s,p}(f_{n-1},f_n) \geq\left\langle \eta_n\widehat{\partial\mathcal{E}}(f_{n-1})\big|_{\BB_{L_n}^*}, f_n-v\right\rangle\\
= &\eta_n\left\langle \widehat{\partial\mathcal{E}}(f_{n-1})\big|_{\BB_{L_n}^*}, f_n-f_{n-1}\right\rangle+
\eta_n\left\langle \widehat{\partial\mathcal{E}}(f_{n-1})\big|_{\BB_{L_n}^*}, f_{n-1}-v\right\rangle.
\end{aligned}
\end{equation}
By the convexity of $\ell(u,y)$, we have $\ell(u,y)-\ell(v,y)\leq g_\ell(u,y)(u-v)$ for any $u,v\in\mathbb{R},\ y\in\Y$. Applying this inequality to the second term on the right-hand side of \eqref{eq:sec5.1}, we obtain 
\begin{equation}\label{eq:sec5.2}
\begin{aligned}
&\left\langle \widehat{\partial\mathcal{E}}(f_{n-1})\big|_{\BB_{L_n}^*}, f_{n-1}-v\right\rangle = g_\ell(f_{n-1}(X_n),Y_n)\left(f_{n-1}(X_n)-v(X_n)\right)\\
\geq&\ell\left(f_{n-1}(X_n),Y_n\right)-\ell\left(v(X_n),Y_n\right)\\
\Rightarrow\quad\quad&\EE\left[\left\langle \widehat{\partial\mathcal{E}}(f_{n-1})\Big|_{\BB_{L_n}^*}, f_{n-1}-v\right\rangle\big|\FF_{n-1}\right]\geq \ee(f_{n-1})-\ee(v).
\end{aligned}
\end{equation}
We now turn to bounding the first term on the right-hand side of \eqref{eq:sec5.1}. In what follows, we restrict our attention to the case $2\leq p<\infty$ and $q>1$. The case $1<p<2$ and the case $q=1$ are treated separately in the appendix (see \autoref{lemma:A.5} and \autoref{lemma:A.8}), respectively.
\begin{align*}
&\eta_n\left|\left\langle \widehat{\partial\mathcal{E}}(f_{n-1})\big|_{\BB_{L_n}^*}, f_n-f_{n-1}\right\rangle\right|\leq \eta_n\left|g_\ell(f_{n-1}(X_n),Y_n)\right|\cdot\left|\sum_{j=1}^{L_n}\langle \phi_j^*,f_n-f_{n-1}\rangle\phi_j(X_n)\right|\\
\overset{\text{(i)}}{\leq}& \eta_n\left|g_\ell(f_{n-1}(X_n),Y_n)\right|\cdot\left(\sum_{j=1}^{L_n}j^{2sp}\left|\langle \phi_j^*,f_n-f_{n-1}\rangle\right|^p\right)^{1/p}\left( \sum_{j=1}^{L_n}j^{-2sp'}|\phi_j(X_n)|^{p'}\right)^{1/p'}\\
\overset{\text{(ii)}}{\leq}& \eta_n\left|g_\ell(f_{n-1}(X_n),Y_n)\right|\cdot\Vert f_n-f_{n-1}\Vert_{\BB^s_{p,p}}\left( \sum_{j=1}^{L_n}j^{-2sp'}|\phi_j(X_n)|^{p'}\right)^{1/p'}\\
\overset{\text{(iii)}}{\leq}& \frac{1}{p'}\left(\frac{C_{p,s}}{2}\right)^{-p'/p}\eta_n^{p'}\left|g_\ell(f_{n-1}(X_n),Y_n)\right|^{p'}\sum_{j=1}^{L_n}j^{-2sp'}|\phi_j(X_n)|^{p'}+
\frac{1}{p}\frac{C_{p,s}}{2}\Vert f_n-f_{n-1}\Vert_{\BB^s_{p,p}}^p.
 \end{align*}
In (i), we employ H\"{o}lder's inequality. The estimate in (ii) relies on the fact that $f_n, f_{n-1} \in \BB_{L_n}$ together with the definition of the norm in $\BB_{p,p}^s$. Step (iii) follows from an application of Young's inequality, where $C_{p,s}$ is the $p$-convexity constant associated with $\R_{s,p}(f)=\frac{1}{\max\{p,2\}}\Vert f\Vert_{\BB^s_{p,p}}^{\max\{p,2\}}$. Taking the conditional expectation $\EE\left[\cdot\mid\FF_{n-1}\right]$ on both sides, where $\FF_{n-1}$ is the $\sigma$-algebra generated by $\{(X_i,Y_i)\}_{i=1}^{n-1}$, we obtain
\begin{equation}\label{eq:sec5.3}
\begin{aligned}
&\eta_n\EE\left[\left|\left\langle \widehat{\partial\mathcal{E}}(f_{n-1})\big|_{\BB_{L_n}^*}, f_n-f_{n-1}\right\rangle\right|\big|\FF_{n-1}\right]\leq \frac{1}{p}\frac{C_{p,s}}{2}\EE\left[\Vert f_n-f_{n-1}\Vert_{\BB^s_{p,p}}^p\big|\FF_{n-1}\right]\\
& +\frac{1}{p'}\left(\frac{C_{p,s}}{2}\right)^{-p'/p}\eta_n^{p'}\sum_{j=1}^{L_n}j^{-2sp'}\EE\left[\left|g_\ell(f_{n-1}(X_n),Y_n)\right|^{p'}|\phi_j(X_n)|^{p'}\big|\FF_{n-1}\right].
\end{aligned}
\end{equation}
We next bound the second term on the right-hand side of the above inequality. Let $\gamma = \frac{p-1}{p-q}$ and $\gamma' = \frac{p-1}{q-1}$ be the corresponding H\"{o}lder conjugates. Using H\"older's inequality together with the subgradient bound in \autoref{lemma:add A.1}, we obtain
\begin{align*}
&\sum_{j=1}^{L_n}j^{-2sp'}\EE\left[\left|g_\ell(f_{n-1}(X_n),Y_n)\right|^{p'}|\phi_j(X_n)|^{p'}\big|\FF_{n-1}\right]\\
\leq&\left(2^qC_\ell\right)^{p'}\sum_{j=1}^{L_n}j^{-2sp'}\EE\left[\left(1+|f_{n-1}(X_n)|+|Y_n|\right)^{(q-1)p'}|\phi_j(X_n)|^{p'}\big|\FF_{n-1}\right]\\
\leq&\left(2^qC_\ell\right)^{p'}\left(\EE\left[\left(1+|f_{n-1}(X_n)|+|Y_n|\right)^{(q-1)p'\gamma'}\big|\FF_{n-1}\right]\right)^{\frac1 {\gamma'}}
\sum_{j=1}^{L_n}j^{-2sp'}\left(\EE\left[|\phi_j(X_n)|^{p'\gamma}\right]\right)^{\frac1\gamma}\\
\overset{\text{(i)}}{\leq}&2^{\frac{3p}{\gamma'}}\left(2^qC_\ell\right)^{p'}\left( 2+\Vert f_{n-1}-v\Vert_{\LL^p}^{p}+\left(\EE[| Y_n|^p]\right)^{\frac{1}{\gamma'}}+\Vert v\Vert_{\LL^p}^{\frac{p}{\gamma'}} \right)\sum_{j=1}^{L_n}j^{-2sp'}\Vert \phi_j\Vert_{\LL^t}^{p'}\\
\overset{\text{(ii)}}{\leq}&2^{\frac {3p}{\gamma'}}\left(2^qC_\ell\right)^{p'}M^2\left( C_1^p\Vert f_{n-1}-v\Vert_{\BB^s_{p,p}}^{p}+2+\left(\EE[| Y_n|^p]\right)^{\frac{1}{\gamma'}}+\Vert v\Vert_{\LL^p}^{\frac{p}{\gamma'}} \right),
\end{align*}
where (i) follows from the elementary inequality, valid for all $a,b\geq 0$, $\gamma'>1$, and $p>1$,
$$((a+b)^p)^{1/\gamma'}\leq (2^p(a^p+b^p))^{1/\gamma'}=2^{\frac{p}{\gamma'}} (a^p+b^p)^{1/\gamma'}\leq 2^{\frac{p}{\gamma'}} \left(a^{\frac{p}{\gamma'}}+b^{\frac{p}{\gamma'}}
\right),\ \forall a,b\geq0,\ \gamma'>1,\ p>1$$
together with $a^{1/\gamma'} \le 1 + a$ for $a \ge 0,\gamma'>1$. In (ii), we use \autoref{lemma:A.2} and define $C_1 := \left(\frac{2sp'}{2sp'-1}\right)^{1/p'}$, together with  the assumption $\sum_{j=1}^\infty j^{-2sp'}\Vert\phi_j\Vert_{\LL^t}^{p'}=M^2<\infty$ for $t\geq\frac{p}{p-q}$. Substituting the above estimate into \eqref{eq:sec5.3}, we obtain
\begin{equation*}
\begin{aligned}
&\eta_n\EE\left[\left|\left\langle \widehat{\partial\mathcal{E}}(f_{n-1})\big|_{\BB_{L_n}^*}, f_n-f_{n-1}\right\rangle\right|\big|\FF_{n-1}\right]\\
\leq&  \frac{1}{p}\frac{C_{p,s}}{2}\EE\left[\Vert f_n-f_{n-1}\Vert_{\BB^s_{p,p}}^p\big|\FF_{n-1}\right]+\left(\frac{1}{p'}\left(\frac{C_{p,s}}{2}\right)^{-p'/p}2^{\frac {3p}{\gamma'}}\left(2^qC_\ell\right)^{p'}M^2C_1^p\right)\eta_n^{p'}\Vert f_{n-1}-v\Vert_{\BB^s_{p,p}}^{p}\\
&+\left(\frac{1}{p'}\left(\frac{C_{p,s}}{2}\right)^{-p'/p}2^{\frac {3p}{\gamma'}}\left(2^qC_\ell\right)^{p'}M^2\left( 2+\left(\EE[| Y_n|^p]\right)^{\frac{1}{\gamma'}}+\Vert v\Vert_{\LL^p}^{\frac{p}{\gamma'}} \right)\right)\eta_n^{p'}.
\end{aligned}
\end{equation*}
Finally, taking the conditional expectation $\EE\left[\cdot\big|\FF_{n-1}\right]$ in \eqref{eq:sec5.1}, applying the above estimate together with \autoref{lemma:A.5}, \autoref{lemma:A.8}, and \eqref{eq:sec5.2}, and defining $p_1=\max\{p,2\}$, we obtain
\begin{align*}
&\eta_n\left(\ee(f_{n-1})-\ee(v)\right)\\
\leq &D_{s,p}(f_{n-1},v)-\EE\left[D_{s,p}(f_{n},v)|\FF_{n-1}\right]-\frac{C_{p,s}}{2}\EE\left[\Vert f_n-f_{n-1}\Vert_{\BB^s_{p,p}}^{p_1}\big|\FF_{n-1}\right]\\
&+\eta_n\EE\left[\left|\left\langle \widehat{\partial\mathcal{E}}(f_{n-1})\big|_{\BB_{L_n}^*}, f_n-f_{n-1}\right\rangle\right|\big|\FF_{n-1}\right]\\
\leq& D_{s,p}(f_{n-1},v)-\EE\left[D_{s,p}(f_{n},v)|\FF_{n-1}\right]+\left(\frac{\left(2^qC_\ell\right)^{p_1'}}{p_1'}\left(\frac{C_{p,s}}{2}\right)^{-\frac{p_1'}{p_1}}2^{\frac{3p}{\gamma'}}M^2C_1^{p_1}\right)\eta_n^{p_1'}\Vert f_{n-1}-v\Vert_{\BB^s_{p,p}}^{p_1}\\
&+\left(\frac{\left(2^qC_\ell\right)^{p_1'}}{p_1'}\left(\frac{C_{p,s}}{2}\right)^{-\frac{p_1'}{p_1}}2^{\frac{3p}{\gamma'}}M^2\left( 2+\left(\EE[| Y_n|^p]\right)^{\frac{1}{\gamma'}}+\Vert v\Vert_{\LL^p}^{\frac{p}{\gamma'}} \right)\right)\eta_n^{p_1'}.
\end{align*}
Here, $p_1'=\frac{p_1}{p_1-1}=\min\left\{\frac{p}{p-1},2\right\}$. Moreover, if $2\leq p<\infty$ and $q>1$, then $\gamma'=\frac{p-1}{q-1}$, whereas if $1<p<2$ and $q>1$, then $\gamma'=\frac{p}{2(q-1)}$. If $q=1$, we define $\frac{1}{\gamma'}=0$. Finally, invoking the inequality $D_{s,p}(f,g)\geq \frac{C_{p,s}}{2}\Vert f-g\Vert_{\BB^{s}_{p,p}}^{\max\{p,2\}}$, we obtain the desired conclusion of the lemma.
\end{proof}

\autoref{lemma:5.1} serves as a fundamental ingredient in the proofs of \autoref{theorem:main} and \autoref{theorem:main2}. We next bound the term $\EE\left[D_{s,p}(f_{n},v)\right]$. Write $f_\rho=\sum_{j=1}^\infty \beta_j^\rho\phi_j$, and take $v$ to be the restriction of $f_\rho$ to $\BB_{L_n}$, namely, $ f_\rho^{L_n}:=S_{L_n}(f_\rho) =\sum_{j=1}^{L_n} \beta_j^\rho\phi_j$ (see the definition and properties of $S_n$ in \autoref{section:Preliminary}). By the uniform boundedness of $\{S_n\}_{n\geq1}$, we have $\sup_{L_n\geq1}\Vert S_{L_n}(f_\rho)\Vert_{\LL^p}<\infty$. Combining this with the assumption $\mathbb{E}_\rho[|Y|^p]<\infty$, we obtain
$$C_3 :=  C''\left( 2+\left(\EE[| Y|^p]\right)^{\frac{1}{\gamma'}}+\sup_{n\in\NN}\Vert f^{L_n}_\rho\Vert_{\LL^p}^{\frac{p}{\gamma'}} \right)<\infty.$$
Since $f^{L_n}_\rho-f^{L_{n-1}}_\rho\in\text{span}\{ \phi_j\}_{L_{n-1}+1\leq j\leq L_n}$ and $\partial\R_{s,p}(f_{n-1})\in\BB_{L_{n-1}}^*$, it follows that
\begin{equation*}
\begin{aligned}
&D_{s,p}(f_{n-1},f_\rho^{L_n}) = \R_{s,p}(f_\rho^{L_n})-\R_{s,p}(f_{n-1})+\langle \partial\R_{s,p}(f_{n-1}), f_{n-1}-f_\rho^{L_n}\rangle \\
=& D_{s,p}(f_{n-1},f_\rho^{L_{n-1}})+\left(\R_{s,p}(f_\rho^{L_n}) - \R_{s,p}(f_\rho^{L_{n-1}})\right).
\end{aligned}
\end{equation*}
Combining the above estimates with \eqref{eq:lemma5.1} yields
\begin{align}\label{eq:sec5.4}
&\EE\left[D_{s,p}(f_{n},f_\rho^{L_n})\right]\leq \left(1+C'\eta_n^{p_1'}\right)\EE\left[D_{s,p}(f_{n-1},f_\rho^{L_{n}})\right]+\eta_n\EE\left[\mathcal{E}(f_\rho^{L_{n}}) - \mathcal{E}(f_{n-1})\right]+C_3\eta_n^{p_1'}\notag\\
\leq & \left(1+C'\eta_n^{p_1'}\right)\EE\left[D_{s,p}(f_{n-1},f_\rho^{L_{n-1}})\right]+\eta_n\EE\left[\mathcal{E}(f_\rho^{L_{n}}) - \mathcal{E}(f_{n-1})\right]+C_3\eta_n^{p_1'}\notag\\
& +\left(1+C'\eta_n^{p_1'}\right)\left(\R_{s,p}(f_\rho^{L_n}) - \R_{s,p}(f_\rho^{L_{n-1}})\right)\notag\\
\leq & \prod_{i=1}^n\left(1+C'\eta_i^{p_1'}\right)D_{s,p}(f_{0},f_\rho^{L_0})+\sum_{i=1}^n\left(\prod_{j=i+1}^n\left(1+C'\eta_j^{p_1'}\right)\right)\eta_i\EE\left[\mathcal{E}(f_\rho^{L_{i}}) - \mathcal{E}(f_{i-1})\right]\\
& +C_3\sum_{i=1}^n\left(\prod_{j=i+1}^n\left(1+C'\eta_j^{p_1'}\right)\right)\eta_i^{p_1'}
+\sum_{i=1}^n\left(\prod_{j=i}^n\left(1+C'\eta_j^{p_1'}\right)\right)\left(\R_{s,p}(f_\rho^{L_i}) - \R_{s,p}(f_\rho^{L_{i-1}})\right)\notag\\
\overset{\text{(i)}}{\leq}& \prod_{i=1}^n\left(1+C'\eta_i^{p_1'}\right)\left(D_{s,p}(f_{0},f_\rho^{L_0})+\sum_{i=1}^n\eta_i\left(\mathcal{E}(f_\rho^{L_{i}}) - \mathcal{E}(f_\rho)\right)+C_3\sum_{i=1}^n\eta_i^{p_1'}+\R_{s,p}(f_\rho^{L_n})\right).\notag
\end{align}
In (i), we use the inequality $\mathcal{E}(f_\rho^{L_{i}}) - \mathcal{E}(f_{i-1})=\left(\mathcal{E}(f_\rho^{L_{i}})-\ee(f_\rho)\right)+ (\ee(f_\rho)- \mathcal{E}(f_{i-1}))\leq\mathcal{E}(f_\rho^{L_{i}}) - \mathcal{E}(f_\rho)$.
With the estimates of $\mathcal{E}(f_{n-1})-\mathcal{E}(f^{L_n}_\rho)$ in \eqref{eq:lemma5.1} and of $\EE\left[D_{s,p}(f_{n},f_\rho^{L_n})\right]$ in \eqref{eq:sec5.4} at hand, we are now ready to prove \autoref{theorem:main} and \autoref{theorem:main2}.
\subsection{Proof of \autoref{theorem:main}}\label{proof of theorem 1}

We first establish, via the following lemma, that $\EE\left[D_{s,p}(f_{n},f_\rho^{L_n})\right]$ is uniformly bounded.

\begin{lemma}\label{lemma:5.2}
Under the assumptions of \autoref{theorem:main}, $\EE\left[D_{s,p}(f_{n},f_\rho^{L_n})\right]$ is uniformly bounded, namely,
$$\sup_{n\in\NN}\EE\left[D_{s,p}(f_{n},f_\rho^{L_n})\right] :=C_9<\infty.$$
\end{lemma}
\begin{proof}
Since $\eta_n=\eta_0n^{-1/p_1'}\left(\ln(n+1)\right)^{-2/p_1'}$, it follows that
$$\sum_{i=1}^\infty \eta_i^{p_1'}=\eta_0^{p_1'}\sum_{i=1}^\infty i^{-1}\left(\ln(i+1)\right)^{-2}=: C_5<\infty.$$
We also have
\begin{equation*}
\begin{aligned}
&\prod_{i=1}^\infty\left(1+C'\eta_i^{p_1'}\right)=\exp\left(\sum_{i=1}^\infty\ln\left(1+C'\eta_i^{p_1'}\right)\right)\leq\exp\left(C'\sum_{i=1}^\infty\eta_i^{p_1'}\right)\\
= &\exp\left(C'\eta_0^{p_1'}\sum_{i=1}^\infty i^{-1}\left(\ln(i+1)\right)^{-2}\right)=\exp\left(C'C_5\right) =: C_6<\infty.
\end{aligned}
\end{equation*}
By using $\theta>\frac{p}{p_1\nu\left(2sp-(p-1)\right)}$, we obtain $\theta\nu\left(2s-\frac{1}{p'}\right)+\frac{1}{p_1'}>1$. Hence, there exists $\delta>0$ such that $1+\delta = \theta\nu\left(2s-\frac{1}{p'}\right)+\frac{1}{p_1'}$. Applying \autoref{lemma:A.6}, we conclude that
\begin{equation}\label{eq:sec5.5}
\begin{aligned}
0\leq&\eta_n\left(\ee(f_\rho^{L_n})-\ee(f_\rho)\right)\leq C_4\eta_0n^{-\frac{1}{p_1'}}\left(\ln(n+1)\right)^{-\frac{2}{p_1'}} n^{\theta\nu\left(-2s+\frac{1}{p'}\right)}\\
\leq& C_4\eta_0\left(\ln(2)\right)^{-2}n^{-1-\delta}=:C_7 n^{-1-\delta}.
\end{aligned}
\end{equation}
Substituting the above three estimates into \eqref{eq:sec5.4}, and invoking the inequality $\R_{s,p}(f_\rho)\geq\R_{s,p}(f_\rho^{L_n})$ and $\mathcal{E}(f_\rho^{L_{i}}) - \mathcal{E}(f_\rho)\geq0$, we obtain
\begin{equation*}
\begin{aligned}
&\sup_{n\in\NN}\EE\left[D_{s,p}(f_{n},f_\rho^{L_n})\right]\\
\leq & \prod_{i=1}^\infty\left(1+C'\eta_i^{p_1'}\right)\left(D_{s,p}(f_{0},f_\rho^{L_0})+\sum_{i=1}^\infty\eta_i\left(\mathcal{E}(f_\rho^{L_{i}}) - \mathcal{E}(f_\rho)\right)+C_3\sum_{i=1}^\infty\eta_i^{p_1'}+\R_{s,p}(f_\rho)\right)\\
\leq& C_6\left(D_{s,p}(f_{0},f_\rho^{L_0})+C_7\sum_{i=1}^\infty i^{-1-\delta}+C_3C_5+\R_{s,p}(f_\rho)\right)<\infty.
\end{aligned}
\end{equation*}
The proof is finished.
\end{proof}
Finally, we turn to estimating $\EE\left[\mathcal{E}(\bar{f}^s_{\alpha n})-\mathcal{E}(f_\rho)\right]$. This requires us to bound the following term.
\begin{equation*}
\begin{aligned}
&\frac{1}{\alpha n}\eta_n^{-1}\sum_{i=(1-\alpha)n+1}^{n}\eta_i\left(\mathcal{E}(f_\rho^{L_i})-\mathcal{E}(f_\rho)\right)\leq\frac{1}{\alpha n}\eta_0^{-1}n^{\frac{1}{p_1'}}\left(\ln(n+1)\right)^{\frac{2}{p_1'}}\sum_{i=(1-\alpha)n+1}^{n}C_7 i^{-1-\delta}\\
\leq& C_7\eta_0^{-1}n^{\frac{1}{p_1'}}\left(\ln(n+1)\right)^{\frac{2}{p_1'}}\left((1-\alpha)n\right)^{-1-\delta}= C_7\eta_0^{-1}\left(1-\alpha\right)^{-1-\delta}n^{\frac{1}{p_1'}-1}\left(n^{-\delta}\left(\ln(n+1)\right)^{\frac{2}{p_1'}}\right)\\
\leq& C_8n^{-1+\frac{1}{p_1'}}=C_8n^{-\frac{1}{p_1}},
\end{aligned}
\end{equation*}
where $C_8$ is a constant independent of $n$. Combining the above estimate with \eqref{eq:sec5.4}, and using the fact that $\eta_n$ is monotonically decreasing so that $\eta_i\geq\eta_n$, we obtain
\begin{align*}
&\EE\left[\mathcal{E}(\bar{f}^s_{\alpha n})-\mathcal{E}(f_\rho)\right]\leq \frac{1}{\alpha n}\eta_n^{-1}\sum_{i=(1-\alpha)n+1}^{n}\eta_i\EE\left[\mathcal{E}(f_{i-1})-\mathcal{E}(f_\rho)\right]\\
\leq & \frac{1}{\alpha n}\eta_n^{-1}\sum_{i=(1-\alpha)n+1}^{n}\eta_i\EE\left[\mathcal{E}(f_{i-1})-\mathcal{E}(f_\rho^{L_i})\right]+\frac{1}{\alpha n}\eta_n^{-1}\sum_{i=(1-\alpha)n+1}^{n}\eta_i\left(\mathcal{E}(f_\rho^{L_i})-\mathcal{E}(f_\rho)\right)\\
\leq&\frac{1}{\alpha n}\eta_n^{-1}\sum_{i=(1-\alpha)n+1}^{n}\Bigg( \left(\EE\left[D_{s,p}(f_{i-1},f_\rho^{L_{i-1}})\right]-\EE\left[D_{s,p}(f_{i},f_\rho^{L_i})\right]\right)+C_3\eta_i^{p_1'}\\
&+C'\eta_i^{p_1'}\EE\left[D_{s,p}(f_{i-1},f_\rho^{L_{i-1}})\right]+\left(1+C'\eta_i^{p_1'}\right)\left(\R_{s,p}(f_\rho^{L_i}) - \R_{s,p}(f_\rho^{L_{i-1}})\right)\Bigg)+C_8n^{-\frac{1}{p_1}}\\
\leq&\frac{1}{\alpha n}\eta_n^{-1}\Bigg( \EE\left[D_{s,p}(f_{(1-\alpha)n },f_\rho^{L_{(1-\alpha)n }})\right]+C'\sum_{i=(1-\alpha)n+1}^{n}\eta_i^{p_1'}\EE\left[D_{s,p}(f_{i-1},f_\rho^{L_{i-1}})\right]\\
&+C_3\sum_{i=(1-\alpha)n+1}^{n}\eta_i^{p_1'}+\sum_{i=(1-\alpha)n+1}^{n}\left(1+C'\eta_i^{p_1'}\right)\left(\R_{s,p}(f_\rho^{L_i}) - \R_{s,p}(f_\rho^{L_{i-1}})\right)\Bigg)+C_8n^{-\frac{1}{p_1}}\\
\overset{\text{(i)}}{\leq}&\frac{\eta_n^{-1}}{\alpha n}\left( C_9+\left(C'C_9+C_3\right)\sum_{i=1}^\infty\eta_i^{p_1'}+\sum_{i=1}^\infty\left(1+C'\eta_1^{p_1'}\right)\left(\R_{s,p}(f_\rho^{L_i}) - \R_{s,p}(f_\rho^{L_{i-1}})\right)\right)+C_8n^{-\frac{1}{p_1}}\\
\overset{\text{(ii)}}{\leq}&\frac{1}{\alpha \eta_0}\left( C_9+\left(C'C_9+C_3\right)C_5+\left(1+C'\eta_1^{p_1'}\right)\R_{s,p}(f_\rho) \right)n^{-\frac{1}{p_1}}\left(\ln(n+1)\right)^{2\left(1-\frac{1}{p_1}\right)}+C_8n^{-\frac{1}{p_1}}\\
\lesssim&n^{-\frac{1}{p_1}}\left(\ln(n+1)\right)^{2\left(1-\frac{1}{p_1}\right)}.
\end{align*}
In (i), we use \autoref{lemma:5.2}. Step (ii) follows from $\R_{s,p}(f_\rho^{L_i})\uparrow\R_{s,p}(f_\rho)$. We complete the proof.

\subsection{Proof of \autoref{theorem:main2}}

Here we denote $\mathcal{E}(f_\rho^{L_i})-\mathcal{E}(f_\rho)=\xi_i$. Analogously to the proof of \autoref{theorem:main}, we first derive an upper bound for $\EE\left[D_{s,p}(f_{n},f_\rho^{L_n})\right]$.

\begin{lemma}\label{lemma:5.3}
Under the assumptions of \autoref{theorem:main2}, there exist constants $P_1,P_2>0$, independent of $n$, such that
$$\EE\left[D_{s,p}(f_{n},f_\rho^{L_n})\right] \leq P_1n^{2\theta sp_1(1-r)}+P_2\sum_{i=1}^n\eta_i\xi_i.$$
\end{lemma}
\begin{proof}
Based on $\eta_n=\eta_0n^{-\tau}$ with $\frac{1}{p_1'}=1-\frac{1}{p_1}<\tau<1$, and arguing similarly to the proof of \autoref{lemma:5.2}, we obtain
\begin{equation*}
\sum_{i=1}^\infty \eta_i^{p_1'}=:P_3<\infty, \quad \prod_{i=1}^\infty\left(1+ C'\eta_i^{p_1'}\right)=:P_4<\infty.
\end{equation*}
Then, we also need to bound the term $\R_{s,p}(f_\rho^{L_n})$,
\begin{equation*}
\begin{aligned}
\R_{s,p}(f_\rho^{L_n})&=\frac{1}{p_1}\left(\sum_{j=1}^{L_n}j^{2sp}|\beta_j^\rho|^p\right)^{p_1/p}\leq\frac{1}{p_1}\left(\sum_{j=1}^{L_n}j^{2spr}|\beta_j^\rho|^p\right)^{p_1/p}
L_n^{2sp_1(1-r)}\\
&\overset{\text{(i)}}{\leq} \frac{2^{2 sp_1(1-r)}}{p_1}\left(\sum_{j=1}^{\infty}j^{2spr}|\beta_j^\rho|^p\right)^{p_1/p}n^{2\theta sp_1(1-r)}=:P_5 n^{2\theta sp_1(1-r)},
\end{aligned}
\end{equation*}
where (i) follows from $f_\rho\in\BB^{s,r}_{p,p}$. Combining the above two estimates with \eqref{eq:sec5.4}, one has
\begin{equation*}
\begin{aligned}
&\EE\left[D_{s,p}(f_{n},f_\rho^{L_n})\right]\\
\leq & \prod_{i=1}^\infty\left(1+C'\eta_i^{p_1'}\right)\left(D_{s,p}(f_{0},f_\rho^{L_0})+\sum_{i=1}^n\eta_i\left(\mathcal{E}(f_\rho^{L_{i}}) - \mathcal{E}(f_\rho)\right)+C_3\sum_{i=1}^\infty\eta_i^{p_1'}+\R_{s,p}(f_\rho^{L_n})\right)\\
\leq & P_4\left(D_{s,p}(f_{0},f_\rho^{L_0})+\sum_{i=1}^n\eta_i\xi_i+C_3P_3+P_5n^{2\theta sp_1(1-r)}\right)\leq P_1n^{2\theta sp_1(1-r)}+P_2\sum_{i=1}^n\eta_i\xi_i.
\end{aligned}
\end{equation*}
The proof is completed.
\end{proof}
We now turn to estimating $\EE\left[\mathcal{E}(\bar{f}^s_{\alpha n})-\mathcal{E}(f_\rho)\right]$. Similar to the proof of \autoref{theorem:main}, we obtain
\begin{equation*}
\begin{aligned}
&\EE\left[\mathcal{E}(\bar{f}^s_{\alpha n})-\mathcal{E}(f_\rho)\right]\\
\leq&\frac{\eta_n^{-1}}{\alpha n}\Bigg( \EE\left[D_{s,p}(f_{(1-\alpha)n },f_\rho^{L_{(1-\alpha)n }})\right]+C'\sum_{i=(1-\alpha)n+1}^{n}\eta_i^{p_1'}\EE\left[D_{s,p}(f_{i-1},f_\rho^{L_{i-1}})\right]\\
&+C_3\sum_{i=(1-\alpha)n+1}^{n}\eta_i^{p_1'}+\sum_{i=(1-\alpha)n+1}^{n}\left(1+C'\eta_i^{p_1'}\right)\left(\R_{s,p}(f_\rho^{L_i}) - \R_{s,p}(f_\rho^{L_{i-1}})\right)\Bigg)+\frac{1}{\alpha n\eta_n}\sum_{i=1}^n\eta_i\xi_i\\
\overset{\text{(i)}}{\leq}&\frac{\eta_n^{-1}}{\alpha n}\left( \left(1+C'P_3\right)\left(P_1n^{2\theta sp_1(1-r)}+P_2\sum_{i=1}^n\eta_i\xi_i\right)+C_3P_3
+\left(1+C'\eta_1^{p_1'}\right)\R_{s,p}(f_\rho^{L_n})+\sum_{i=1}^n\eta_i\xi_i\right)\\
\leq &\frac{\eta_n^{-1}}{\alpha n}\left( \left(P_1\left(1+C'P_3\right)+P_5\left(1+C'\eta_1^{p_1'}\right)\right)n^{2\theta sp_1(1-r)}+\left(P_2\left(1+C'P_3\right)+1\right)\sum_{i=1}^n\eta_i\xi_i+C_3P_3\right)\\
\lesssim&n^{-1+\tau+2sp_1\theta(1-r)}+n^{-1+\tau}\sum_{i=1}^ni^{-\tau}\xi_i.
\end{aligned}
\end{equation*}
In (i), we use \autoref{lemma:5.3}. Since $f_\rho\in\mathcal{L}_{\rho_X}^p(\Omega)$ and $\{\phi_j\}$ is a Schauder basis of this space, it follows from the first part of the proof of \autoref{lemma:A.6} that
$$\ee(f_{\rho}^{L_i})-\ee(f_{\rho})\leq 2^qC_\ell\left(1+\Vert f_\rho\Vert_{\LL^p}+\sup_{l\geq1}\Vert S_{l}\Vert\cdot\Vert f_\rho\Vert_{\LL^p}+\left(\EE\left[|Y|^p\right]\right)^{\frac 1p}\right)^{q-1}\Vert f_\rho^{L_i}-f_\rho\Vert_{\LL^p}\to0.$$
Consequently, $\xi_i= \ee(f_{\rho}^{L_i})-\ee(f_{\rho})\to0$. Finally, we invoke a classical and easily verified result from real analysis to complete the proof. If $a_i\geq0$ and $a_i\to0$, then for any $0<\kappa<1$, one has $\lim_{n\to\infty}n^{-1+\kappa}\sum_{i=1}^n i^{-\kappa}a_i=0$. This implies that
\begin{equation*}
\begin{aligned}
\lim_{n\to\infty}\EE\left[\mathcal{E}(\bar{f}^s_{\alpha n})-\mathcal{E}(f_\rho)\right]=0.
\end{aligned}
\end{equation*}

\subsection{Proof of \autoref{theorem:main1}}

In this subsection, we prove the main result using an argument similar to the proof of \autoref{theorem:main}. Combining the three-point identity for the Bregman distance, the first-order optimality condition, and the strong convexity estimate for the functional $\R_{s,1}^{L_N}$ established in \autoref{lemma:A.7}, we obtain, for every $\alpha^\pm\in\A_L^{2L_N}$,
\begin{equation*}
\begin{aligned}
&D_{s,1}^{L_N}(\alpha^{\pm,(n-1)},\alpha^\pm)-D_{s,1}^{L_N}(\alpha^{\pm,(n)},\alpha^\pm)-\frac{1}{2L} \Vert \mathcal{T}^\sharp(\alpha^{\pm,(n-1)})-\mathcal{T}^\sharp(\alpha^{\pm,(n)})\Vert_{\BB^s_{1,1}}^2\\
\geq &\eta_n\left\langle \xi_n^\pm(\mathcal{T}^\sharp \alpha^{\pm,(n-1)}), \alpha^{\pm,(n)}-\alpha^{\pm}\right\rangle_{\ell_2}\\
\overset{\text{(i)}}{=}&\eta_n\left\langle \widehat{\partial\mathcal{E}}(\mathcal{T}^\sharp \alpha^{\pm,(n-1)})\big|_{\BB_{L_N}^*},\mathcal{T}^\sharp \alpha^{\pm,(n-1)}-\mathcal{T}^\sharp \alpha^{\pm}\right\rangle+\eta_n\left\langle \xi_n^\pm(\mathcal{T}^\sharp \alpha^{\pm,(n-1)}), \alpha^{\pm,(n)}-\alpha^{\pm,(n-1)}\right\rangle_{\ell_2},
\end{aligned}
\end{equation*}
where (i) follows from \eqref{eq:property of gradient}. Taking the expectation of the first term in the last line and using the convexity of the loss function $\ell(u,y)$ in $u$, we obtain
$$\EE\left[\left\langle \widehat{\partial\mathcal{E}}(\mathcal{T}^\sharp \alpha^{\pm,(n-1)})\big|_{\BB_{L_N}^*},\mathcal{T}^\sharp \alpha^{\pm,(n-1)}-\mathcal{T}^\sharp \alpha^{\pm}\right\rangle\right]\geq \EE\left[\ee(\mathcal{T}^\sharp \alpha^{\pm,(n-1)})-\ee(\mathcal{T}^\sharp \alpha^{\pm})\right].$$
For the second term in the above inequality, we have
\begin{equation*}
\begin{aligned}
&\EE\left[\eta_n\left|\left\langle \xi_n^\pm(\mathcal{T}^\sharp \alpha^{\pm,(n-1)}), \alpha^{\pm,(n)}-\alpha^{\pm,(n-1)}\right\rangle_{\ell_2}\right|\right]\\
\leq &\EE\left[\eta_n\left|g_\ell\left(\mathcal{T}^\sharp\alpha^{\pm,(n-1)}(X_n),Y_n\right)\right|
\left|\sum_{j=1}^{L_N}\left((\alpha^{+,(n)}_j-\alpha^{+,(n-1)}_j)-(\alpha^{-,(n)}_j-\alpha^{-,(n-1)}_j)\right)\phi_j(X_n)\right|\right]\\
\leq &2C_\ell\EE\left[\eta_n
\sup_{j\geq1}\left(j^{-2s}\sup_{x\in\Omega}|\phi_j(x)|\right)\sum_{j=1}^{L_N}\left|(\alpha^{+,(n)}_j-\alpha^{+,(n-1)}_j)-(\alpha^{-,(n)}_j-\alpha^{-,(n-1)}_j)\right|j^{2s}\right]\\
\leq&\frac{L}{2}(2C_\ell)^2\eta_n^2M_1^2+\frac{1}{2L}\EE\left[\left\Vert \mathcal{T}^\sharp \alpha^{\pm,(n)}-\mathcal{T}^\sharp \alpha^{\pm,(n-1)}\right\Vert_{\BB^s_{1,1}}^2\right].
\end{aligned}
\end{equation*}
Collecting the above three bounds, we obtain the estimate
\begin{align*}
&\eta_n \EE\left[\ee(\mathcal{T}^\sharp \alpha^{\pm,(n-1)})-\ee(\mathcal{T}^\sharp \alpha^{\pm})\right]\\
\leq&\EE\left[D_{s,1}^{L_N}(\alpha^{\pm,(n-1)},\alpha^\pm)-D_{s,1}^{L_N}(\alpha^{\pm,(n)},\alpha^\pm)\right]+\frac{L}{2}(2C_\ell)^2\eta_n^2M_1^2.
\end{align*}
Here we choose $\alpha^\pm=\mathcal{T}S_{L_N}(f_\rho)$. We next show that $\EE\left[D_{s,1}^{L_N}\left(\alpha^{\pm,(n)},\mathcal{T}S_{L_N}(f_\rho)\right)\right]$ admits an upper bound independent of both $n$ and $N$ for all $0\leq n\leq N$. Let us first bound $D_{s,1}^{L_N}\left(\alpha^{\pm,(0)},\mathcal{T}S_{L_N}(f_\rho)\right)$. Writing $f_\rho$ as $f_\rho=\sum_{j=1}^{\infty}\beta_j^\rho\phi_j$, and using the initialization $\alpha_j^{+,(0)}=\alpha_j^{-,(0)}=\frac{3L}{\pi^2}j^{-2s-2}$ for $1\leq j\leq L_N$ together with the definition of $\mathcal T$, we obtain
\begin{equation*}
\begin{aligned}
&D_{s,1}^{L_N}(\alpha^{\pm,(0)},\mathcal{T}S_{L_N}(f_\rho))=\sum_{j=1}^{L_N}j^{2s}|\beta_j^\rho|\log\left(\frac{|\beta_j^\rho|}{\alpha_j^{+,(0)}}\right)+2\sum_{j=1}^{L_N}j^{2s}\alpha_j^{+,(0)}
-\sum_{j=1}^{L_N}j^{2s}|\beta_j^\rho|\\
\leq&\sum_{j=1}^{L_N}j^{2s}|\beta_j^\rho|\log\left(\frac{\pi^2}{3L}|\beta_j^\rho|j^{2s+2}\right)+\frac{6L}{\pi^2}\sum_{j=1}^{L_N}j^{-2}\\
\leq&2\sum_{j=1}^{L_N}j^{2s}|\beta_j^\rho|\log\left(j\right)+\sum_{j=1}^{L_N}j^{2s}|\beta_j^\rho|\log\left(\frac{\pi^2}{3L}|\beta_j^\rho|j^{2s}\right)+
\frac{6L}{\pi^2}\sum_{j=1}^{\infty}j^{-2}\leq2M_2+L\log\left(\frac{\pi^2}{3}\right)+L.
\end{aligned}
\end{equation*}
For any $N\in\NN_+$ and $1\leq n\leq N$, using
$\mathcal{T}^\sharp\mathcal{T} S_{L_N}(f_\rho)=S_{L_N}(f_\rho)$, we have
\begin{equation*}
\begin{aligned}
&\EE\left[D_{s,1}^{L_N}(\alpha^{\pm,(n)},\mathcal{T}S_{L_N}(f_\rho))\right]\leq \EE\left[D_{s,1}^{L_N}(\alpha^{\pm,(n-1)},\mathcal{T}S_{L_N}(f_\rho))\right]+\frac{L}{2}(2C_\ell)^2\eta_n^2M_1^2\\
& + \eta_n \EE\left[\ee(S_{L_N}(f_\rho)) - \ee(f_\rho) + \ee(f_\rho)-\ee(\mathcal{T}^\sharp \alpha^{\pm,(n-1)})\right]\\
\leq&D_{s,1}^{L_N}(\alpha^{\pm,(0)},\mathcal{T}S_{L_N}(f_\rho))+\frac{L}{2}(2C_\ell)^2M_1^2\sum_{n=1}^N\eta_n^2+\sum_{n=1}^N\eta_n\left(\ee(S_{L_N}(f_\rho)) - \ee(f_\rho)\right)\\
\overset{\text{(i)}}{\leq} &\left(2M_2+L\log\left(\frac{\pi^2}{3}\right)+L\right)+\frac{L}{2}(2C_\ell)^2\eta_0^2M_1^2+\eta_0(2C_\ell)\Vert f_\rho\Vert_{\BB^s_{1,1}}=:C_{10}
\end{aligned}
\end{equation*}
where (i) follows from \autoref{lemma:A.6} and the choice $\eta_n=\eta_0N^{-\frac12}$. The constant $C_{10}$ is independent of both $n$ and $N$. Combining the above estimates and applying Jensen's inequality,
\begin{align*}
&\EE\left[\mathcal{E}(\bar{f}^s_{\alpha N})-\mathcal{E}(f_\rho)\right]\leq \frac{1}{\alpha N}\sum_{n=(1-\alpha)N+1}^{N}\EE\left[\ee(\mathcal{T}^\sharp \alpha^{\pm,(n-1)})-\ee(f_\rho^{L_N})\right]+\left(\ee(f_\rho^{L_N})-\mathcal{E}(f_\rho)\right)\\
\leq &\frac{\eta_N^{-1}}{\alpha N}\sum_{n=(1-\alpha)N+1}^{N}\left\{ 
\EE\left[D_{s,1}^{L_N}(\alpha^{\pm,(n-1)},\mathcal{T}f_\rho^{L_N})-D_{s,1}^{L_N}(\alpha^{\pm,(n)},\mathcal{T}f_\rho^{L_N})\right]+\frac{L}{2}(2C_\ell)^2\eta_n^2M_1^2\right\}\\
&+(2C_\ell)\Vert f_\rho\Vert_{\BB^s_{1,1}}N^{-\frac12}\\
\leq& \frac{1}{\alpha\eta_0} \EE\left[D_{s,1}^{L_N}(\alpha^{\pm,((1-\alpha)N)},\mathcal{T}f_\rho^{L_N})\right]\cdot N^{-\frac12}+\frac{L}{2}(2C_\ell)^2\eta_0M_1^2N^{-\frac12}+(2C_\ell)\Vert f_\rho\Vert_{\BB^s_{1,1}}N^{-\frac12}\\
\leq& \left(\frac{1}{\alpha\eta_0}C_{10} + \frac{L}{2}(2C_\ell)^2\eta_0M_1^2 +(2C_\ell)\Vert f_\rho\Vert_{\BB^s_{1,1}}\right) N^{-\frac12},
\end{align*}
where we continue to use the notation $f_{\rho}^{L_N}=S_{L_N}(f_\rho)$. This completes the proof.

\section{Conclusions}

In this paper, we study a class of risk functional minimization problems with general losses in $\LL_{\rho_X}^p(\Omega)$ and develop an adaptive Schauder stochastic mirror descent (AS-SMD) algorithm. The Bregman distance is constructed according to the geometry of the underlying Banach space. When $1<p<\infty$, the uniform convexity of $\BB_{p,p}^s$ allows us to construct a globally $\max\{p,2\}$-convex mirror map. For $p=1$, where $\BB_{1,1}^s$ is no longer uniformly convex, we instead use a locally strongly convex entropy mirror map on the bounded coefficient set $\mathcal{A}_L^{2L_N}$. We introduce a new regularization strategy that restricts the SMD iterates to finite-dimensional subspaces, making the associated subproblems explicitly solvable. The subspace dimension also serves as a regularization parameter that balances the approximation and optimization errors. Our theoretical analysis shows that, for $1<p<\infty$ and $f_\rho\in\BB_{p,p}^s$, the expected excess risk converges at the rate $\O\left(n^{-\min\{\frac12,\frac1p\}}\right)$ up to logarithmic factors. In the misspecified setting, we prove that the expected risk converges to its minimum under weaker regularity assumptions and derive explicit polynomial rates when the regularity exceeds a specified threshold. For $p=1$, we establish an $\O(N^{-1/2})$ convergence rate for the excess risk. We also extend the proposed method to statistical inverse problems, leaving the corresponding theoretical analysis for future work.  

The present work mainly focuses on $1\leq p<\infty$. When $p=\infty$, the space $\LL_{\rho_X}^\infty(\Omega)$ is generally nonseparable and does not admit a countable Schauder basis, making the current framework difficult to apply. A possible direction for future research is the continuous function space $\mathcal{C}(\Omega)$, which may admit a countable Schauder basis for suitable $\Omega$, although it is not uniformly convex. Extending the algorithm to this setting may require identifying an appropriate underlying subspace and constructing a corresponding Bregman distance in $\mathcal{C}(\Omega)$.

\bibliographystyle{plain}
\bibliography{reference.bib}

\appendix
\numberwithin{theorem}{section} 
\numberwithin{lemma}{section}
\numberwithin{proposition}{section}

\section{Appendix}\label{sec:appendix}

\begin{lemma}\label{lemma:A.1}
If we reorder the wavelet basis $\{\Phi_{0,j}\}_{j \in \Lambda_0}\cup\{\Psi_{k,j}\}_{j \in \Lambda_k,k> 0}$ into a single sequence $\{\Psi_j\}_{j \ge 1}$ according to the lexicographic order, then the Besov spaces $\B_{p,p}^s$ defined in \eqref{Besov space 1} and \eqref{Besov space 2} coincide, and the corresponding norms are equivalent. 
\end{lemma}
\begin{proof}
Since both definitions are given in terms of norms, it suffices to show that these two norms are equivalent. Here we denote the norm defined in \eqref{Besov space 2} by
$$\Vert f\Vert_{\BB_{p,p}^s}=\left(\sum_{j=1}^\infty \sigma_j|\beta_j|^p\right)^{1/p}, \quad \sigma_j = j^{p\left(\frac{s}{d}+\left(\frac{1}{2}-\frac{1}{p} \right)\right)}.$$
For each $j$, let $k$ denote the scale associated with the reordered basis function $\Psi_j$. To prove the equivalence of the two norms, it suffices to show that there exist constants $A_1,A_2>0$ such that, for any $j$, $A_2\sigma_j\leq 2^{kp\left(s+d(1/2-1/p)\right)}\leq A_1\sigma_j$. Here, we only prove the case $s>\max\left\{0,d\left(\frac1p-\frac12\right)\right\}$, as the other cases can be treated similarly. By the ordering of the basis functions, we have
\begin{equation*}
\begin{aligned}
&c_2 2^{(k-1)d} \leq c_2\sum_{l=0}^{k-1} 2^{ld}\leq\sum_{l=0}^{k-1}|\Lambda_l|\leq j\leq \sum_{l=0}^k|\Lambda_l|\leq c_1\sum_{l=0}^k 2^{ld}\leq 2c_1 2^{kd}\\
\Leftrightarrow\quad &\frac{1}{2c_1} j\leq 2^{kd}\leq \frac{2^d}{c_2} j\\
\Leftrightarrow\quad &\left(\frac{1}{2c_1}\right)^{p\left(\frac{s}{d}+\left(\frac{1}{2}-\frac{1}{p} \right)\right)}j^{p\left(\frac{s}{d}+\left(\frac{1}{2}-\frac{1}{p} \right)\right)}\leq 2^{kp\left(s+d\left(\frac12-\frac1p\right)\right)}\leq \left(\frac{2^d}{c_2}\right)^{p\left(\frac{s}{d}+\left(\frac{1}{2}-\frac{1}{p} \right)\right)} j^{p\left(\frac{s}{d}+\left(\frac{1}{2}-\frac{1}{p} \right)\right)}\\
\Leftrightarrow\quad &\left(\frac{1}{2c_1}\right)^{p\left(\frac{s}{d}+\left(\frac{1}{2}-\frac{1}{p} \right)\right)}\sigma_j\leq 2^{kp\left(s+d\left(\frac12-\frac1p\right)\right)}\leq \left(\frac{2^d}{c_2}\right)^{p\left(\frac{s}{d}+\left(\frac{1}{2}-\frac{1}{p} \right)\right)} \sigma_j.
\end{aligned}
\end{equation*}
The proof is completed.
\end{proof}

\begin{lemma}\label{lemma:A.2}
For $s>0,1\leq p<\infty$, the generalized Besov space $\BB_{p,p}^s$ defined in \eqref{generalized Besov space} is a Banach space. Moreover, for $1<p<\infty$, $\BB_{p,p}^s$ is reflexive, and, with $p'=\frac{p}{p-1}$, its dual space is given by
\begin{equation}\label{eq:dual space 2}
\left(\BB_{p,p}^s\right)^*=\left\{ f^*=\sum_{j=1}^\infty \beta_j\phi_j^*\,\Big\vert\, \Vert f^*\Vert_{\left(\BB_{p,p}^s\right)^*}:=\left(\sum_{j=1}^\infty j^{-2sp'}|\beta_j|^{p'}\right)^{1/p'}<\infty \right\}.
\end{equation}
For $s > \frac{1}{2} - \frac{1}{2p}$, if the Schauder basis of $\mathcal{L}_{\rho_X}^p(\Omega)$ satisfies $\Vert \phi_j\Vert_{\mathcal{L}^p}\le 1$, then the continuous embedding $\BB_{p,p}^s \subset \mathcal{L}_{\rho_X}^p(\Omega)$ holds, and we have
\begin{equation}\label{eq:embedding}
\left\Vert \sum_j\beta_j\phi_j\right\Vert_{\LL^p}\leq C_1\left\Vert \sum_j\beta_j\phi_j\right\Vert_{\BB_{p,p}^s},\quad\forall
\ \sum_j\beta_j\phi_j\in\BB_{p,p}^s,
\end{equation}
where, if $1<p<\infty$, $C_1=\left(\frac{2sp'}{2sp'-1}\right)^{1/p'}$, and if $p=1$, $C_1=1$.
\end{lemma}
\begin{proof}
Since the proof of the Minkowski inequality and the completeness of $\BB_{p,p}^s$ are analogous to those for the sequence space $\ell^p$, we omit the details here. We next turn to the characterization of the dual space of $\BB_{p,p}^s$. For an arbitrary \(f = \sum_j\alpha_j\phi_j\in\BB_{p,p}^s\), since \(s>(p-1)/(2p)\) and \(\|\phi_j\|_{L^p}\le1\), we have
\begin{equation*}
\begin{aligned}
|\langle f^*,f\rangle| &= \left|\left\langle \sum_{j=1}^\infty \beta_j\phi_j^*,\sum_{j=1}^\infty \alpha_j\phi_j\right\rangle\right|=\left|\sum_{j=1}^\infty \alpha_j\beta_j\right|=\left|\sum_{j=1}^\infty \alpha_j j^{2s} j^{-2s}\beta_j\right|\\
&\leq \left(\sum_{j=1}^\infty|\beta_j|^{p'}j^{-2sp'}\right)^{1/p'}\left(\sum_{j=1}^\infty|\alpha_j|^pj^{2sp}\right)^{1/p}=\Vert f^*\Vert_{\left(\BB_{p,p}^s\right)^*}\Vert f\Vert_{\BB_{p,p}^s}.
\end{aligned}
\end{equation*}
Therefore, $f^*$ belongs to the dual space of $\BB_{p,p}^s$. Take an arbitrary $g^*$ in the dual space, and denote $g^*(\phi_j) =\gamma_j$. If $\sum_{j=1}^n|\gamma_j|^{p'}j^{-2sp'}=0$, the desired inequality is immediate. We choose $g_n=\frac{1}{\left(\sum_{j=1}^n|\gamma_j|^{p'}j^{-2sp'}\right)^{1/p}}\sum_{j=1}^n|\gamma_j|^{p'-1}\text{sign}(\gamma_j) j^{-2sp'}\phi_j$, then $\Vert  g_n\Vert_{\BB_{p,p}^s}=1$, and we have
$$\langle g^*,g_n\rangle = \left(\sum_{j=1}^n|\gamma_j|^{p'}j^{-2sp'}\right)^{1/p'}\leq \Vert g^*\Vert:=\sup_{\Vert f\Vert_{\BB_{p,p}^s}\leq 1}|g^*(f)|,\quad \forall n\in\NN.$$
Thus, we obtain $g^*\in \left(\BB_{p,p}^s\right)^*$; consequently, we have proved that the dual space of $\BB_{p,p}^s$ is $\left(\BB_{p,p}^s\right)^*$. The above two estimates show respectively that the norm defined in \eqref{eq:dual space 2} provides an upper bound and a lower bound for the norm of the dual space; therefore, the norm defined in \eqref{eq:dual space 2} is exactly the norm of the dual space. Finally, we prove that the continuous embedding $\BB_{p,p}^s \subset \mathcal{L}_{\rho_X}^p(\Omega)$ holds; in fact, it suffices to show that inequality \eqref{eq:embedding} holds. The case $p=1$ is straightforward. Hence, we only prove the result for $1<p<\infty$. For an arbitrary $\sum_j\beta_j\phi_j\in\BB_{p,p}^s$, using $s>\frac{p-1}{2p}=\frac{1}{2p'}$ and $\|\phi_j\|_{L^p}\leq1$, we have
\begin{align*}
&\left\Vert \sum_j\beta_j\phi_j\right\Vert_{\LL^p}\leq \sum_{j=1}^\infty|\beta_j|\left\Vert \phi_j\right\Vert_{\LL^p}\leq \left(\sum_{j=1}^\infty|\beta_j|^pj^{2sp}\right)^{1/p}\left(\sum_{j=1}^\infty j^{-2sp'}\cdot 1\right)^{1/p'}\\
\leq& \left\Vert \sum_j\beta_j\phi_j\right\Vert_{\BB_{p,p}^s}\left(1+\int_{1}^\infty x^{-2sp'}dx\right)^{1/p'}=\left\Vert \sum_j\beta_j\phi_j\right\Vert_{\BB_{p,p}^s}\left(\frac{2sp'}{2sp'-1}\right)^{1/p'}.
\end{align*}
This completes the proof.
\end{proof}

\begin{lemma}\label{lemma:A.3}
For $s>0$ and $1<p<\infty$, the generalized Besov spaces $\BB_{p,p}^s$ are smooth and $\max\{p,2\}$-convex. For $p=1$, the space $\BB_{1,1}^s$ is not uniformly convex.
\end{lemma}
\begin{proof}
First, we show that $\BB_{p,p}^s$ is smooth. For any nonzero $f=\sum_{j=1}^\infty \alpha_j\phi_j\in \BB_{p,p}^s$, the existence of $f^*$ is guaranteed by the Hahn–Banach theorem; it suffices to prove its uniqueness. If there exists $f^*=\sum \beta_j\phi_j^*\in \left(\BB_{p,p}^s\right)^*$ such that $\langle f^*, f\rangle = \Vert f\Vert_{\BB_{p,p}^s}$ and $\Vert f^*\Vert_{\left(\BB_{p,p}^s\right)^*}=1$, then one has
$$\langle f^*, f\rangle=\sum_{j=1}^\infty \beta_j \alpha_j\overset{\text{(i)}}{=}\left(\sum_{j}|\beta_j|^{p'}j^{-2sp'}\right)^{1/p'}\left(\sum_{j}|\alpha_j|^{p}j^{2sp}\right)^{1/p}= \Vert f\Vert_{\BB_{p,p}^s}\Vert f^*\Vert_{\left(\BB_{p,p}^s\right)^*}.$$
By the definition of $f^*$, equality holds in H\"{o}lder's inequality, i.e., (i) is satisfied. By the equality condition in H\"{o}lder's inequality, there exists a constant $c$ such that $j^{-2sp'}|\beta_j|^{p'} = cj^{2sp}|\alpha_j|^{p}$. Using $\Vert f^*\Vert_{\left(\BB_{p,p}^s\right)^*}=1$, we can deduce that $\beta_j=\frac{1}{\left(\sum_{l=1}^\infty l^{2sp}|\alpha_l|^{p}\right)^{\frac{1}{p'}}}j^{2sp}|\alpha_j|^{p-1}\text{sign}(\alpha_j)$, which implies that $f^*$ is unique. We next show that the space $\mathcal{B}_{p,p}^s$ is $\max\{p,2\}$-convex. In \autoref{section:Preliminary}, we have shown that the sequence space $\ell^p$ is $\max\{p,2\}$-convex for $1<p<\infty$. By the definition of $\max\{p,2\}$-convexity, there exists a constant $C_p$ such that for any $(a_j)_{j\geq1}, (b_j)_{j\geq1}\in\ell_p$, if $\Vert (a_j)_{j\geq1}\Vert_{\ell_p} = \Vert (b_j)_{j\geq1}\Vert_{\ell_p} =1$ and $\Vert (b_j)_{j\geq1}-(a_j)_{j\geq1}\Vert_{\ell_p}\geq \tau$, we have
$$1 - \left\| \frac{(b_j)_{j\geq1}+(a_j)_{j\geq1}}{2} \right\|_{\ell^p}\geq C_p\tau^{\max\{p,2\}}.$$
Here we take arbitrary $f=\sum\alpha_j\phi_j, g=\sum\beta_j\phi_j\in\BB_{p,p}^s$ satisfying $\Vert f\Vert_{\BB_{p,p}^s} = \Vert g\Vert_{\BB_{p,p}^s} =1$ and $\Vert f-g\Vert_{\BB_{p,p}^s}\geq \tau$. Using $\Vert f\Vert_{\BB^s_{p,p}}=\Vert (j^{2s}\alpha_j)_{j\geq1}\Vert_{\ell^p}$ and $\Vert g\Vert_{\BB^s_{p,p}}=\Vert (j^{2s}\beta_j)_{j\geq1}\Vert_{\ell^p}$,
we set $a_j =j^{2s}\alpha_j$ and $b_j = j^{2s}\beta_j$. We can obtain
$$ 1 - \left\| \frac{f+g}{2} \right\|_{\BB^s_{p,p}}=1 - \left\| \frac{(j^{2s}\beta_j)_{j\geq1}+(j^{2s}\alpha_j)_{j\geq1}}{2} \right\|_{\ell^p} \geq C_p\tau^{\max\{p,2\}}.$$
Therefore, $\BB_{p,p}^s$ is a $\max\{p,2\}$-convex space. For $p=1$, we take $f=\phi_1$ and $g=2^{-2s}\phi_2$. Then $\Vert f \Vert_{\BB^s_{1,1}} = \Vert g\Vert_{\BB^s_{1,1}}=1$ and $\Vert f-g \Vert_{\BB^s_{1,1}}=2$, but $1-\Vert \frac{f+g}{2} \Vert_{\BB^s_{1,1}} =0$. This implies that $\delta_{\BB^s_{1,1}}(2) = 0$, and hence $\BB^s_{1,1}$ is not uniformly convex. This completes the proof. 
\end{proof}

\begin{lemma}\label{lemma:add A.1}
If the loss function $\ell$ satisfies \autoref{assume:1}, then, for every $(u,y)\in\RR\times\Y$, any subgradient $g$ of $\ell(\cdot,y)$ at $u$ satisfies $|g|\leq 2^q C_\ell(1+|u|+|y|)^{q-1}$. Consequently,
$$|\ell(u_1,y)-\ell(u_2,y)| \leq 2^q C_\ell |u_1-u_2|(1+|u_1|+|u_2|+|y|)^{q-1}, \qquad \forall\,u_1,u_2\in\RR,\ y\in\Y.$$
\end{lemma}
\begin{proof}
By the subgradient inequality, we have
\begin{align*}
\ell(u+(1+|u|+|y|),y) &\geq \ell(u,y)+(1+|u|+|y|)g,\\
\ell(u-(1+|u|+|y|),y) &\geq \ell(u,y)-(1+|u|+|y|)g.
\end{align*}
Combining these inequalities with $\ell(u,y)\geq0$ yields
\begin{align*}
&-\frac{\ell(u-(1+|u|+|y|),y)}{1+|u|+|y|} \leq g \leq \frac{\ell(u+(1+|u|+|y|),y)}{1+|u|+|y|}\\
\Rightarrow\quad &|g| \leq \frac{C_\ell(1+||u|+(1+|u|+|y|)|+|y|)^q}{1+|u|+|y|} =2^q C_\ell(1+|u|+|y|)^{q-1}.
\end{align*}
We next prove that $\ell(\cdot,y)$ is locally Lipschitz continuous. Let $g_1$ and $g_2$ be subgradients of $\ell(\cdot,y)$ at $u_1$ and $u_2$, respectively. By the subgradient inequality,
\begin{align*}
\ell(u_1,y) &\geq \ell(u_2,y)+g_2(u_1-u_2),\\
\ell(u_2,y) &\geq \ell(u_1,y)+g_1(u_2-u_1).
\end{align*}
Together with the preceding subgradient bound, these inequalities imply
\begin{align*}
|\ell(u_1,y)-\ell(u_2,y)|\leq \max\{|g_1|,|g_2|\}|u_1-u_2|\leq 2^q C_\ell(1+|u_1|+|u_2|+|y|)^{q-1}|u_1-u_2|.
\end{align*}
This completes the proof.
\end{proof}

\begin{lemma}\label{lemma:A.4}
The functional $\EE_\rho\left[g_\ell(f(X),Y)h(X)\right]$ is bounded in $\BB_{p,p}^s$ with respect to $h$, for any fixed function $f\in\mathcal{L}_{\rho_X}^p(\Omega)$. If $\widehat{\partial\mathcal{E}}(f)$ is a bounded linear functional, then $\widehat{\partial\mathcal{E}}(f)\big|_{\BB_{L_n}^*}$ is the projection of $\widehat{\partial\mathcal{E}}(f)$ onto $\BB_{L_n}^*$.
\end{lemma}
\begin{proof}
We consider only the case $q>1$, since the case $q=1$ is straightforward. We first establish the boundedness of the linear functional. For any $h\in\mathcal{B}_{p,p}^s$, it follows from H\"older's inequality and \autoref{lemma:add A.1} that
\begin{align*}
&\left|\EE_\rho\left[g_\ell(f(X),Y)h(X)\right]\right|\leq 2^q C_\ell \left(\EE_\rho\left[(1+|f(X)|+|Y|)^{q'(q-1)}\right]\right)^{\frac1{q'}}\left(\EE_\rho\left[|h(X)|^{q}\right]\right)^{\frac1q}\\
\leq&2^q C_\ell\left(\EE_\rho\left[(1+|f(X)|+|Y|)^q\right]\right)^{1/q'}\Vert h\Vert_{\LL^p}\\
\overset{\text{(i)}}{\leq}& 2^q C_\ell\left(1+\Vert f\Vert_{\LL^q}+\left(\EE_\rho\left[|Y|^q\right]\right)^{1/q}\right)^{q/q'}\left(\frac{2sp'}{2sp'-1}\right)^{1/p'}\left\Vert h\right\Vert_{\BB_{p,p}^s},
\end{align*}
where $q'=\frac{q}{q-1}$. Step (i) follows directly from \autoref{lemma:A.2} and Minkowski's inequality. We now prove the second part of the lemma. Let $f^* = \sum_{j=1}^{L_n} \alpha_j\phi_j^* \in \mathcal{B}_{L_n}^*$ be arbitrary, and define $w(f)= g_\ell(f(X),Y)$. Then 
\begin{align*}
&\left\Vert \widehat{\partial\mathcal{E}}(f) -f^*\right\Vert_{\left(\BB_{p,p}^s\right)^*}^{p'}=\sum_{j=L_n+1}^\infty j^{-2sp'} |w(f)\phi_j(X)|^{p'}+\sum_{j=1}^{L_n}j^{-2sp'}|w(f)\phi_j(X)-\alpha_j|^{p'}\\
\geq &\sum_{j=L_n+1}^\infty j^{-2sp'} |w(f)\phi_j(X)|^{p'}=\left\Vert \widehat{\partial\mathcal{E}}(f) -\widehat{\partial\mathcal{E}}(f)\big|_{\BB_{L_n}^*}\right\Vert_{\left(\BB_{p,p}^s\right)^*}^{p'}.
\end{align*}
This completes the proof.
\end{proof}

\begin{lemma}\label{lemma:A.5}
Suppose that the assumptions of \autoref{lemma:5.1} hold. Then, for $1<p<2$ and $q>1$, the first term on the right-hand side of \eqref{eq:sec5.1} can be bounded as follows:
\begin{equation*}
\begin{aligned}
&\eta_n\EE\left[\left|\left\langle \widehat{\partial\mathcal{E}}(f_{n-1})\big|_{\BB_{L_n}^*}, f_n-f_{n-1}\right\rangle\right|\big|\FF_{n-1}\right]\\
\leq&  \frac{1}{2}\frac{C_{p,s}}{2}\EE\left[\Vert f_n-f_{n-1}\Vert_{\BB^s_{p,p}}^2\big|\FF_{n-1}\right]+\left(\frac{1}{2}\left(\frac{C_{p,s}}{2}\right)^{-1}\left(2^qC_\ell\right)^2 M^22^{\frac{3p}{\gamma'}}C_1^2\right)\eta_n^{2}\Vert f_{n-1}-v\Vert_{\BB^s_{p,p}}^{2}\\
&+\left(\frac{1}{2}\left(\frac{C_{p,s}}{2}\right)^{-1}\left(2^qC_\ell\right)^2 M^22^{\frac{3p}{\gamma'}}\left( 2+\left(\EE[| Y_n|^p]\right)^{\frac{1}{\gamma'}}+\Vert v\Vert_{\LL^p}^{\frac{p}{\gamma'}} \right)\right)\eta_n^{2},
\end{aligned}
\end{equation*}
where $M^2 = \sum_{j=1}^\infty j^{-4s}\Vert \phi_j\Vert_{\LL^t}^{2}$ with $t\geq \frac{2p}{p-2(q-1)}$. Moreover, the H\"{o}lder conjugate exponents are given by $\gamma' =\frac{p}{2(q-1)}$ and $\gamma =\frac{p}{p-2(q-1)}$.
\end{lemma}
\begin{proof}
The proof of this lemma is analogous to the proof for the case $2 \leq p < \infty$. We therefore only present the main steps.
\begin{align*}
&\eta_n\left|\left\langle \widehat{\partial\mathcal{E}}(f_{n-1})\big|_{\BB_{L_n}^*}, f_n-f_{n-1}\right\rangle\right|\\
\leq& \eta_n\left|g_\ell\left(f_{n-1}(X_n),Y_n\right)\right|\cdot\left(\sum_{j=1}^{L_n}j^{2sp}\left|\langle \phi_j^*,f_n-f_{n-1}\rangle\right|^p\right)^{1/p}\left( \sum_{j=1}^{L_n}j^{-2sp'}|\phi_j(X_n)|^{p'}\right)^{1/p'}\\
\overset{\text{(i)}}{\leq}&\frac{1}{2}\left(\frac{C_{p,s}}{2}\right)^{-1}\eta_n^2\left|g_\ell\left(f_{n-1}(X_n),Y_n\right)\right|^{2}\left(\sum_{j=1}^{L_n}j^{-2sp'}|\phi_j(X_n)|^{p'}\right)^{2/p'}
+\frac{1}{2}\frac{C_{p,s}}{2}\Vert f_n-f_{n-1}\Vert_{\BB^s_{p,p}}^2,
\end{align*}
where (i) follows from Young's inequality. We use \autoref{lemma:add A.1} to bound the first term on the right-hand side.
\begin{align*}
&\EE\left[\left|g_\ell\left(f_{n-1}(X_n),Y_n\right)\right|^{2}\left(\sum_{j=1}^{L_n}j^{-2sp'}|\phi_j(X_n)|^{p'}\right)^{2/p'}\Big|\FF_{n-1}\right]\\
\overset{\text{(i)}}{\leq}&\left(2^qC_\ell\right)^2\sum_{j=1}^{L_n}j^{-4s}\EE\left[\left(1+|f_{n-1}(X_n)|+|Y_n|\right)^{2(q-1)}|\phi_j(X_n)|^{2}\Big|\FF_{n-1}\right]\\
\overset{\text{(ii)}}{\leq}&\left(2^qC_\ell\right)^2\left(\EE\left[\left(1+|f_{n-1}(X_n)|+|Y_n|\right)^{2(q-1)\gamma'}\Big|\FF_{n-1}\right]\right)^{1/\gamma'}
\sum_{j=1}^{L_n}j^{-4s}\left(\EE\left[|\phi_j(X_n)|^{2\gamma}\right]\right)^{1/\gamma}\\
\overset{\text{(iii)}}{\leq} &\left(2^qC_\ell\right)^2 M^22^{\frac{3p}{\gamma'}}\left(C_1^2\Vert f_{n-1}-v\Vert_{\BB^s_{p,p}}^{2}+2+\Vert v\Vert_{\LL^p}^{\frac{p}{\gamma'}}+\left(\EE[| Y_n|^p]\right)^{\frac{1}{\gamma'}}\right).
\end{align*}
Step (i) follows from the fact that $2/p' < 1$ and $(a+b)^{2/p'}\leq a^{2/p'}+b^{2/p'}$ for $a,b>0$. Step (ii) is obtained by applying H\"{o}lder's inequality with conjugate exponents $\gamma' =\frac{p}{2(q-1)}$ and $\gamma =\frac{p}{p-2(q-1)}$. In (iii), we use $M^2 = \sum_{j=1}^{\infty}j^{-4s}\Vert\phi_j\Vert_{\LL^t}^2$ with $t\geq \frac{2p}{p-2(q-1)}$. Combining the above inequalities, we conclude the proof. 
\end{proof}

\begin{lemma}\label{lemma:A.8}
Suppose that the assumptions of \autoref{lemma:5.1} hold. For $1 < p < \infty$ and $q=1$, the first term on the right-hand side of \eqref{eq:sec5.1} can be bounded as 
\begin{equation*}
\begin{aligned}
&\eta_n\EE\left[\left|\left\langle \widehat{\partial\mathcal{E}}(f_{n-1})\big|_{\BB_{L_n}^*}, f_n-f_{n-1}\right\rangle\right|\big|\FF_{n-1}\right]\\
\leq&  \frac{1}{p_1}\frac{C_{p,s}}{2}\EE\left[\Vert f_n-f_{n-1}\Vert_{\BB^s_{p,p}}^{p_1}\big|\FF_{n-1}\right]+\left(\frac{1}{p_1'}\left(\frac{C_{p,s}}{2}\right)^{-\frac{p_1'}{p_1}}(2^qC_\ell)^{p_1'}M^2\right)\eta_n^{p_1'}.
\end{aligned}
\end{equation*}
\end{lemma}

\begin{remark}
Since the proof of \autoref{lemma:A.8} is similar to the arguments for the case $2\leq p<\infty$ in \autoref{lemma:5.1} and for the case $1<p<2$ in \autoref{lemma:A.5}, we state only the conclusion of the lemma and omit the proof.
\end{remark}

\begin{lemma}\label{lemma:A.6}
Assume that $f_\rho=\sum_{j=1}^{\infty}\beta_j^\rho\phi_j\in\B^s_{p,p}$. Let $L_n=\lceil n^\theta\rceil$ with $\theta>0$, and denote the truncation of $f_\rho$ to $\BB_{L_n}$ by $f_\rho^{L_n}:=S_{L_n}(f_\rho)=\sum_{j=1}^{L_n}\beta_j^\rho\phi_j$. Suppose that \autoref{assume:1} holds, and set $\nu=1$. If \autoref{assume:2} also holds, take $\nu$ to be the exponent specified therein. For $1<p<\infty$, we have
\begin{equation*}
\ee(f_\rho^{L_n})-\ee(f_\rho)\leq C_4 n^{\theta\nu\left(-2s+\frac{1}{p'}\right)},
\end{equation*}
where
\begin{equation*}
C_4:=
\begin{cases}
\begin{aligned}
&2^qC_\ell\Vert f_\rho\Vert_{\BB^s_{p,p}}\left(\frac{1}{2sp'-1}\right)^{1/p'}\\
&\quad{}\times\left[1+\left(1+\sup_{l\geq1}\Vert S_l\Vert\right)\Vert f_\rho\Vert_{\LL^p}+\left(\EE[|Y|^p]\right)^{1/p}\right]^{q-1},
\end{aligned}
& \nu=1,\\[3ex]
\begin{aligned}
&\frac{C_\nu}{\nu}\Vert f_\rho\Vert_{\BB^s_{p,p}}^\nu\left(\frac{1}{2sp'-1}\right)^{\nu/p'}\\
&\quad{}\times\left[1+2\left(1+\sup_{l\geq1}\Vert S_l\Vert\right)\Vert f_\rho\Vert_{\LL^p}+\left(\EE[|Y|^p]\right)^{1/p}\right]^{q-\nu},
\end{aligned}
& 1<\nu\leq\min\{q,2\}.
\end{cases}
\end{equation*}
For $p=q=1$, we have $\ee(f_\rho^{L_n})-\ee(f_\rho)\leq 2C_\ell\Vert f_\rho\Vert_{\BB^s_{1,1}}n^{-2s\theta}=:C_4n^{-2s\theta}$.
\end{lemma}
\begin{proof}
We first consider the case where only \autoref{assume:1} is assumed. By \autoref{lemma:add A.1}, we have
\begin{equation}\label{eq:A.3}
\begin{aligned}
&\ee(S_{L_n}f_\rho)-\ee(f_\rho)\leq 2^qC_\ell\EE\left[|S_{L_n}f_\rho(X)-f_\rho(X)|\cdot(1+|f_\rho(X)|+|S_{L_n}f_\rho(X)|+|Y|)^{q-1}\right]\\
\leq &2^qC_\ell\Vert S_{L_n}f_\rho-f_\rho\Vert_{\LL^p}\left(\EE\left[(1+|f_\rho(X)|+|S_{L_n}f_\rho(X)|+|Y|)^{(q-1)p'}\right]\right)^{\frac 1{p'}}\\
\overset{\text{(i)}}{\leq}&2^qC_\ell\Vert S_{L_n}f_\rho-f_\rho\Vert_{\LL^p}\left(1+\Vert f_\rho\Vert_{\LL^p}+\Vert S_{L_n}f_\rho\Vert_{\LL^p}+\left(\EE\left[|Y|^p\right]\right)^{1/p}\right)^{q-1}\\
\leq &2^qC_\ell\Vert S_{L_n}f_\rho-f_\rho\Vert_{\LL^p}\left(1+\Vert f_\rho\Vert_{\LL^p}+\sup_{l\geq1}\Vert S_{l}\Vert\cdot\Vert f_\rho\Vert_{\LL^p}+\left(\EE\left[|Y|^p\right]\right)^{1/p}\right)^{q-1}.
\end{aligned}
\end{equation}
where (i) follows from $\left(\EE[|\cdot|^{(q-1)p'}]\right)^{\frac{1}{(q-1)p'}} \leq \left(\EE[|\cdot|^p]\right)^{\frac{1}{p}}$, since $0<(q-1)p'\leq p$. For $q=1$, (i) holds directly. We next prove the lemma under the additional assumption that \autoref{assume:2} holds. For all $u_1,u_2\in\RR$ and $y\in\Y$, we have
\begin{align*}
&|\ell(u_1,y)-\ell(u_2,y)-\partial_u\ell(u_2,y)(u_1-u_2)|\\
=&\left|\int_0^1\left(\partial_u\ell(u_2+t(u_1-u_2),y)-\partial_u\ell(u_2,y)\right)(u_1-u_2)dt\right|\\
\leq&C_\nu|u_1-u_2|^\nu\left(1+2|u_1|+2|u_2|+|y|\right)^{q-\nu}\int_0^1t^{\nu-1}dt\\
=&\frac{C_\nu}{\nu}|u_1-u_2|^\nu\left(1+2|u_1|+2|u_2|+|y|\right)^{q-\nu}.
\end{align*}
It is easy to verify that the subgradient at $f_\rho$ satisfies
\begin{equation}\label{eq:A.1}
\partial\mathcal{E}(f_\rho)(h)= \mathbb{E}_\rho\left[\partial_u\ell(f_\rho(X),Y)h(X)\right]=0, \qquad h\in \LL^p_{\rho_X}(\Omega).
\end{equation} 
Setting $u_1=S_{L_n}f_\rho(X)$, $u_2=f_\rho(X)$, and $y=Y$, and using \eqref{eq:A.1}, we obtain
\begin{equation}\label{eq:A.4}
\begin{aligned}
&\ee(S_{L_n}f_\rho)-\ee(f_\rho)\\
\leq& \frac{C_\nu}{\nu}\E\left[|S_{L_n}f_\rho(X)-f_\rho(X)|^\nu\cdot(1+2|f_\rho(X)|+2|S_{L_n}f_\rho(X)|+|Y|)^{q-\nu}\right]\\
\leq& \frac{C_\nu}{\nu}\Vert S_{L_n}f_\rho-f_\rho\Vert_{\LL^p}^\nu\left(1+2\Vert f_\rho\Vert_{\LL^p}+2\Vert S_{L_n}f_\rho\Vert_{\LL^p}+\left(\EE\left[|Y|^p\right]\right)^{1/p}
\right)^{q-\nu}.
\end{aligned}
\end{equation}
Finally, we estimate the approximation error $\Vert f_\rho-S_{L_n}f_\rho\Vert_{\LL^p}$.
\begin{align*}
&\Vert f_\rho-S_{L_n}f_\rho\Vert_{\LL^p}\leq \sum_{j=L_n+1}^\infty|\beta_j^\rho|\Vert\phi_j \Vert_{\LL^p}\leq\left(\sum_{j=L_n+1}^\infty|\beta_j^\rho|^pj^{2sp}\right)^{1/p}\left(\sum_{j=L_n+1}^\infty j^{-2sp'}\right)^{1/p'}\\
\leq&\Vert f_\rho\Vert_{\BB^s_{p,p}}\left(\int_{L_n}^\infty x^{-2sp'}dx\right)^{1/p'}\leq\Vert f_\rho\Vert_{\BB^s_{p,p}}\left(\frac{1}{2sp'-1}\right)^{1/p'}L_n^{-2s+\frac{1}{p'}}.
\end{align*}
Substituting the above estimate into \eqref{eq:A.3} and \eqref{eq:A.4} completes the proof for $p>1$.
For $p=q=1$, we have
\begin{align*}
&\ee(S_{L_n}f_\rho)-\ee(f_\rho)\leq 2C_\ell\EE\left[\left|S_{L_n}f_\rho(X)-f_\rho(X)\right|\right]\leq2C_\ell \sum_{j=L_n+1}^\infty |\beta_j^\rho|\Vert \phi_j\Vert_{\LL^1}\\
\leq & 2C_\ell(L_n+1)^{-2s}\sum_{j=L_n+1}^\infty |\beta_j^\rho|j^{2s}\leq 2C_\ell\Vert f_\rho\Vert_{\BB^s_{1,1}}n^{-2s\theta}.
\end{align*}
Thus, we have completed the proof.
\end{proof}

\begin{lemma}\label{lemma:A.7}
The functional $\R^{L_N}_{s,1}$ satisfies the following strong convexity estimate:
\begin{equation*}
\begin{aligned}
D_{s,1}^{L_N}(\alpha^\pm,\beta^\pm)\geq \frac{1}{2L} \Vert \mathcal{T}^\sharp(\alpha^\pm)-\mathcal{T}^\sharp(\beta^\pm)\Vert_{\BB^s_{1,1}}^2,\quad \forall \alpha^{\pm}\in\mathcal{A}_L^{2L_N}\cap\RR^{2L_N}_{++},\ \beta^\pm\in\mathcal{A}_L^{2L_N}.
\end{aligned}
\end{equation*}
\end{lemma}
\begin{proof}
We write $\alpha^\pm = \left(\alpha_1^+,\dots,\alpha_{L_N}^+,\alpha_1^-,\dots,\alpha_{L_N}^-\right)$ and $\beta^\pm =\left(\beta_1^+,\dots,\beta_{L_N}^+, \beta_1^-,\dots,\beta_{L_N}^-\right)$. We first prove the result for the case $\beta^\pm\in\mathcal{A}_L^{2L_N}\cap\RR^{2L_N}_{++}$. Since $\mathcal{A}_L^{2L_N}\cap\RR^{2L_N}_{++}$ is convex, we define $z(t) = (1-t)\alpha^\pm +t\beta^\pm$ for $t\in[0,1]$. Consider the composite function $u(t) = \R^{L_N}_{s,1}\left(z(t)\right)$. Clearly, $u(t)$ is twice continuously differentiable on $[0,1]$, and its second derivative is given by
$$ u''(t) = \sum_{j=1}^{L_N}\sum_{\sigma\in\{-1,+1\}}\frac{j^{2s}\left(\beta_j^\sigma-\alpha_j^\sigma\right)^2}{{z(t)}_j^\sigma},$$
where $z(t)=\left({z(t)}_1^+,\dots {z(t)}_{L_N}^+,{z(t)}_1^-,\dots {z(t)}_{L_N}^-\right)$. Then, we have
\begin{equation*}
\begin{aligned}
&\Vert \mathcal{T}^\sharp(\alpha^\pm)-\mathcal{T}^\sharp(\beta^\pm)\Vert_{\BB^s_{1,1}}^2= \left(\sum_{j=1}^{L_N}|(\alpha_j^+-\alpha_j^-)-(\beta_j^+-\beta_j^-)|j^{2s}\right)^2\\
\leq& \left(\sum_{j=1}^{L_N}\sum_{\sigma\in\{+1,-1\}}|\alpha_j^\sigma - \beta_j^\sigma|j^{2s}\right)^2\leq  \left(\sum_{j=1}^{L_N}\sum_{\sigma\in\{+1,-1\}}j^s\sqrt{z(t)_j^\sigma}\frac{|\alpha_j^\sigma - \beta_j^\sigma|j^{s}}{\sqrt{z(t)_j^\sigma}}\right)^2\\
\leq& \left(\sum_{j=1}^{L_N}\sum_{\sigma\in\{+1,-1\}}j^{2s}z(t)_j^\sigma\right)\left(\sum_{j=1}^{L_N}\sum_{\sigma\in\{+1,-1\}}\frac{|\alpha_j^\sigma - \beta_j^\sigma|^2j^{2s}}{z(t)_j^\sigma}\right)\leq Lu''(t).
\end{aligned}
\end{equation*}
Combining this with the above inequality, we obtain
\begin{equation*}
\begin{aligned}
&D_{s,1}^{L_N}(\alpha^\pm,\beta^\pm) = u(1) - u(0) - u'(0) = \int_{0}^1 (1-t)u''(t)dt\geq\frac1{2L} \Vert \mathcal{T}^\sharp(\alpha^\pm)-\mathcal{T}^\sharp(\beta^\pm)\Vert_{\BB^s_{1,1}}^2.
\end{aligned}
\end{equation*}
Finally, we consider the case $\beta^\pm\in\mathcal{A}_L^{2L_N}$. For any $t\in(0,1)$, define $\beta^{\pm}(t) = (1-t)\beta^\pm+t\alpha^{\pm}\in\mathcal{A}_L^{2L_N}\cap\RR^{2L_N}_{++}$. It follows that
$$ D_{s,1}^{L_N}(\alpha^\pm,\beta^\pm(t)) \geq\frac1{2L} \Vert \mathcal{T}^\sharp(\alpha^\pm)-\mathcal{T}^\sharp(\beta^\pm(t))\Vert_{\BB^s_{1,1}}^2.$$
By the continuity of $D_{s,1}^{L_N}$ on its domain, letting $t\to0$ yields the desired strong convexity estimate.
\end{proof}

\begin{lemma}\label{lemma:A.9}
Suppose that the distribution $\rho_X$ on $\Omega$ is absolutely continuous with respect to a known $\sigma$-finite measure $\omega$. Then the Radon--Nikodym derivative $\frac{d\rho_X}{d\omega}$ exists. Assume further that there exist constants $0<C'<C''<\infty$ such that $C'\leq\frac{d\rho_X}{d\omega}\leq C''$ almost everywhere on $\Omega$. For $1\leq p<\infty$, the spaces $\LL_{\rho_X}^p(\Omega)$ and $\LL_{\omega}^p(\Omega)$ contain the same elements, and their norms are equivalent, i.e., for every $f\in\LL_{\rho_X}^p(\Omega)$,
$$
(C')^{1/p}\Vert f\Vert_{\LL_{\omega}^p(\Omega)}
\leq
\Vert f\Vert_{\LL_{\rho_X}^p(\Omega)}
\leq
(C'')^{1/p}\Vert f\Vert_{\LL_{\omega}^p(\Omega)}.
$$
Furthermore, if $\{\phi_j\}_{j\geq1}$ is a Schauder basis of $\LL_{\omega}^p(\Omega)$, it is also a Schauder basis of $\LL_{\rho_X}^p(\Omega)$.
\end{lemma}
\begin{proof}
We first establish the norm equivalence, which in particular implies that $\LL_{\rho_X}^p(\Omega)$ and $\LL_{\omega}^p(\Omega)$ contain the same elements. For $f\in\LL_{\rho_X}^p(\Omega)$, we have
\begin{equation*}
\begin{aligned}
(C')^{\frac1p}\left(\int_\Omega |f|^p d\omega\right)^{\frac1p}\leq \left(\int_\Omega |f|^p \frac{d\rho_X}{d\omega}d\omega\right)^{\frac1p}=\left(\int_\Omega |f|^p d\rho_X\right)^{\frac1p}\leq (C'')^{\frac1p}\left(\int_\Omega |f|^p d\omega\right)^{\frac1p}.
\end{aligned}
\end{equation*}
For any $f\in\LL_{\rho_X}^p(\Omega)$, we have $f\in\LL_{\omega}^p(\Omega)$, which gives
$$\limsup_{n\to\infty}\left\Vert f-\sum_{j=1}^n\langle \phi_j^*,f\rangle \phi_j\right\Vert_{\LL_{\rho_X}^p(\Omega)}\leq (C'')^{\frac1p}\limsup_{n\to\infty}\left\Vert f-\sum_{j=1}^n\langle \phi_j^*,f\rangle \phi_j\right\Vert_{\LL_{\omega}^p(\Omega)} = 0.$$
Therefore, $\{\phi_j\}_{j\geq1}$ is also a Schauder basis of $\LL_{\rho_X}^p(\Omega)$.
\end{proof}

\begin{lemma}\label{lemma:A.10}
Assume that $f_\rho=\sum_{j=1}^{\infty}\beta_j^\rho\phi_j\in\B^{s,r}_{p,p}$ with $r>\frac{1}{2sp'}$. Let $L_n=\lceil n^\theta\rceil$ with $\theta>0$, and denote the truncation of $f_\rho$ to $\BB_{L_n}$ by $f_\rho^{L_n}:=S_{L_n}(f_\rho)=\sum_{j=1}^{L_n}\beta_j^\rho\phi_j$. Suppose that \autoref{assume:1} holds, and set $\nu=1$. If \autoref{assume:2} also holds, take $\nu$ to be the exponent specified therein. For $1<p<\infty$, we have
\begin{equation*}
\ee(f_\rho^{L_n})-\ee(f_\rho)
\leq C_{11}n^{\theta\nu\left(-2sr+\frac{1}{p'}\right)},
\end{equation*}
where $C_{11}$ is a constant independent of $n$.
\end{lemma}
\begin{proof}
The proof of \autoref{lemma:A.6} shows that there exists a constant $C_{12}>0$, independent of $n$, such that
\begin{equation*}
\ee(S_{L_n}f_\rho)-\ee(f_\rho)\leq C_{12}\Vert f_\rho-S_{L_n}f_\rho\Vert_{\LL^p}^{\nu}.
\end{equation*}
Similarly, we estimate the approximation error as follows:
\begin{equation*}
\begin{aligned}
&\Vert f_\rho-S_{L_n}f_\rho\Vert_{\LL^p}\leq \sum_{j=L_n+1}^\infty|\beta_j^\rho|\Vert\phi_j \Vert_{\LL^p}\leq\left(\sum_{j=L_n+1}^\infty|\beta_j^\rho|^pj^{2srp}\right)^{1/p}\left(\sum_{j=L_n+1}^\infty j^{-2srp'}\right)^{1/p'}\\
\leq&\Vert f_\rho\Vert_{\BB^{s,r}_{p,p}}\left(\int_{L_n}^\infty x^{-2srp'}dx\right)^{1/p'}\leq\Vert f_\rho\Vert_{\BB^{s,r}_{p,p}}\left(\frac{1}{2srp'-1}\right)^{1/p'}L_n^{-2sr+\frac{1}{p'}}.
\end{aligned}
\end{equation*}
Combining the above two estimates completes the proof.
\end{proof}

\end{document}